\documentclass [12pt, makeidx]{amsart}
\usepackage{graphicx} 
\usepackage{colordvi}
\usepackage{amsthm}
\usepackage{amssymb}
\usepackage[all]{xy}
\usepackage{ulem}
\usepackage{xkeyval}
\usepackage{sseq}
\usepackage{float}
\usepackage{enumitem}
\usepackage{enumitem}
\usepackage{url}
\usepackage{comment}
\usepackage{varioref}
\usepackage{float}
\usepackage{epigraph}
\usepackage{acronym}
\usepackage{subfig}
\usepackage{fourier}
\usepackage{xfrac}
\usepackage{rotating}
 \usepackage{todonotes}
\usepackage[nolabel]{showlabels}
\newtheorem{thm}{Theorem}[section]
\newtheorem{cor}[thm]{Corollary}
\newtheorem{lem}[thm]{Lemma}

\newtheorem{prop}[thm]{Proposition}

\newtheorem{defn}[thm]{Definition}
\theoremstyle{definition}
\newtheorem{prop-defn}[thm]{Proposition and Definition}
\theoremstyle{remark}
\newtheorem{rem}[thm]{Remark}
\newtheorem{rems}[thm]{Remarks}
\newtheorem{example}[thm]{Example}
\newtheorem{examples}[thm]{Examples}

\numberwithin{equation}{section}

\newcommand{\thistheoremname}{}
\newtheorem*{genericprop*}{\thistheoremname}
\newenvironment{namedprop*}[1]
  {\renewcommand{\thistheoremname}{#1}%
   \begin{genericprop*}}
  {\end{genericprop*}}

\newtheorem*{genericlem*}{\thistheoremname}
\newenvironment{namedlem*}[1]
  {\renewcommand{\thistheoremname}{#1}%
   \begin{genericlem*}}
  {\end{genericlem*}}
  
  \newtheorem*{genericthm*}{\thistheoremname}
\newenvironment{namedthm*}[1]
  {\renewcommand{\thistheoremname}{#1}%
   \begin{genericthm*}}
  {\end{genericthm*}}

\newcommand{\Id}{{\rm Id}}

\newcommand{\Image}{{\rm Im}}
\newcommand{\Ker}{\rm Ker}

\usepackage{pdfsync}

\newcommand{\F}{{\mathcal F}}
\newcommand{\G}{{\mathcal G}}

\newcommand{\I}{{\mathcal I}}
\newcommand{\cI}{{\mathcal I}^{\bullet}}

\newcommand{\J}{{\mathcal J}}

\newcommand \id {{\rm id}}

\newcommand{\supp}{{\rm supp}}

\newcommand{\Iso}{{\rm Isom}}

\def \dispdot {}

\usepackage{scrtime}

\makeindex

 \usepackage[mathscr]{eucal}

\newcommand {\gr}{\mathrm {gr}}

\newcommand {\diam}{\mathrm {diam}}
\newcommand {\Isom}{\mathrm {Isom}}

\usepackage{setspace}
\usepackage{hyperref}
\DeclareMathAlphabet{\mathpzc}{OT1}{pzc}{m}{it}
\usepackage[percent]{overpic}
\usepackage{datetime}
\usepackage{lscape}
\usepackage{xr}
\usepackage{calligra}
\newcommand{\DHam }{\mathfrak{DHam}}
\newcommand{\HH }{\mathfrak{Ham}}
\newcommand {\LL}{\mathfrak L}
\makeindex

\title[Stochastic homogenization for variational solutions]{Stochastic homogenization for variational solutions of Hamilton-Jacobi equations}
\author{C. Viterbo }
\thanks{DMA, \'Ecole Normale Sup\'erieure, 45 Rue d'Ulm, 75230 Cedex 05, FRANCE. On leave from Department of Mathematics, Universit\'e de Paris-Sud, Orsay. We also acknowledge support from ANR MICROLOCAL (ANR-15-CE40-0007) and  NSF grant DMS- 1440140.}	
\begin{document}
\def \Z {\mathbb Z}
\def \H {\mathcal H}
\def \cF {\F^{\bullet}}
\def \cG {\G^{\bullet}}
\def \cI {\I^{\bullet}}
\def \cJ {\J^{\bullet}}
\def \Char {{\rm Char}}
\def \card {{\rm card}}
\def \cstar {\varoast}

\acrodef{GFQI}[GFQI]{Generating Function Quadratic at Infinity}

\maketitle
\today ,\;\; \currenttime

\tableofcontents
\section{Introduction}
Let $(\Omega, \mu)$ be a probability space endowed with an ergodic action $\tau$ of $( {\mathbb R} ^n, +)$. This means that if $X\subset \Omega$ satisfies $\tau_aX \subset X$ for all $a\in {\mathbb R} ^n$, then $\mu(X)=0\;\text{or}\; 1$. 
Let $H(x,p; \omega)=H_\omega(x,p)$ be a smooth Hamiltonian on $T^* {\mathbb R} ^n$ parametrized by  $\omega\in \Omega$ and such that  \begin{displaymath} H(a+x,p;\tau_a\omega)=H(x,p;\omega) \end{displaymath}  We shall specify later the assumptions satisfied by $H$. 
We now consider for an initial condition $f\in C^0 ( {\mathbb R}^n)$,  the family of stochastic Hamilton-Jacobi equations
\begin{displaymath}\tag{$HJS_ \varepsilon $}\left\{ \begin{aligned}  \frac{\partial u^{ \varepsilon }}{\partial t}(t,x;\omega)+H\left (\frac{x}{ \varepsilon } , \frac{\partial u^\varepsilon }{\partial x}(t,x;\omega);\omega \right )=0 &\\
u^\varepsilon (0,x;\omega)=f(x)&
\end{aligned} \right .\end{displaymath}

Fixing $\omega$, we can consider different type of generalized solutions (there is  generally  no   smooth solution) for this equation. The most interesting ones are  either the viscosity
solution of Crandall-Lions (see \cite{C-L} and also \cite{Ba, B-C}), or the variational solutions defined in \cite{Sik,Ch,Viterbo-Ott, Viterbo-Montreal}, both requiring some assumptions on $f$ and $H$ that will be specified later. 
The problem of stochastic homogenization for the above equation is to determine whether for $\mu$-a.e. in $\omega$,  the sequence $u^{ \varepsilon }(t,x;\omega)$ $C^0$-converges   on compact sets to $ \overline u (t,x)$, solution of 
\begin{displaymath}\tag{HJH}\left\{ \begin{aligned}  \frac{\partial v}{\partial t}(x)+{\overline H}\left (\frac{\partial v }{\partial x}(x) \right )=0 &\\
v (0,x)=f(x)&
\end{aligned} \right .\end{displaymath}
 where $\overline H$ is to be determined (and in general cannot be defined  explicitly). Note that $\overline H$ does {\bf not} depend on $\omega$ by the ergodicity hypothesis. 

A classical case is the so-called (non-stochastic) periodic case, corresponding to the case where  
$\Omega= {\mathbb T}^n$, $\tau_a$ is the translation on the torus, and 
$H(x,p;\omega)=K(x-\omega,p)$ where $K: T^*{\mathbb T}^n \longrightarrow {\mathbb R} $ is a Hamiltonian on the torus\footnote{Indeed, if $u^\varepsilon (t,x)$ is the solution (either viscosity or variational) of  $\frac{\partial u}{\partial t}(t,x)+K( \frac{x}{ \varepsilon }-\omega, \frac{\partial u^\varepsilon }{\partial x}(t,x))=0$ then $v^ \varepsilon (t,y)=u^\varepsilon (t,y+ \varepsilon \omega)=0$  satisfies $\frac{\partial u}{\partial t}(t,x)+K( \frac{y}{ \varepsilon }, \frac{\partial v^\varepsilon }{\partial y}(t,y))=0$. Thus $u^\varepsilon (t,y)=v^\varepsilon (t,y-\varepsilon \omega)$, and convergence of $v^\varepsilon $ to $\overline v$ as $ \varepsilon $ goes to $0$ is equivalent to convergence of $u^\varepsilon $ to $\overline u = \overline v$. See the proof of Corollary \ref{Cor-1.6} for another method of reducing to the periodic case.}, 
 For viscosity solutions, homogenization in the periodic non-stochastic case has been settled by \cite{L-P-V} in 1988, and for variational solutions by \cite{SHT} in 2008. 

For  the general stochastic case, this problem has been solved for  viscosity solutions by Rezakhanlou-Tarver and Souganidis in \cite{R-T, Soug}, assuming  $H$ is convex in $p$.  Beyond the quasi-convex case  and some very special cases (see for instance \cite{ATY}), nothing is known for viscosity solutions in the general (i.e. for $H$ non-convex in $p$) case, and counterexamples have been found, first by Zilliotto and then by Feldman-Souganidis (\cite{Zilliotto}, \cite{F-S}). 

We settle here the case of variational solutions without any convexity assumption. Note that the construction of a variational solution relies on the choice of a field of coefficients for the homology theory we use, but once the field is chosen, the variational solution is uniquely defined \footnote{See for example in \cite{Cardin-Viterbo} and more explicitly in \cite{WQ2} and Appendix B in \cite{Roos2}}. We shall here fix once and for all a coefficient field (the reader can think of $\mathbb Z/2\mathbb Z $ or $\mathbb R$ for example).  As in \cite{SHT}, our results hold when $H$ is either compact supported or coercive in the $p$ direction.  Note that fixing $\omega$,  if $V_t(H)f=u(t,x)$ is the variational solution operator\footnote{This means that it sends $f$ to the variational solution of 
\begin{displaymath}\tag{$HJS$}\left\{ \begin{aligned}  \frac{\partial u}{\partial t}(t,x)+H\left (x , \frac{\partial u }{\partial x}(t,x) \right )=0 &\\
u (0,x)=f(x)&
\end{aligned} \right .\end{displaymath}
Note that the operator is neither linear, not a semigroup (since variational solutions do not have the Markov property).
}
of the Hamilton-Jacobi equation, and $S_t(H)f$ the viscosity semigroup, we know that for $H$ convex in $p$ we have $S_t(H)=V_t(H)$. Our result thus implies the stochastic homogenization for viscosity solutions in the convex case\footnote{However in that case our method is much more complicated.} as in \cite{R-T} and \cite{Soug}.
In the general case it has been proved in \cite{WQ1, WQ2} (see also \cite{R}, theorem 1.19) that 
 \begin{displaymath}  S_t (H)= \lim_{n\to +\infty} (V_{t/n}(H))^n \end{displaymath} 
However, in the non-convex case,   stochastic homogenization of the viscosity solutions is not a consequence of stochastic homogenization for variational solutions, unless\footnote{Indeed, if we knew that $V_t( \varepsilon )= V_t(H_ \varepsilon )={\overline V}_t + R_t( \varepsilon )$ where $\Vert R_t( \varepsilon )\Vert \leq Ct \varepsilon $, and $\overline V_t=V_t(\overline H)$ is the homogenized operator, we would get that $ \Vert (V_{t/n}( \varepsilon ))^n- (\overline V_{t/n})^n\Vert \leq  Ct \varepsilon $   hence, setting $\overline S_t = \lim_n (\overline V_{t/n})^n$ we would have  $\Vert S^t( \varepsilon ) - \overline S_t \Vert \leq Ct \varepsilon $ hence $\lim_{ \varepsilon \to 0} S_0^t( \varepsilon )=\overline S_0^t$. Since there are counterexamples to stochastic homogenization for viscosity solutions (see \cite{Zilliotto, F-S}), the above inequality cannot hold in general. }  the above convergence is uniform
in $\varepsilon \in ]0,1]$, for $H_\epsilon (x,p)=H( \frac{x}{ \varepsilon },p)$. 

Of course, as in \cite{SHT}, the equation \thetag{$HJS_ \varepsilon $} is related to the Hamiltonian flow of $H( \frac{x}{ \varepsilon }, p; \omega)$ given by 
$$ \varphi^{t}_{ \varepsilon , \omega}=\rho_ \varepsilon^{-1}\varphi_\omega ^{ \frac{t}{ \varepsilon }}\rho_{ \varepsilon }$$  where $\varphi_\omega^t$ is the flow of $H(x,p;\omega)$ and $\rho_ \varepsilon (x,p)= ( \frac{x}{ \varepsilon },p)$.

We shall prove analogously to \cite{SHT}  that for almost all $\omega$, we have $$\varphi^{t}_{ \varepsilon , \omega} \overset{\gamma_c}\longrightarrow \overline{ \varphi}^{t}_{\omega}$$
but since we are on a non-compact base we have to redefine the $\gamma$-distance, that we shall denote by $\gamma_c$. 
 
Our main result is
\begin{thm} [Main theorem]\label{Thm-main}
Let $H(x,p;\omega)$ be a stochastic Hamiltonian on $T^*{\mathbb R}^n\times \Omega$, where $(\Omega,\mu)$ is a probability space endowed with an action $\tau$ of $ {\mathbb R}^n$.  We assume the following conditions are satisfied : 
 \begin{enumerate} \label{C}
 \item\label{C1} For all $a \in {\mathbb R}^n$, the map $\tau_a$ is measure preserving
and the action $\tau$ is ergodic for the measure $\mu$ (i.e. invariant sets have measure $0$ or $1$)
  \item\label{C2}  We have for all $a \in {\mathbb R}^n, (x,p) \in T^*{\mathbb R}^n$ and almost all $\omega \in \Omega$ the identity $H(x+a, p, \tau_a\omega)=H(x,p, \omega)$
 \item\label{C3}  The map $(x,p) \mapsto H(x,p, \omega)$ is $C^{1,1}$ for $\mu$-almost all $\omega$ 
  \item \label{C4} For almost all $\omega$,  $H$ is compact supported in the $p$ direction i.e. the set $\left \{p \mid \exists x\in {\mathbb R}^n, \;  H(x,p;\omega)\neq 0\right \}$ is bounded.
  \item \label{C5}  There exists $C$ such that for almost all $\omega$ and for all $(x,p)\in T^*{\mathbb R}^n$ we have $ \left \vert \frac{\partial H}{\partial p}(x,p;\omega) \right\vert  \leq C$. 
  \item \label{C6}  There exists $C$ such that for almost all $\omega$ we have $\sup_{(x,p)\in T^*{\mathbb R}^n} \vert H(x,p;\omega) \vert \leq C$.  \end{enumerate} 
   
   Then if $\varphi_{\varepsilon ,\omega}$ is the flow of $H_{ \varepsilon ,\omega}(x,p)=H( \frac{x}{ \varepsilon },p)$  there is a function $\overline H$ in $C^0( {\mathbb R}^n, {\mathbb R} )$ such that 
\begin{displaymath} \varphi^{t}_{ \varepsilon , \omega} \overset{\gamma_c}\longrightarrow \overline{ \varphi}^{t}_{\omega} 
\end{displaymath} 
  for the topology $\gamma_c$ that will be defined in section \ref{Section-spectral}. Here $\overline \varphi^t$ denotes the flow of $\overline H$ in $\widehat{\DHam} (T^* {\mathbb R}^n)$ the $\gamma_c$-completion of $\DHam(T^* {\mathbb R}^n)$.
As a consequence  if  $f$ is uniformly continuous on $ {\mathbb R}^n$,     then a.s. in $\omega \in \Omega$ the variational solution $u^\varepsilon (t,x;\omega)$ of \thetag{$HJS_ \varepsilon $}
converges to the variational solution $\overline u (t,x)$ of  \thetag{$HJH$}. 
 
\end{thm} 

\begin{rem} 
\begin{enumerate} 
\item Existence and uniqueness of the variational solution for \thetag{$HJS_ \varepsilon $} follows from \cite{Cardin-Viterbo},  pp. 266-276 (since we are in the case of a non-compact base). The bounded propagation speed condition in \cite{Cardin-Viterbo} are more general than the ones in the present paper and are obviously satisfied in the fiberwise compact supported case. 
\item Set $\Omega_c=\{\omega \in \Omega \mid \sup_{(x,p)\in T^*{\mathbb R}^n} \vert H(x,p;\omega) \vert \geq c\}$ is $\tau$ invariant. If it has measure $0$, then \ref{C6}) holds, otherwise $\sup_{(x,p)\in T^*{\mathbb R}^n} \vert H(x,p;\omega) \vert =+\infty$ for a.e. $\omega$. 
\item The set $\Omega '_R=\{\omega \in \Omega \mid \supp (H)\subset {\mathbb R}^n\times B(R)\}$ is also invariant by $\tau$. It thus either has measure $1$ for some $R$, and then the bound in (\ref{C4}) is independent from $\omega$ in  set of full measure, or it has measure $0$ for all $R$ and then for a.e. $\omega$, condition  (\ref{C4}) is not satisfied. In the first case, we shall say that the  $H_\omega$ have {\it\bf uniform fiber compact support}
\end{enumerate} 
\end{rem}

The compact supported case is usually not the most interesting in applications. However the above theorem implies
\begin{cor} [Main corollary]\label{Cor-main}
Let $H(x,p;\omega)$ be a stochastic Hamiltonian on $T^*{\mathbb R}^n\times \Omega$, where $(\Omega,\mu)$ is a probability space endowed with an action $\tau$ of $ {\mathbb R}^n$.  We assume the following conditions are satisfied : 
 \begin{enumerate}[label=({\arabic*a})] \label{C'}
 \item \label{C1a} Conditions (\ref{C1})-(\ref{C3}) as in the Main Theorem
  \item\label{C2a}  For all $(x,p;\omega)$ we have $ \vert \frac{\partial H}{\partial p}(x,p;\omega) \vert  \leq h^1( \vert p \vert )$  for almost all $\omega$ for some continuous function $h^1: {\mathbb R} \longrightarrow {\mathbb R} $. 
  \item \label{C3a} for almost all $\omega$  $H$ is coercive, that is $\lim_{ \vert p \vert \to +\infty} \vert H(x,p;\omega)\vert =+\infty$ uniformly in $x$ 
  \end{enumerate}

If $H$ satisfies the above assumptions and $f$ is uniformly continuous on $ {\mathbb R}^n$,   there is a coercive function $\overline H$ in $C^0( {\mathbb R}^n, {\mathbb R} )$   such that a.e. in $\omega$,  the variational solution $u^\varepsilon (t,x;\omega)$ of 
\begin{displaymath}\tag{$HJS_ \varepsilon $}  \left\{ \begin{aligned}  \frac{\partial u^{ \varepsilon }}{\partial t}(t,x;\omega)+H\left (\frac{x}{ \varepsilon } , \frac{\partial u^\varepsilon }{\partial x}(t,x;\omega);\omega \right )=0 &\\
u^\varepsilon (0,x;\omega)=f(x)&
\end{aligned} \right .\end{displaymath} 
converges to the variational solution $\overline u (t,x)$ of  \begin{displaymath} \tag{HJH} \left\{ \begin{aligned}  \frac{\partial v}{\partial t}(t,x)+{\overline H}\left (\frac{\partial v }{\partial x}(t,x) \right )=0 &\\
v (0,x)=f(x)&
\end{aligned} \right .\end{displaymath}  
\end{cor} 

\begin{examples} 
\begin{enumerate} 
\item Let $\Omega$ be the space of $C^1$ functions on $ {\mathbb R} ^n$, $(\tau_a f) (x)=f(x+a)$ and $\mu$ some measure on $\Omega$ invariant by $\tau_a$ and ergodic. Let $V$ be a bounded function. Set $H(x,p;\omega)=  \frac{1}{2} h(p) -V(\omega (x))$, where $h$ is coercive. This satisfies the assumptions of the Corollary and  corresponds to a random potential, with probability $\mu$. 
\item (\cite{P-R}, example 2.4(ii)) Let $H_0(q,p)$ be a Hamiltonian and $H(q,p;\omega) =\sum_{j\in {\mathbb Z}} H_0(q-q_j,p)$ where $\omega=(q_j)_{j\in {\mathbb Z} }$ is a stationary point process, that is a probability on $ {\mathbb R} ^{\mathbb Z}$ invariant by translation. This makes sense provided $H^0$ decreases fast enough as $q$ goes to infinity. Then $H$ satisfies the assumption of the above Corollary. 
\end{enumerate} 
\end{examples} 
\begin{rem} \label{Rem-1.5}
Here are a few comments
\begin{enumerate} 
\item We could of course also state a convergence result in the coercive case for the sequence $\varphi_ {\varepsilon , \omega}$, it is just that the statement of convergence would be a little more complicated to state
\item Note that ergodicity implies that \begin{displaymath} \Omega(m,p)=\left\{ \omega \mid \sup_{x \in {\mathbb R}^n} H(x,p;\omega) \leq m \right \} \end{displaymath}  is an invariant set for the $\tau$-action, hence has 
 measure either $0$ or $1$. As a result by ergodicity property \ref{C3a} or (\ref{C4}) hold either for a set of $\omega$ of zero measure or for a.e. $\omega$. We set $h_+(p)$ to be the smallest $m$ such that $\Omega (m,p)$ has measure $1$, where  $h_+(p)\in {\mathbb R} \cup \{+\infty\}$.  Thus we  have  $\sup_{x \in {\mathbb R}^n}H(x,p;\omega)=h_+(p)$ a.e. in $\Omega$ and similarly   $\inf_{x \in {\mathbb R}^n}H(x,p;\omega)=h_-(p)$  a.e. in $\Omega$, (where $h_-(p)\in \{-\infty\} \cup {\mathbb R} $).   Notice that assumption \ref{C3a} implies that $h_\pm(p)$ is finite, and that $\lim_{ \vert p \vert \to +\infty} h_\pm (p)=+\infty$. This condition is more or less explicit in both \cite{R-T} (see their conditions (Aii)-(Aiii))
and \cite{Soug} (see his Condition 0.2). Similarly 
 \begin{displaymath} h_+^1(p, \omega)=\sup_{x\in {\mathbb R}^n}  \left \vert \frac{\partial H}{\partial p}(x,p;\omega ) \right \vert  \end{displaymath} 
 is invariant by $\tau$ hence  independent from $\omega$ a.e. in $\Omega$, and equal to $h_+^1(p)$, so \ref{C2a} and (\ref{C5}) either hold a.s. or do not hold a.s in $\Omega$. 
   
 \item Setting $h_+(r)=\sup\{ h_+(p) \mid \vert p \vert \leq r\}$ and  $h_-(r)=\inf\{ h_-(p) \mid \vert p \vert \geq r\}$ we have that a.s. in $\omega$, the set $\{(q,p) \mid H_\omega(q,p)\leq c \}$ contains $ {\mathbb R}^n \times B_r$ for $c \geq h_+(r)$ and is contained in  $ {\mathbb R}^n \times B_r$ for $c \leq h_-(r)$.  In particular in case \ref{C3a} the coercivity is automatically  uniform in $x$ (i.e. we can bound $\{ p \mid H(x,p;\omega) \leq c\}$ independently from $x$). . 
\item We shall reduce the case \ref{C3a} where $H$ is coercive to the  uniformly fiberwise compact supported case by replacing $H$ by $\chi_R(H)$ that is compact supported where $\chi_R : {\mathbb R} \to {\mathbb R} $ is a function supported in $]-\infty , R+1]$ such that $\chi'(t)=1$ for $t\leq R$ (see \cite{Cardin-Viterbo}, Appendix B). Then $H_{\chi_R}=\chi_R(H)$,  also satisfies  $H_{\chi_R} (x+a,p;\tau_a \omega)=H_{\chi_R} (x,p;\omega)$. 
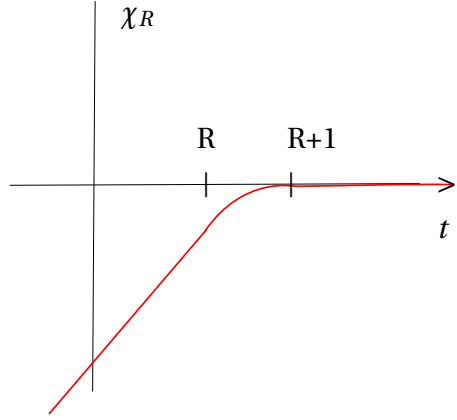
\begin{figure}[ht]
\center\begin{overpic}[width=6cm]{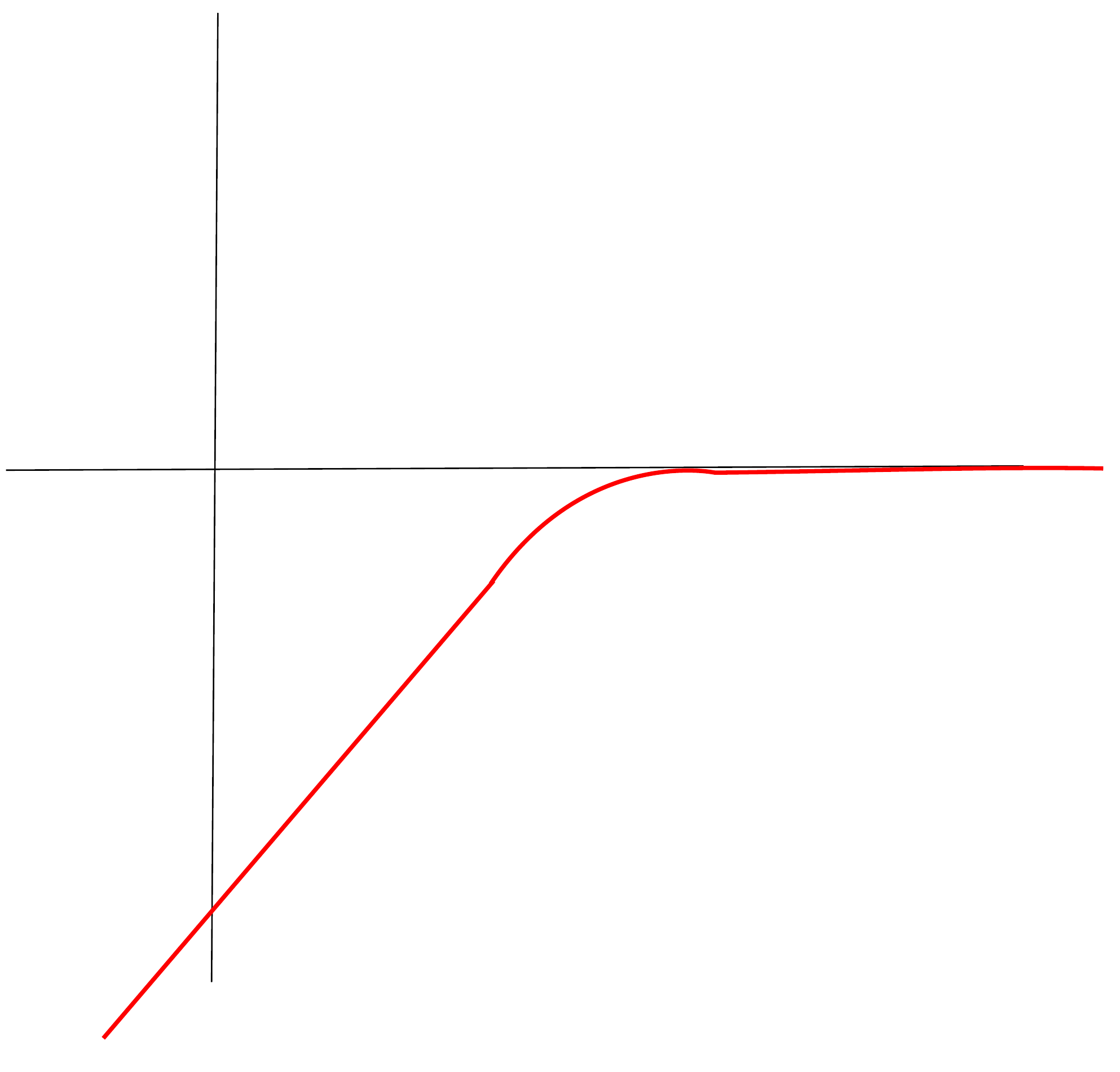}
 \put (25,90) {$\chi_R$} 
  \put (42,62){R}  \put (42,52){$\bf \mid $}
    \put (62,62){R+1}  \put (62,52){$\bf \mid $}
  \put (95,42) {$t$}  \put (95,52) {$>$}   
\end{overpic}
\caption{Graph of $\chi_R$
}\label{fig-1}
\end{figure}
\item Let us consider a Hamiltonian $H$ convex in $p$ and uniformly coercive (i.e. there exists a function $h_\pm(p)$ going to infinity, such that $h_-(p)\leq H(x,p;\omega)\leq h_+(p)$) as follows form the assumptions in both \cite{R-T} (2.5) (ii) and (2.8) p. 280 and \cite{Soug} (see Condition 0.2). Then we claim that its truncation $H_{\chi_R}=\chi_R(H)$ satisfies assumption (\ref{C5}) of the Main Theorem (condition (\ref{C4}) is obvious). This is because $ \frac{\partial H_{\chi_R}}{\partial p}=\chi_R'(H)  \frac{\partial H}{\partial p}$, so it is enough to prove that $ \frac{\partial H}{\partial p}$ is bounded on a set $ \vert p \vert \leq C$. But if $ \vert  \frac{\partial H}{\partial p}(x_0,p_0)\vert \geq A$, we can find $p_1$ with $\vert p_1 \vert \leq 2C$ such that  $p_0-p_1$ is collinear with  $ \frac{\partial H}{\partial p}(x_0,p_0)$ and $ \vert p_0-p_1 \vert =C$, so that  
\begin{gather*} 
 \sup_{ \vert p \vert \leq 2C}h_+(p)- \inf_{ \vert p \vert \leq 2C}  h_-(p)\geq H(x,p_1)-H(x,p_0) \geq\\  \left\langle \frac{\partial H}{\partial p}, p_0-p_1\right\rangle \geq C  \left \vert\frac{\partial H}{\partial p} \right\vert = CA 
\end{gather*}  
hence $A$ is bounded. 
\end{enumerate}
 \end{rem}
 
Our result can be easily extended, since  we do not need the full action of $  {\mathbb R}^n$. For example if we have an action of $ \mathbb Z^n$ we get the following
 
\begin{cor}\label{Cor-1.6}
With the same assumptions as in the Main Theorem except that we have an action of $ {\mathbb Z} ^n$ (instead of $ {\mathbb R}^n$) on $\Omega$, still denoted by $\tau$, 
and the first two assumptions are replaced by 
  \begin{enumerate}[label=({\arabic*b})] \label{D}
 \item For all $z \in {\mathbb Z}^n$, the map $\tau_z$ is measure preserving and ergodic
 \item   We have for all $z \in {\mathbb Z}^n, (x,p) \in T^*{\mathbb R}^n$ and almost all $\omega \in \Omega$ the identity 
 \begin{displaymath} H(x+z, p, \tau_z\omega)=H(x,p, \omega) \end{displaymath} 
 \end{enumerate} 
while conditions (\ref{C3})-(\ref{C6}) are unchanged.  We then  have the same conclusion as in the Main Theorem. 
\end{cor} 

Finally, note that ergodicity of $\tau$ on $\Omega $ is not required, since we can use the ergodic decomposition theorem (cf. \cite{Greschonig-Schmidt}), which holds for Borel spaces\footnote{That is isomorphic (as a measured space) to a complete separable metric space with a measure defined on its Borel algebra.} and obtain the
 \begin{cor} 
 With the same assumptions as in the main theorem (resp. Corollary \ref{Cor-1.6}) except that the action $\tau$ is not supposed to be ergodic but we assume $(\Omega, \mu)$ is a Borel space, we have the same conclusion,  except that $\overline H(p;\omega)$ now depends on $\omega\in \Omega$ and is constant on each ergodic component of $\tau$. 
 \end{cor} 
Our proof will require the following steps, starting from the uniformly fiber compact supported case :
\begin{enumerate} 
\item On  ${\HH}_{fc}(T^* {\mathbb R}^n)$  the set of uniformly fiberwise compact supported Hamiltonians on $T^* {\mathbb R} ^n$ we define a metric $\gamma_c$ (see Section \ref{Section-4} and \ref{Section-5}).
\item  We identify $\Omega$ to $\mathfrak{H}_\Omega$ the set of $H_\omega$ for $\omega\in \Omega$, and $\widehat{\mathfrak {H}}_\Omega$ its completion for $\gamma_c$. We then prove that ergodicity implies compactness of the metric space $(\widehat{\mathfrak H}_\Omega, \gamma_c)$ (see Section \ref{Section-6} and \ref{Section-7}).
 The action of $ {\mathbb R} ^n$ on ${\mathfrak H}_\Omega$ given by $(\tau_aH)(x,p;\omega)=H(x-a,p;\omega)=H(x,p;\tau_a\omega)$ extends to an action of a compact connected metric abelian group ${\mathbb A}_\Omega$ on  
$(\widehat{\mathfrak {H}}_\Omega, \gamma_c)$, and $ {\mathbb R}^n$,  through the action $\tau$, is identified to a dense subgroup of ${\mathbb A}_\Omega$. Moreover we prove that for $\mu$-almost all $H$ in $\mathfrak{H}_\Omega$, the ${\mathbb A}_\Omega$ orbit of $H$ is equal to $\widehat{\mathfrak {H}}_\Omega$. 
\item In Section \ref{Section-8} we prove a regularization theorem showing that the  action of ${\mathbb A}_\Omega$ on $\widehat{\mathfrak {H}}_\Omega$ can be approximated by an action of a finite-dimensional torus (note that ${\mathbb A}_\Omega$ is not in general  a finite dimensional torus,  but is a projective limit of finite dimensional tori). 
\item \label{item3} We prove in Section \ref{Section-9} that  homogenization holds when  ${\mathbb A}_\Omega$ is a finite dimensional torus (quasi-periodic case)  and $\omega \mapsto H_\omega$ is continuous for the $C^0$-topology instead of the $\gamma_c$-topology. 
\item In Section \ref{Section-10} we conclude the proof in the fiberwise compact case, and in Section \ref{Section-11} for the coercive case and in Section \ref{Section-12} for the discrete case. 
\end{enumerate}

\section{Notations and abbreviations}
\begin{itemize} 
\item $\Omega$  a probability space with measure $\mu$
\item a.s. or a.e. : almost surely or almost everywhere in $(\Omega, \mu)$
\item GFQI : \acl{GFQI}
\item $H^*$, $H_*$ cohomology and homology with coefficients in some field $ \mathbb K$. 
\item $\mu_N$ the fundamental class in $H^d(N)$ (for a closed manifold) or $H^d(N,\partial N)$ (for a manifold with boundary) or $H_c^d(N)$ (for a non-compact manifold) where $d=\dim(N)$
\item $1_N$ the generator of $H^0(N)$
\item $T^*N$ the cotangent bundle of $N$ with the standard symplectic form $\omega=d\lambda$, where $\lambda =pdq$
\item $\overline{T^*N}$  the cotangent bundle of $N$ with the opposite of the standard symplectic form $\omega=-d\lambda$, where $\lambda =pdq$
\item $C^0_{fc}([0,1]\times T^* {\mathbb R} ^n)$ set of continuous functions on $[0,1]\times T^* {\mathbb R} ^n$ (viewed as ``continuous Hamiltonians'' which are fiberwise compact
\item $\HH_{fc}(T^*N)$ the set of uniformly fiberwise compact supported\footnote{i.e. the support is contained in $ {\mathbb R}^n\times B(R)$ for some $R$} autonomous Hamiltonians
\item $\HH_{fc}([0,1]\times T^* {\mathbb R}^n)$ the set of uniformly fiberwise compact supported time-dependent Hamiltonians
\item For a Hamiltonian $H$ on $T^*N$, $\varphi_H^t $ is the solution of $ \frac{d}{dt} \varphi_H^t(z)=X_H(t,\varphi_H^t(z)$ such that $\varphi_H^0(z)=z$. We set $\varphi_H=\varphi_H^1$
\item ${\DHam}_{fc}(T^* N)$ is the image by $H \mapsto \varphi_H$ of $\HH_{fc}([0,1]\times T^* {\mathbb R}^n)$
\item FPS :  Finite Propagation Speed (see Definition \ref{Def-FPS})
\item BPS :  Bounded Propagation Speed (see Definition \ref{Def-BPS})
\item  $\DHam_{FP}(T^*N)$ (resp. $\HH_{FP} (T^*N)$, $\HH_{FP} ([0,1]\times T^*N)$ ) elements in $\DHam(T^*N)$ (resp. $\HH(T^*N)$, $\HH ([0,1]\times T^*N)$) having FPS.
\item  $\DHam_{BP}(T^*N)$ (resp. $\HH_{BP} (T^*N)$, $\HH_{BP} ([0,1]\times T^*N)$ ) elements in $\DHam(T^*N)$ (resp. $\HH(T^*N)$, $\HH ([0,1]\times T^*N)$) having BPS.
\item $\LL_{} (T^*N)$ the image of $\DHam_{FP}(T^*N)$ by $\varphi \mapsto \varphi(0_N)$
\item $\gamma_c$ the uniform  topology on  $\LL_{} (T^*N)$ (see Definition \ref{Def-4.11})
\item $\widehat{\LL}_{} (T^*N)$ the completion for $\gamma_c$ of $\LL_{} (T^*N)$ (see Definition \ref{Def-4.11})
\item $\widehat {\DHam}_{FP}(T^*N)$ (resp. $\widehat {\DHam}_{BP}(T^*N), \widehat {\DHam}_{fc}(T^*N)$ the completion for $\gamma_c$ of 
${\DHam}_{FP}(T^*N)$ (see Definition \ref{Def-5.21})
\item $G_f$ the graph of $df$ in $T^*N$
\item $\overline L$:  for $L\in \LL (T^*N)$ we define $\overline L=\left \{(x,-p) \mid (x,p)\in L\right \}$ where $f_{\overline L}=-f_L$. 
\end{itemize} 
\section{Acknowledgments and general remarks.}
I  would like to acknowledge the hospitality of M.S.R.I. during the semester "Hamiltonian systems, from topology to applications through analysis" in the fall\footnote{supported by NSF grant DMS- 1440140} of 2018. The idea of this work came through attending talks and several discussions with Fraydoun Rezakhanlou, that I want to thank wholeheartedly here for generously sharing his ideas. 

As the reader will check, and analogously to \cite{SHT}, the methods used here are  drawn from symplectic topology. This paper can be considered as part of a program to study symplectic topology in a random framework (or random phenomena having a symplectic structure) of which a foundational  example is the random version of Poincar\'e-Birkhoff theorem from \cite{P-R}.  

\section{Non-compact supported Hamiltonians}\label{Section-4}
Let $N$ be a non-compact manifold. We shall assume that $N$ is homeomorphic to the interior of a manifold with smooth boundary\footnote{We eventually only use the case $N= {\mathbb R}^n$. For this section we actually only need that there is an exhausting sequence of open bounded sets $(U_j)_{j\in \mathbb N}$ such that  $U_j\subset U_{j+1}$ and for $j$ large enough, $U_j$ is ambient isotopic to $U_{j+1}$.}
\begin{defn}\label{Def-FPS} Let $\varphi\in \DHam (T^*N)$. We say that $\varphi$ has {\bf finite propagation speed}\index{Finite propagation speed}\index{Propagation speed! finite} (FPS. for short), if for each bounded set $U$, there is a bounded set $V$ such that $\varphi(T^*U)\subset T^*V$. 
A subset in $\DHam(T^*N)$  has {\bf uniformly finite propagation speed}\index{Uniformly finite propagation speed}\index{Propagation speed! uniformly finite} if each element has finite propagation speed, and moreover given $U$, the set $V$ can be chosen to be the same for all the elements in the subset. We write $\DHam_{FP}(T^*N)$ for the set of Hamiltonians maps with finite propagation speed. By abuse of language, we use the same terminology in $\HH (T^*N)$ : $H$ has {\bf finite propagation speed} if $\varphi_H$ has finite propagation speed, etc. We use the notation $\HH_{FP}(T^*N)$ for this set. 
\end{defn} 
Note that for instance if $ \left\vert  \frac{\partial H}{\partial p}(t,q,p) \right\vert  \leq C_U$ for all $(q,p)\in T^*U$ then $H$ has FPS.

The following lemma will prove useful
\begin{lem} \label{Lemma-4.2} Let $U\subset V $ be relatively compact open sets in $N$ such that for any compact set $K$ in $N$ there exists an isotopy of $N$ sending $K$ in $V$. 
Let $\varphi\in \DHam(T^*N)$ be such that $\varphi(T^*U) \subset T^*V$. Then we can find a Hamiltonian isotopy $(\varphi^t)_{t\in [0,1]}$ from the identity to $\varphi$ such that for all $t\in [0,1]$ we have
$\varphi^t(T^*U) \subset T^*V$. 
\end{lem} 
\begin{proof} Let $\psi^t$ be an isotopy from $\id$ to $\psi^1=\varphi$.
Let $X$ be a vector field corresponding to the isotopy for a compact set containing the projection of $\bigcup_{t\in [0,1]} \psi^t(U) =K$ and  pointing inwards on $\partial V$. Let $\rho^t$ be the Hamiltonian vector field of $H(t,x,p)=\langle p, X(t,x)\rangle$ which projects on the flow of $X$.  Possibly replacing $\rho^t$ by a $ \rho^{\alpha(t)}$, we may assume that for all $t\in [0,1]$ we have $\rho^t\circ \psi^t(T^*U)\subset T^*V$. Then $\rho^1\psi^1(T^*U) \subset T^*V$ and since $\psi^1(T^*U)\subset T^*V$ the set of $t$ such that $\rho^t\psi^1(T^*U)\subset T^*V$ is an interval, it must contain $[0,1]$ hence concatenating the Hamiltonian isotopy $t \mapsto \rho^t\psi^t$ with $t\mapsto \rho^{1-t}\psi^1$, we get a new Hamiltonian isotopy 
 that we denote $\varphi^t$ such that $\varphi^t(T^*U)\subset T^*V$ for all $t\in [0,1]$. 
\end{proof} 
 Note that our hypothesis on $N$ implies that we can find an exhausting sequence $(U_j)_{j\geq 1}$ of $N$ satisfying the assumptions of  Lemma \ref{Lemma-4.2}.

We shall now prove  that $\DHam_{fc}$ the set of Hamiltonians which are  uniformly fiberwise compact supported is contained in $ \HH_{FP}$. 
 \begin{prop}  \label{Prop-4.3}
If $H\in \HH_{fc}(T^* {\mathbb R}^n)$ is uniformly fiberwise compact supported, then $H$ has FPS.
 \end{prop} 
 \begin{proof} Indeed, if for some $C$,   $\varphi$ is the identity outside of $DT_C^*(N)= \{(q,p) \mid \vert p \vert \leq C\}$, then 
 $\varphi(T^*U)\subset T^*U \cup \varphi (T^*U\cap T_C^*N)$, but since $T^*U\cap T_C^*N$ is compact its image is contained in some $T^*V$ for $V$ bounded, and we get $\varphi(T^*U)\subset T^*(U\cup V)$. 
 \end{proof} 
 
 The usefulness of this notion will be clear on several occasions. 
 Remember that a generating function quadratic at infinity for $(L,f_L)$ where $L$ is a smooth Lagrangian, and $f_L$ a function such that  of $df_L=\Lambda_L $,
 is a smooth function $S: E=N \times F \longrightarrow {\mathbb R} $ where $E$ is a finite dimensional vector space\footnote{All this discussion also works if we replace $N\times F$ by a general finite -dimensional vector bundle. Then we must replace in the sequel the K\"unneth isomorphism by the Thom isomorphism.}, such that 
\begin{enumerate} 
\item $S$ coincides with a non-degenerate quadratic form $Q$ on the vector space $F$ for $\xi$ large enough
\item  $(x,\xi) \mapsto \frac{\partial S}{\partial \xi}(x,\xi)$ is transverse to $0$ 
\item  setting $\Sigma_S= \{ (x,\xi) \mid \frac{\partial S}{\partial \xi}(x,\xi)\}$ the image of this submanifold by 
$i_S: (x,\xi) \mapsto \frac{\partial S}{\partial x}(x,\xi)$ has image $L$
\item $f_L\circ i_S=S$
\end{enumerate} 

Let  $S_1, S_2$ be two \ac{GFQI} . They are said to be equivalent if they are fiberwise diffeomorphic  after stabilization, that is there are two non-degenerate quadratic forms $q_1, q_2$ such that if $$\tilde S_j(x,\xi_j,\eta_j)=S_j(x, \xi_j)+q_j(\eta_j)$$  there is fiber preserving diffeomorphism $$(x,\xi_1,\eta_1) \longrightarrow (x, \xi_2(x,\xi_1,\eta_1), \eta_2(x,\xi_1,\eta_1))$$
such that 
$$S_2(x, \xi_2(x,\xi_1,\eta_1), \eta_2(x,\xi_1,\eta_1))= S_1(x,\xi_1,\eta_1)$$

We shall say that $S_1, S_2$ are equivalent over $U$ if the fiber-preserving diffeomorphism is defined for $x\in U$. 
Note that the customary ``addition of a constant'' for the equivalence of generating functions is not needed here, since generating functions are normalized so that $S_{\mid L}=f_L$. 

We cannot expect a non-compact Lagrangian to have a \ac{GFQI} in this sense, since the number of variables required could go to infinity. We can either assume $F$ is a Hilbert space, but then positive and negative eigenspaces will generally be infinite dimensional so that $H^*(S^b, S^a)=0$ which is a notorious drawback\footnote{That we could avoid by using Floer homology everywhere, but would make reading this paper even harder for the Hamilton-Jacobi community !}. 
Here we have 
 
 \begin{defn} 
 We say that $L$  has a \ac{GFQI} if for each bounded set, $U$,  there is a \ac{GFQI} defined over $U\times F$(where $F$ depends on $U$), $S_U$, and there is a set $V \supset U$ such that the $S_W$ are all equivalent over $U$ for $W\supset V$. 
 Two \ac{GFQI} are equivalent if they are equivalent over each bounded set. 
 \end{defn} 
 \begin{thm}\label{Thm-4.5} Let $\varphi$ be an element in $\DHam_{FP}(T^*N)$. Then $\varphi(0_N)$ has a  \ac{GFQI} . Moreover such a \ac{GFQI} is unique up to equivalence. 
 \end{thm} 
 \begin{proof} 
See Appendix \nameref{Section-Appendix2}. 
 \end{proof} 
 \begin{rems}\label{rem-4.6}
 \begin{enumerate} Notice that   \item
 If $\varphi$ does not have FPS, $\varphi(0_N)$ does not even need to have surjective projection on $N$ : for example take on $T^* {\mathbb R} $ the Hamiltonian $ \frac{\pi}{4}(x^2+p^2)$, then $\varphi(0_ {\mathbb R})= \{0\}\times {\mathbb R}$ ! 
\item Using Lemma \ref{Lemma-4.2} we may assume we have a sequence $U_\nu$ of domains such that for all $t\in [0,1]$ we have $\varphi^t(T^*U_\nu)\subset T^*U_{\nu+1}$. We denote by $S_\nu=S_{U_\nu}$ and notice that we may assume that the restriction of $S_\mu$ over $U_\nu$ is exactly $S_\nu\oplus q_{\nu, \mu}$  by composing $S_\mu$ with an extension of the fiber preserving diffeomorphism realizing the equivalence\footnote{The existence of the extension follows from the fact that we may assume that for $\mu, \nu$ large enough, the inclusion $U_\nu \subset U_\mu$ is a homotopy equivalence.}. We shall always make this assumption in the sequel.
\item \label{rem-4.6-3}We will use the expression ``$S$ is  a \ac{GFQI} for $L$'' meaning there is a sequence $(S_\nu)_{\nu\geq1} $ of \ac{GFQI} for $L$ over $U_\nu$ to avoid cumbersome indexes. Most of the time this means we consider $S_\nu$ for $\nu$ large enough.
 \end{enumerate} 
 \end{rems} 
 \begin{defn} 
 We denote by $\mathfrak L (T^*N)$ the set of Lagrangians  of the type $\varphi(0_N)$ where $\varphi \in  \DHam_{FP}(T^*N)$. 
 
 \end{defn} 
  
 On a Riemannian manifold, there is a more precise notion than FPS.
 \begin{defn} \label{Def-BPS}
 Let $N$ be a manifold with a distance $d$ and  $\varphi\in \DHam (T^*N)$. We say that $\varphi$ has {\bf bounded propagation speed}\index{Bounded propagation speed}\index{Propagation speed! bounded} (BPS for short), if there is a constant $r_0$ such that for any ball $B(x_0,r)$ we have  $\varphi(T^*B(x_0, r))\subset T^*B(x_0, r+r_0)$. 
A subset in $\DHam(T^*N)$  has {\bf uniformly bounded propagation speed}\index{Uniformly bounded propagation speed}\index{Propagation speed! uniformly bounded} if each has bounded propagation speed, and moreover the constant $r_0$ can be chosen to be the same for all the elements in the subset. We write $\DHam_{BP}(T^*N)$ for the set of Hamiltonians maps with bounded propagation speed. By abuse of language, we use the same terminology in $\HH (T^*N)$ : $H$ has {\bf bounded propagation speed} if $\varphi_H$ has bounded propagation speed. 
\end{defn} 

 \begin{example} 
If $ \left\vert  \frac{\partial H}{\partial p}(t,q,p) \right\vert  \leq C$ for all $(q,p)\in T^*{\mathbb R}^n$ then $H$ has BPS. In particular Assumption (\ref{C5}) implies B.PS. . 
 \end{example} 
  \begin{rem} 
  \begin{enumerate} 
  \item Of course Bounded Propagation Speed implies Finite Propagation Speed. 
  \item Our definition of finite propagation speed does not exactly coincide with the terminology of \cite{Cardin-Viterbo} Definition B.5 p. 271. Our definition was more involved and the notion of finite propagation speed defined there is weaker than the present one,  but would still  be sufficient to prove our theorems. However this would have made an already long paper even longer.  
  \end{enumerate} 
  \end{rem} 
\section{Spectral invariants in cotangent bundles of non-compact manifolds}\label{Section-spectral}\label{Section-5}
The goal of this section is to define and state the main properties of the metric $\gamma$ that occurs in the statement of the Main Theorem. 
This has been done in \cite{Viterbo-STAGGF} in the case of a compact base, the present situation, for a non-compact base, is unfortunately slightly more involved. 

\subsection{The case of Lagrangians}
Let $L$ be an exact Lagrangian in $T^*N$ with $N$ not necessarily compact (but assumed, for simplicity, to be connected). We assume  a primitive of $\lambda _{\mid L}$, $f_L$ is given\footnote{Even though we write $L$, we always mean the pair $(L,f_L)$.}. 
 
We shall assume that  $L$ has a unique \ac{GFQI} , $S$ such\footnote{Remember cf Remark \ref{rem-4.6}(\ref{rem-4.6-3}) that this means there is a sequence $S_\nu$ of \ac{GFQI} over $U_\nu$ such that for $\nu\leq \mu$, the function $S_\mu$ restricts to the stabilization of $S_\nu$ over $U_\nu$.} that $f_L=S$ on $L$ (through the identification $i_S(x,\xi)=(x, \frac{\partial S}{\partial \xi}(x,\xi))$). For example according to Theorem \ref{Thm-4.5}, this is the case if $L=\varphi_H(0_N)$ with $\varphi\in \DHam_{FP} (T^*N)$. 

We denote by $T_F$ the generator of $H^i(D(F^-), S(F^-))$ where $F^-$ is the negative eigenspace of $Q$, $i= \dim (F^-)$ and $D(F^-), S(F^-)$ are respectively the disc and sphere in $F^-$, so that $\alpha \mapsto \alpha \otimes T_F$ is an isomorphism (the K\"unneth isomorphism) from $ H^*(U)$ to $H^{*+i}(U\times D(F^-), U\times S(F^-))= H^*(U)\otimes H^*(D(F^-), S(F^-))$ for $U\subset N$. By abuse of language we again denote by $T_F$ its homological counterpart in $H_i(D(F^-), S(F^-))$. 
We denote by $S_U^t= \{ (x,\xi) \in U \times E \mid S(x,\xi) \leq t\}$ (we omit the subscript for $U=N$) and $S_U^{-\infty}$  (resp. $S_U^{+\infty}$) any of the $S_U^{-c}$ (resp. $S^c$) for $c$ large enough (by Morse's lemma they are all isotopic). Classically we have a homotopy equivalence between $(S_U^{+\infty}, S_U^{-\infty})$ and $ U\times (D(F^-), S(F^-))$. 
\begin{defn} 
Let $S$ be  a \ac{GFQI} for $L\in \LL_{} (T^*N)$ and $U$ a bounded open set with smooth boundary. We define
\begin{enumerate} 

\item  For $\alpha \in H^*(U)$ $$c(\alpha, S)= \inf \left \{ t \mid T\otimes \alpha \neq 0 \;\text{in}\; H^*(S_{\mid U}^t, S_{\mid U}^{-\infty}) \right\}$$
\item For $a \in H_*(U,\partial U)$ 
$$c(a, S)= \inf\left\{ t \mid T\otimes a \,\text{is in the image of }\;  H_*(S_{\mid U}^t, S_{\mid U}^{-\infty}\cup S_{\mid \partial U}^t) \right \}$$
\item For $\alpha \in H_c^*(U)=H^*(U,\partial U)$  $$c(\alpha, S)= \inf \left \{ t \mid T\otimes \alpha \neq 0 \;\text{in}\; H^*(S_{\mid U}^t, S_{\mid U}^{-\infty}\cup S_{\mid \partial U}^t)\right \}$$
\item For $a \in H_*(U)$ 
\begin{displaymath} c(a, S)= \inf \left \{ t \mid T\otimes a  \,\text{is in the image of }\;H_*(S_{\mid U}^t, S_{\mid U}^{-\infty})\right \}
\end{displaymath} 
\item For $L_1, L_2$ in $\LL_{}(T^*N)$,  having unique \ac{GFQI} , $S_1, S_2$, we set $(S_1\ominus S_2)(x;\xi,\eta)=S_1(x;\xi)-S_2(x;\eta)$ and $\alpha \in H^*(U)$ or $H^*(U,\partial U)$
$$c(\alpha, L_1,L_2)=c(\alpha , (S_1\ominus S_2))$$
We set $c(\alpha, L)=c(\alpha, L, 0_N)$. 
\item We set $\mu_U\in H^n(U,\partial U), 1_U\in H^0(U)$ and $\gamma_U(L_1,L_2)=c(\mu_U,L_1,L_2)-c(1_U,L_1,L_2)$. If $c(1_U, L_1,L_2)=0$ we shall write $L_2\preceq_U L_1$ and if this holds for all bounded sets $U$, we write $L_2\preceq L_1$. 
\item We set  $GH^*(L_1,L_2;a,b)=H^{*-i}((S_1\ominus S_2)^b,(S_1\ominus S_2)^a)$
\end{enumerate} 
\end{defn} 
\begin{rem}\label{rem-5.2} We notice that 
\begin{enumerate} 
\item As we said, $S$ is shorthand for $S_\nu$ defined on $U_\nu$. As long as $U\subset U_\nu$ it is easy to see that for $\alpha \in H^*(U)$ (resp. $H^*(U,\partial U)$) the $c(\alpha, S_\nu)$ do not depend on $\nu$. 
\item  The function $(S_1\ominus S_2)$ is not quadratic at infinity, but a standard trick allows us to deform it to a function quadratic at infinity (see \cite{Viterbo-Montreal}, prop. 1.6).
The $GH^*$ functor is called Generating function homology (see \cite{Traynor}) and coincides with Floer homology (see \cite{Viterbo-FCFH2} for the equivalence of the two homologies) that we shall not introduce here.
 \item \label{rem-5.2-2} Note that if $S$ has no fiber variables, $c(1_U,S)=\inf_{x\in U} S(x)$ and  $c(\mu_U,S)=\sup_{x\in U} S(x)$. 
\end{enumerate} \end{rem} 

It is often convenient to express the cohomological critical values in terms of their homology counterpart. Note that $H^*(U)$ is dual to $H_{n-*}(U,\partial U)$ and
$H^*(U,\partial U)$ is dual to $H_{n-*}(U)$ by Lefschetz duality (cf \cite{Hatcher}, p. 254). 
We have a fundamental class $\mu_U \in H^n(U,\partial U)$ dual to $[pt_U] \in H_0(U)$  and $1_U\in H^0(U)$ dual to $[U] \in H_n(U,\partial U)$. 
The following Lemma will be useful
\begin{lem} \label{Lemma-5.3}
We have for $S$ a \ac{GFQI}
\begin{enumerate} 
\item $c(1_U,S)=c([pt_U],S)$
\item $c(\mu_U,S)=c([U],S)$
\end{enumerate} 
We also have the duality identity \begin{displaymath} c(1_U,\overline L)=-c(\mu_U,L) \end{displaymath} 
\end{lem} 
\begin{proof} The first two properties follow from Proposition B.3 in \cite{SHT}. 
 The duality identity is a consequence of the identity $c(1_U,-S)=-c(\mu_U,S)$. Both are easily adapted from the case $U=N$ closed to the present situation. 
 This follows from the following argument (see \cite{Viterbo-STAGGF}, prop. 2.7, p. 692).
 We give here details for the proof of duality.
 First notice that $({-S})^t=E\setminus S^{-t}$, so we look for the smallest $t$ such that $1_U\neq 0$ in $H^*(E_{\mid U}\setminus S_{\mid U}^{-t}, E_{\mid U}\setminus S_{\mid U}^{-\infty})$. We then apply Alexander duality, noting that even though $E_U=U\times {\mathbb R}^k$ is non-compact, it is contained in the compact manifold $U\times S^k$. We then get the diagram where vertical maps correspond to long exact sequences of triples, and horizontal to Alexander isomorphisms (omitting the subscript $U$)
\begin{displaymath}  \xymatrix@C3em{H_*(S^{-t}, S^{-\infty}) \ar[r]^{\simeq}\ar[d]&H^*(E\setminus S^{-\infty}, E\setminus S^{-t})\ar[d]&= H^*((-S)^t,(-S)^{-\infty})\ar[d] \\
H_*(S^{+\infty}, S^{-\infty}) \ar[r]^{\simeq}\ar[d]&H^*(E\setminus S^{-\infty}, E\setminus S^{+\infty})\ar[d]&= H^*((-S)^{+\infty},(-S)^{-\infty})\ar[d]\\
H_*(S^{+\infty}, S^{-t}) \ar[r]^{\simeq}&H^*(E\setminus S^{-t}, E\setminus S^{\infty})&= H^*((-S)^{t},(-S)^{-\infty}) }
 \end{displaymath} 
 Using the universal coefficient theorem (remember, our coefficient ring is a field) we see that  $H_*(S_{\mid U}^{+\infty}, S_{\mid U}^{-\infty})$  is a vector space dual to $H^*(S_{\mid U}^{+\infty}, S_{\mid U}^{-\infty})$. We denote again $1_U$ the element in $H_*(S_{\mid U}^{+\infty}, S_{\mid U}^{-\infty})$ sent to $1_U\in H^*(S_{\mid U}^{+\infty}, S_{\mid U}^{-\infty})$ to $1$, and  we have $c(1_U,S)$ is the same whether we consider $1_U$ in homology or cohomology. 
 On the other hand the second line of the diagram sends $1_U$ to $\mu_U$, since in this case Alexander duality corresponds to Poincar\'e duality. 
 Now saying that $1_U$ is in the image of  $H_*(S^{-t}, S^{+\infty})$ is equivalent to saying that $\mu_U$ is in the image of $H^*((-S)^t,(-S)^{-\infty})$. In other words, $-t \leq c(1_U,S)$ is equivalent to $t\leq c(\mu_U,S)$ and this means $c(1_U,-S)=-c(\mu_U,S)$. 
\end{proof} 

\begin{defn} 
Let $U$ be a bounded domain with smooth boundary, $\partial U$. We say that the sequence of smooth  functions $(f_k)_{k\geq 1}$ {\bf defines} $U$ if
 \begin{enumerate} 
 \item  there is a decreasing  family $F_k$ of closed subset of $U$ such that $\bigcap_k F_k=\overline U$
 \item $f_k=0$ on $F_k$
 \item  $f_k$ is a decreasing sequence converging to $-\infty$ on $N\setminus U$. 
 \end{enumerate} 
 We say that $(f_k)_{k\geq 1}$ is a {\bf standard defining sequence} if there is a function $r\in C^\infty( {\mathbb R})$ such that 
 \begin{enumerate} 
 \item  $r(t)=0$ for $t\leq 0$ 
 \item  $r'(t)<0$ for $0<t<1$, 
 \item $r(t)=-1$ for $t\geq 1$ 
 \end{enumerate} and for some increasing sequence $a_k$ converging to $+\infty$ we have $$f_k(x)=a_k r_k(a_k\cdot d(x,U))\dispdot$$ 
 \end{defn} 
 Notice that given a sequence $f_k$ defining $U$, we can find  standard sequences $g_k, h_k$ such that $g_k\leq f_k\leq h_k$. 
 
 We define for a smooth function $f$ the graph of its differential, $G_f=\{(x, df(x) \mid x \in N\}$. This is an exact Lagrangian, with primitive $f$. If $L$ is a Lagrangian with \ac{GFQI} $S$, we define $L+G_f$ the Lagrangian generated by $S+f$. 

We notice that
\begin{lem} \label{Lemma-4.3}
Let $(f_k)_{k\geq 1}$ be a sequence defining $U$,   and $V$ be any bounded open set such that $V \supset \overline U$. Then  for $L_1, L_2 \in \LL_{}(T^*N)$ we have
 \begin{displaymath} c(1_U,L_1, L_2)=\lim_k c(1_V, L_1- G_{f_k}, L_2) =\lim_k c(1_V, L_1, L_2+ G_{f_k})
 \end{displaymath} 
\end{lem} 
\begin{proof} 
 Let $S_j$ be  G.F.Q.I. for $L_j$ and $S= S_1\ominus S_2$. We have $S^c_{\mid U}=\lim_k (S-f_k)^c_{\mid V}$, therefore for \v Cech cohomology, according to theorem 5 in \cite{L-R} we have
 \begin{displaymath} 
 \lim_k H^*((S-f_k)_{\mid V}^c, (S-f_k)_{\mid V}^{b})=H^*(S_{\mid U}^c, S_{\mid U}^{b})
 \end{displaymath} 
 and from the definition of $c(1_U,S)$ the Proposition follows. 
 \end{proof} 
 
 Let $U$ be an open set with smooth boundary and set $\nu(x)\in T_x^*U$ to be  the exterior conormal to $\partial U$ at $x\in \partial U$, i.e. $\nu(x)=0$ on $T\partial U$ and $\langle \nu(x),n(x) \rangle >0$ where $n(x)$ is the exterior normal to $U$ at $x$. The conormal of $U$ is then defined as
$$\nu^*U=\{(x,p) \in T^*N \mid x\in U, p=0, \; \text{or}\; x\in \partial U, p=c\nu(x), c\leq 0\}$$
We now prove that the values of $c(\alpha, L)$ correspond to intersection points of $L$ and $\nu^*U$ (or $L$ and $\overline \nu^*U$) 

\begin{prop} [Representation theorem] \label{Prop-Representation} 
Let $U$ be a bounded open set with smooth boundary, we then  have 
\begin{enumerate} 
\item For $\alpha \in H^*(U)$, $c(\alpha; L_1,L_2)$ is given by $f_1(x_\alpha, p_{1,\alpha})-f_2(x_\alpha,p_{2,\alpha})$ where   $(x_\alpha, p_{1,\alpha})\in L_\alpha$ and $(x_\alpha, p_{2, \alpha}\in L_2$ and $(x_\alpha, p_{1,\alpha}-p_{2,\alpha}) \in \nu^*U$.
\item The same holds for $\alpha \in H^*(U,\partial U)$ but with $\overline{\nu^*U}$. 
\end{enumerate} 
\end{prop} 
\begin{proof} 
 This is the representation theorem (Proposition 2.4 in \cite{Viterbo-STAGGF}), using a standard defining sequence for $U$ and the fact that $c(1_U;L_1,L_2)=\lim_k c(1_V; L_1-G_{f_k},L_2)$. Indeed, a converging sequence of points in $G_{f_k}$ will converge to a point in $\nu^*U$. Then  compactness of $L_1 \cap T^*U$ and $L_2\cap T^*U$ imply the result.
  \end{proof}

For   $(f_k)_{k\geq 1}$ a defining sequence of $U$, then $\nu^*U$  is the ``limit'' of  the $G_{f_k}$ for $k\geq 1$. We will formally write $c(\alpha, L,\nu_*U)$ for $c(\alpha_U, L)$. 

  \begin{rems}  \label{rem-5.7}
Here are some comments :
  \begin{enumerate} 
 \item The same will hold for $U \subset V$ and any $\alpha_V \in H^*(V)$ having restriction $\alpha_U\in H^*(U)$ : 
  \begin{displaymath} c(\alpha_U,L_1, L_2)=\lim_k c(\alpha_V, L_1- G_{f_k}, L_2) =\lim_k c(\alpha_V, L_1, L_2+ G_{f_k})\dispdot 
 \end{displaymath} 

In particular, if $M$ is a closed manifold containing $N$, we have 
  \begin{displaymath} c(1_U,L_1, L_2)=\lim_k c(1_M, L- G_{f_k}) = \lim_k c(1_M, L_1, L_2+ G_{f_k}) 
 \end{displaymath}
\item \label{rem-5.7-2} Let  $\overline U \subset V$, then with obvious abuse of notations $c(1_V, \nu^*U)=-\infty, c(\mu_V,\nu^*U)=0$ and of course $c(1_U, \nu^*V)=0, c(\mu_U, \nu^*V)=+\infty $. This means that for $(f_k)_{k\geq 1}$  and $(g_k)_{k\geq 1}$ defining $U$ and $V$, we have
$\lim_k c(1_M, G_{f_k}, G_{g_k})=-\infty$ and $\lim_k c(\mu_M, G_{f_k}, G_{g_k})=0$.
 \item \label{rem-5.7-3} Note also that symbolically we have for $\overline U \subset V$ that $\nu^*U+\nu^*V=\nu^*U$, meaning that if $(f_k)_{k\geq 1}$ defines $U$ and $(g_k)_{k\geq 1}$ defines $V$ then $(f_k+g_k)_{k\geq 1}$ defines $U$.
  More generally if $U\cap V \subset W$ we have $\nu^*U+ \nu^*V \preceq \nu^*W$ where this means
  that if $(f_k)_{k\geq 1}$ defines $U$ and $(g_k)_{k\geq 1}$ defines $V$, there is a sequence $(h_k)_{k\geq 1}$ defining $W$ such that $f_k+g_k\leq h_k$
 \end{enumerate} 
 \end{rems}

We will now prove some of the properties of these invariants
 \begin{prop}\label{Prop-3.3} Let $\varphi\in \DHam_{FP}(T^*N)$  and $L=\varphi^1(0_N)$ be a Lagrangian submanifold. 
 We have \begin{displaymath} \gamma_U(L)=c(\mu_U,L) - c(1_U,L) \geq 0 \end{displaymath}  and equality implies that  $L\cap T^*U \supset 0_U$. 
  \end{prop} 
 \begin{proof} 
 The proof follows from the triangle inequality (see \cite{Viterbo-STAGGF}, prop 3.3 p.693)  $c(\alpha \cup \beta ,L) \geq c(\alpha,L)$ applied to  the product 
 \begin{displaymath} H^*(U)\otimes H^*_c(U) \longrightarrow H^*_c(U) \end{displaymath} 
 Thus we have $c(\mu_U,L)=c(1_U\cup \mu_U ,L) \geq c(1_U,L)$ and equality implies that $\mu_U$ is non zero in $K_c\simeq L\cap \overline {\nu^*U}$. But this implies $\pi (L\cap \nu^*U ) \supset U$, hence $L$ contains $0_U$. 
 Note that in general, contrary to the case where $N=U$ is compact  $L\cap T^*U$ may contain other connected components than $0_U$.

 \end{proof} 
 
\begin{prop}\label{Prop-5.8} The following holds for $L_i\in \mathfrak L_{}(T^*N)$
\begin{enumerate} 
\item\label{Prop-5.8-1} We have $c(\mu_U, L_1,L_2)=-c(1_U,L_2,L_1)=-c(1_U,\overline L_1,\overline L_2)$. 
\item \label{Prop-5.8-2} For $U\subset V$ and $L_1,L_2$ Lagrangian submanifolds we have 
\begin{itemize}
\item $c(\mu_U, L_1,L_2) \leq c(\mu_V,L_1,L_2)$
\item $c(1_U, L_1,L_2) \geq c(1_V,L_1,L_2)$
 \item $\gamma_U(L_1,L_2) \leq \gamma_V(L_1,L_2)$
 \end{itemize}
\item \label{Prop-5.8-3} We have $\gamma_U(L_1,L_3)\leq \gamma_U(L_1,L_2)+\gamma_U(L_2,L_3)$
\item  \label{Prop-5.8-4} If $\gamma_U(L_1,L_2)=0$ then $L_1\cap L_2$ has a connected component with projection on $N$ containing $U$. 
\item If $L_1\preceq L_2$ then $c(\alpha, L_1) \leq c(\alpha, L_2) $ for all $\alpha$. 
\end{enumerate} 
\end{prop} 
\begin{proof} 
\begin{enumerate} 
\item The proof is the same as in Lemma \ref{Lemma-5.3}, since $S_{L_1}\ominus S_{L_2}=- (S_{L_2}\ominus S_{L_1})$. 
\item If $U\subset V$ note that \begin{displaymath} c(1_U;L_1,L_2)= \lim_k c(1_N, L_1-G_{f_k}, L_2) \end{displaymath} 
Since we may choose defining sequences $(f_k)_{k\geq 1}, (g_k)_{k\geq 1}$ for $U, V$ such that $f_k \leq g_k$ we have for $S_1$ a \ac{GFQI}   of $L_1$ that   $S_1-f_k \geq S_1-g_k$ hence $c(1_N, L_1- G_{f_k}) \geq c(1_N, L_1-G_{g_k})$, 
  and going to the limit $c(1_U,L_1,L_2)\geq c(1_V,L_1,L_2)$. By the  duality formula (\ref{Prop-5.8-1}), we get $c(\mu_U;L_1,L_2)\leq c(\mu_V;L_1,L_2)$, hence $\gamma_U(L_1,L_2) \leq \gamma_V(L_1,L_2)$. 
\item We have \begin{displaymath} S_1\ominus G_{2\cdot f}\ominus S_3 = (S_1\ominus G_f)\ominus (S_3\oplus G_f)\end{displaymath} and  $(S_1\ominus G_f)\ominus S_2=S_1\ominus (G_f\oplus S_2)$.
Now remarking that if $(f_k)_{k\geq 1}$ defines $U$,  then so does $(2\cdot f_k)_{k\geq 1}$ we have 

\begin{gather*}  \gamma_U(L_1,L_3)= \lim_k \gamma_V( S_1\ominus G_{2\cdot f_k}\ominus S_3) =\\  \lim_k \gamma_V( (S_1\ominus G_{f_k})\ominus (S_3\oplus G_{f_k})) \leq  \lim_k \gamma_V(S_1\ominus G_{f_k}\ominus S_2) +  \lim_k \gamma_V( S_2\ominus (G_{f_k}\oplus S_3))=\\
\gamma_U(L_1,L_2) + \gamma_U(L_2,L_3) \end{gather*} 

 \item This follows from Lusternik-Schnirelmann theory as in the proof of Prop. 2.2 page 691 of \cite{Viterbo-STAGGF} (see also Proposition \ref{Prop-3.3}). 
 \item $L_1\preceq L_2$ implies $c(\mu_U, L_1,L_2)=0$ for all $U$. By the triangle inequality, if $\beta \cup \alpha = \mu_U$ we have
 \begin{displaymath} 
 0=c(\mu_U, L_1, L_2) \geq c(\alpha , L_1, 0_N)+c(\beta , 0_N, L_2)\geq c(\alpha, L_1)-c(\alpha, L_2)
  \end{displaymath} 
  since $c(\beta, 0_N, L)=-c(\alpha,L, 0_N)$ according to the proof of proposition B.3 in \cite{SHT}. 
\end{enumerate}
\end{proof}

We must now see what happens when we make a  coordinates change in  $T^*N$. 
We start with  three lemmata
\begin{lem} \label{Lemma-5.9}
Let $S$ be a G.F.Q.I.  defined on $E=Y\times F $ and for 
 $f:X \longrightarrow Y$ a smooth map a map $\widetilde f : X\times F \longrightarrow Y\times F$ living over $f$, i.e.  the following diagram

\centerline{
\xymatrix{X\times F \ar[d] \ar[r]^{\widetilde f}& Y\times F \ar[d]\\
X \ar[r]^f & Y
}
}
 is commutative.
We then have  for $\alpha \in H^*(Y)$ and $(f )^*(\alpha)\in H^*(X)$
\begin{displaymath} 
c(\alpha , S) \leq c(f^*(\alpha), S\circ \widetilde f)
\end{displaymath} 
\end{lem} 
 \begin{proof} 
 Indeed, if $T\in H^*(D(F^-), S(F^-))$ is the Thom class for $F^-$, then $(\widetilde f)^*(T)=\widetilde T$ is the Thom class for $\widetilde f^*(F^-)$ and we have, denoting by $\pi, \widetilde \pi$ the projections on $Y$ and $X$, 
 \begin{displaymath} 
 (\widetilde f)^*(T\cup \pi^*(\alpha))= \pi^*(f^*(\alpha))\cup \widetilde T
 \end{displaymath} 
 Now if $c \leq c(\alpha, S)$ then $\pi^*(\alpha) \cup T$ vanishes in $H^*(S^c, S^{-\infty})$ and this implies that $ (\widetilde f)^*(T\cup \pi^*(\alpha))= \pi^*(f^*(\alpha))\cup \widetilde T$ vanishes in $H^*((S\circ \widetilde f)^c, (S\circ \widetilde f)^{-\infty})$, i.e. $c \leq c(f^*(\alpha), S\circ \widetilde f)$. This implies the lemma. 
 \end{proof}

 \begin{lem} \label{Lemma-5.10}
 We have 
 \begin{displaymath} 
 c(1_U\boxtimes1_U; L_1\times L_2, \nu^*\Delta_N)= c(1_U;L_1,L_2)
 \end{displaymath} 
 \end{lem} 
 \begin{proof} 
Let $d^\varepsilon : N\times N \longrightarrow  {\mathbb R} $ be smooth function vanishing on $\Delta_N$ and converging as $ \varepsilon$ goes to $0$, to $-\infty\cdot (1-\chi_{\Delta_N})$ where $\chi_{\Delta_N}$ is the characteristic function of $\Delta_N$. For example we can choose 
$$d^\varepsilon (x,y)= \frac{-1}{ \varepsilon } d(x,y)
$$
 Similarly define $d^ \varepsilon_U(x,y)=d^ \varepsilon (x,y) +f^\varepsilon _U(x)+f^\varepsilon _U(y)$ where $f^\varepsilon _U$ converges to $-\infty(1-\chi_U)$ as $ \varepsilon $ goes to $0$.
 
 Setting $\left [S_1\boxtimes (-S_2)\right ] (x_1,x_2,\xi_1,\xi_2)=S_1(x_1,\xi_1)+ S_2(x_2,\xi_2)$, and 
 \begin{displaymath} \left [S_1\oplus (-S_2)\right ] (x,\xi_1,\xi_2)=S_1(x,\xi_1)+ S_2(x,\xi_2)
 \end{displaymath} 
 we may write
\begin{gather*}  c(1_{U\times U}; L_1\times L_2, \Delta_N)=\lim_{ \varepsilon \to 0}c(1_{N\times N}; (L_1-df^\varepsilon _U)\times (L_2-df^\varepsilon _U), \nu^*\Delta_N)=\\ \lim_{ \varepsilon \to 0}c(1_{N\times N}; (S_1-f^\varepsilon _U)\boxtimes (-S_2-f^\varepsilon _U), d^\varepsilon)=c(1_{N\times N}, [(S_1-f^\varepsilon _U)\boxtimes (-S_2-f^\varepsilon _U)]-d^\varepsilon )
\end{gather*} Now $\lim_{ \varepsilon \to 0} ( S_1\boxtimes (-S_2)- d^\varepsilon)^c= ( S_1\oplus (-S_2))^c$ and if $\delta :\Delta_N\longrightarrow N\times N$ is the diagonal map, $\delta^*(1_N\boxtimes 1_N)=1_{\Delta_N}$, so from 
 Lemma \ref{Lemma-5.9}, we get 
 \begin{displaymath} 
  c(1_U\boxtimes1_U; L_1\times L_2, \nu^*\Delta_N)\leq c(1_N, (S_1-f_U^\varepsilon )\oplus (S_2-f_U^\varepsilon ))\leq c(1_U;L_1,L_2)
 \end{displaymath} 
 Conversely we notice that given $c$, for $ \varepsilon $ small enough,  $(S_1\boxtimes (-S_2)-d^\varepsilon)^c$ is contained in a neighborhood of $\Delta_N$. 
 Thus if $1_U\otimes 1_U$ does not vanish in 
 \begin{displaymath} H_*([(S_1-f_U^\varepsilon )\boxtimes (-S_2-f_U^\varepsilon ))-d^\varepsilon]^c, [((S_1-f_U^\varepsilon )\boxtimes (-S_2-f_U^\varepsilon ))-d^\varepsilon ]^{-\infty})
 \end{displaymath} 
  i.e. $c\geq   c(1_U\boxtimes1_U; L_1\times L_2, \nu^*\Delta_N)$ then its restriction to $\Delta_N$, that is $1_U$ does not vanish either, and $c\geq c(1_U;L_1,L_2)$, so 
 \begin{displaymath} 
  c(1_U\boxtimes1_U; L_1\times L_2, \nu^*\Delta_N)\geq c(1_U;L_1,L_2)
 \end{displaymath} 
 and we have equality. 
 \end{proof} 

 \begin{lem} \label{Lemma-5.11}
 Let us consider a bounded open set with boundary $U\subset N$ and $\nu^*\Delta_U \subset T^*N\times \overline{T^*N}$, where $\Delta_U$ is the diagonal in $U$ Then if $\varphi^{1}(T^*U)\subset T^*V$
 we have 
 \begin{displaymath} 
( \varphi\times \varphi)(\nu^*\Delta_U) \preceq \nu^*\Delta_V \dispdot
 \end{displaymath} 
 \end{lem} 
 \begin{proof} 
 Let $(q,p,q,p')\in \nu^*\Delta_U$ and notice that unless $q\in \partial U$ we have $p=p'$. Then according to Lemma \ref{Lemma-4.2} we may assume $\varphi
 ^{t}(T^*U)\subset T^*V$ for all $t\in [0,1]$, so setting  $(\varphi^t\times \varphi^t)(q,p,q,p')=(Q_t,P_t, Q'_t,P'_t)$ we know that when $(q,p,q,p')\in \nu^*\Delta_U\subset T^*\overline{U}$, we have $Q_t, Q'_t \notin \partial V$. So if $(Q_t,P_t, Q'_t,P'_t)\in \nu^*V$ we must have $Q_t=Q'_t, P_t=P'_t$, but then 
 $p=p'$. In other words 
 \begin{displaymath} (\varphi^t\times \varphi^t)(\nu^*\Delta_U) \cap \nu^*\Delta_V=(\varphi^t\times \varphi^t)(\Delta_{T^*U})\cap \Delta_{T^*V}=(\varphi^t\times \varphi^t)(\Delta_{T^*U})
 \end{displaymath} 
 So the intersection $(\varphi^t\times \varphi^t)(\nu^*\Delta_U) \cap \nu^*\Delta_V$ is constant and by a classical argument, this implies that as a function of $t$ $c(\alpha, (\varphi^t\times \varphi^t)(\nu^*\Delta_U) , \nu^*\Delta_V)$ is constant. Since $\nu^*\Delta_U \preceq \nu^*\Delta_V$ we have for all $t$ 
$( \varphi\times \varphi)(\nu^*\Delta_U) \preceq \nu^*\Delta_V $.
 \end{proof} 

Using Proposition \ref{Prop-5.8} (\ref{Prop-5.8-2}), we may conclude that the limits in the following proposition are well-defined in $ {\mathbb R}\cup \{\pm \infty \}$.
\begin{defn} 
When $U$ is an unbounded set we define $\mathcal B (U)$ to be the set of bounded subsets in $U$ and 
\begin{displaymath} 
c(\mu_U, L_1,L_2)=\lim_{V \in \mathcal B(U)} c(\mu_V, L_1, L_2)
\end{displaymath} 
\begin{displaymath} 
c(1_U, L_1,L_2)=\lim_{V \in \mathcal B(U)} c(1_V, L_1, L_2)
 \end{displaymath} 

\end{defn}

We may now prove

\begin{prop} \label{Prop-5.12}
We have for $\varphi \in \DHam_{}(T^*N)$ such that $\varphi (T^*U) \subset T^*V$ and $L_1, L_2 \in \LL_{}(T^*N)$
\begin{displaymath} 
\gamma_U (\varphi(L_1), \varphi(L_2)) \leq \gamma_V(L_1, L_2)
\end{displaymath} 
\end{prop} 
\begin{proof} 
We use Lemma \ref{Lemma-5.9} so we replace $c(1_U, \varphi(L_1), \varphi(L_2))$ by 
\begin{displaymath} c(1_U\boxtimes 1_U, \varphi(L_1)\times \varphi(L_2), \nu^*\Delta_N)
 \end{displaymath}  and this in turn equals 
\begin{displaymath} 
c(1_N\boxtimes 1_N, (\varphi\times \varphi)(L_1\times L_2), \nu^*\Delta_N+\nu^*(U\times U))
\end{displaymath} 
Since \begin{displaymath} 
\nu^*\Delta_N + \nu^*(U\times U) \preceq \nu^*( \Delta_N\cap (U\times U))=\nu^*\Delta_U
\end{displaymath} 
As a result
\begin{gather*} 
c(1_N\boxtimes 1_N, (\varphi\times \varphi)(L_1\times L_2), \nu^*\Delta_N+\nu^*(U\times U)) \geq \\
c(1_N\boxtimes 1_N, (\varphi\times \varphi)(L_1\times L_2), \nu^*\Delta_U) = \\
c(1_N\boxtimes 1_N,(L_1\times L_2),  (\varphi\times \varphi)^{-1}(\nu^*\Delta_U))
\end{gather*} 
and using Lemma \ref{Lemma-5.11} we get that the last term is greater than 
\begin{displaymath} c(1_{N\times N}, L_1\times L_2, \nu^*\Delta_V)=c(1_V, L_1, L_2) 
\end{displaymath} 
We may thus conclude that 
\begin{displaymath} c(1_V, L_1, L_2)\leq c(1_U, \varphi(L_1), \varphi(L_2))
\end{displaymath} 
By duality, we get 
\begin{displaymath} c(\mu_V, L_1, L_2)\geq c(\mu_U, \varphi(L_1), \varphi(L_2))
\end{displaymath} 
and our result follows. 
\end{proof} 
\begin{defn} \label{Def-4.11}
A sequence $(L_k)_{k\geq 1}\in \LL_{} (T^*N)$ $\gamma_c$-converges to $L\in \LL_{} (T^*N)$  if for all bounded domains $U$ the sequence 
$\gamma_U(L_k,L)$ converges to $0$. We shall write $L_k \overset{\gamma_c} \longrightarrow L$. 
The completion of $\LL_{}(T^*N)$ for $\gamma_c$, is the set of equivalence classes of $\gamma-c$-Cauchy sequences $(L_k)_{k\geq 1}$ for the following relation:

$ (L_k)_{k\geq 1} \simeq (L'_k)_{k\geq 1}$ if for all bounded domains $U$, $\gamma_U(L_k,L'_k)$ converges to $0$.  

We denote by $\widehat \LL_{}(T^*N)$ this completion. 
\end{defn} 
\begin{rem} 
Of course we may take a cofinal sequence $U_k$ of bounded open sets in $N$ and define 
$$d(L_1,L_2)=\sum_{j=1}^{+\infty}2^{-j} \max\left\{1, \gamma_{U_j}(L_1,L_2)\right\}$$
 and then take the completion with respect to this metric. It is easy to see that the completion coincides with the above, hence does not depend on the choice of the sequence $U_k$ (this is just rephrasing the fact that the $\gamma_U$ define a uniform structure, see \cite{Weil-uniform}, or  \cite{Bourbaki}, chap. II) . 
\end{rem} 
  \begin{example} 
  Let $f_k$ be a sequence of smooth functions. Then $\gamma$-convergence of the $L_k= \gr (df_k)$  is equivalent to uniform convergence on compact sets of the $f_k$. 
  
  \end{example} 
  We shall need the following Proposition

\begin{prop} \label{Prop-5.16}
We have for $L=\varphi_H^1(0_N)\in {\mathfrak L}(T^*N)$ the inequalities
$$c(\mu_U, L) \leq \sup_{(q,p)\in T^*U} H(q,p)$$
$$c(1_U, L) \geq \inf_{(q,p)\in T^*U} H(q,p)$$
$$\gamma_U(L) \leq \Vert H \Vert _{C^0(T^*U)}$$
\end{prop}
\begin{proof}  
 Let $H(q,p)=h(q)$ and $L_h=\varphi_H(0_N)$. Then according to Remark \ref{rem-5.2} (\ref{rem-5.2-2}) we have $c(\mu_U, L_h) \leq \sup_{q\in U}h(q), c(1_U, L_h) \geq \inf_{q\in U}h(q)$  because $L=\{(q, dh(q))\mid q \in N\}$. Now since for $H \leq h(q)=\sup_{p\in T_q^*N}$ we have $H \leq H$ we get
 $L\leq L_h$, so $c(\mu_U, L) \leq c(\mu, L_h)= \sup_{q\in U}h(q)= \sup _{(q,p)\in T^*U} H(q,p)$ we get the first equality. The other two equalities follow immediately  from this one. 
 \end{proof} 

\subsection{The case of Hamiltonians in $T^*{\mathbb R}^{n}$}
Let $H \in \HH_{fc}([0,1]\times T^* {\mathbb R}^n)$ and $\varphi_H^t$ be its flow. Let $s_1, s_2$ the symplectomorphisms
$$ T^*{\mathbb R}^n \times \overline{ T^* {\mathbb R} ^n} \longrightarrow T^*\Delta_{T^* {\mathbb R}^n}$$
defined respectively by
$$s_1(q,p,Q,P)=(q,P,p-P, Q-q)$$
$$s_2(q,p,Q,P)=(Q,p,p-P,Q-q)$$
Denoting by $(x,y,X,Y)$ the coordinates in $T^* \Delta_{T^* {\mathbb R}^n}$, we have $s_i^*(dY\wedge dy+dX\wedge dx)=dp\wedge dq-dP\wedge dQ$, so the $s_i$ are symplectic. 

The graph of $\varphi_H$ is $(\id\times \varphi_H)(\Delta_{T^* {\mathbb R}^n})$, and its image by $s_1$ is denoted by $\Gamma (\varphi_H)$, while its image by $s_2$ will be $\overline{\Gamma (\varphi_H^{-1})}$. Let $S_H$ be a \ac{GFQI} for $\Gamma (\varphi_H)$ which exists and is unique  if $H \in \HH_{BP}(T^* {\mathbb R}^n)$ by Theorem \ref{Thm-4.5}.  
 \begin{defn} \label{Def-5.15}
 We set for $W$ a  domain contained in  $\Delta_{T^* {\mathbb R}^n}$
 \begin{enumerate} 
 \item $c^-_W(\varphi_H)=c(1_W, \Gamma(\varphi_H))$
 \item  $c^+_W(\varphi_H)=c(\mu_W, \Gamma(\varphi_H))$
\item  $\gamma_W(\varphi_H)=c^+_W(\varphi_H)-c^-_W(\varphi_H)$
 \end{enumerate} 
 \end{defn} 
 \begin{rem} 
 In $T^*N$ we may define for $U\subset N$ the numbers $$\widehat\gamma_U(\varphi_H)=\sup_{L\in \LL(T^*N)} \gamma_U(L, \varphi_H(L))$$ which corresponds to(even though we do not claim it is equal to) $\gamma_{(U \times {\mathbb R}^n)}(\varphi_H)$. 
 \end{rem} 

Let then $(H_\nu)_{\nu \geq 1}$ be a sequence of Hamiltonians in ${\mathfrak Ham}_{FP}(T^* {\mathbb R}^n)$ and $\varphi_\nu=\varphi_{H_\nu}$. 
\begin{defn} \label{Def-5.21}
The sequence $(\varphi_\nu)_{\nu\geq 1}$ $\gamma_c$-converges to $\varphi$ if for all bounded domains $W$ we have $\lim_\nu \gamma_W(\varphi_\nu,\varphi)=0$. 
The $\gamma$-completion  $\widehat{\mathfrak {DHam}}_{FP}(T^* {\mathbb R}^n)$  is defined as the set of Cauchy sequences in  ${\mathfrak {DHam}}_{FP}(T^* {\mathbb R}^n)$ for the uniform structure defined by the $\gamma_W$, in other words the set of sequences which are Cauchy for each $\gamma_W$, modulo the equivalence relation $(\varphi_\nu)_{\nu\geq 1} \simeq (\psi_\nu)_{\nu\geq 1}$ if for all $W$ we have $\lim_{\nu}\gamma_W(\varphi_\nu, \psi_\nu)=0$. 
Similarly we define for $H\in \mathfrak{Ham}_{FP}(T^* {\mathbb R} ^n)$ the metric $\gamma_W(H)=\sup_{t\in [0,1]} \gamma_W(\varphi_H^t)$ and similarly for $\gamma_W(H,K)$. Then we define the $\gamma_c$-convergence of a sequence in  $\mathfrak{Ham}_{FP}(T^* {\mathbb R} ^n)$ and its completion  $\widehat{\mathfrak{Ham}}_{FP}(T^* {\mathbb R} ^n)$. 
\end{defn} 
Note that the property of having FPS or being in ${\mathfrak Ham}_{fc}$ can be checked in the $\gamma$-completion. Indeed  $\varphi (T^*U)\subset T^*V$ is equivalent to 
\begin{displaymath} \Gamma (\varphi) \cap \{(x,p_x,y,p_y) \mid x\in U, y\notin V \} = \emptyset
\end{displaymath} 
and being supported in $ \vert p \vert \leq r$ is equivalent to 
\begin{displaymath} \Gamma (\varphi) \cap \{(x,p_x,y,p_y) \mid  \vert p_x \vert \geq r \} \subset \Gamma (\id)
\end{displaymath} 
 and both  are closed condition, which makes sense in the completion (see \cite{Hum}).
\begin{prop} 
We have 
$$ \gamma_W(H) \leq \Vert H \Vert _{C^0{(W')}}$$
As a result we have an embedding 
$$\left (C^0_{fc}([0,1]\times T^* {\mathbb R} ^n), d_{C^0}\right ) \longrightarrow \mathfrak{Ham}_{fc}(T^* {\mathbb R}^n)\subset  \mathfrak{Ham}_{BP}(T^* {\mathbb R} ^n)$$
\end{prop} 
When dealing with fiberwise compact supported Hamiltonians, we have 
\begin{defn} 
We set for $\varphi \in \mathfrak{Ham}_{fc}(T^*{\mathbb R}^n)$ $$\gamma_r(\varphi)=\gamma_{ {\mathbb R}^n \times B^n(r)}(\varphi)= \lim_{R\to +\infty} \gamma_{B^n(R)\times B^n(r)}(\varphi)$$ and 
$$\gamma_c(\varphi)= \lim_{r\to \infty} \gamma_r(\varphi) \in {\mathbb R} \cup \{+\infty\}$$
\end{defn} 
\begin{prop} 
We have if $h_-(p)\leq H(t,q,p)\leq h_+(p)$ the inequality  
$$\gamma_r(\varphi_H)\leq \sup_{ \vert p \vert \leq r} h_+(p) -\inf_{ \vert p \vert \leq r} h_-(p)$$

\end{prop} 

\begin{rem} 
The quantity $\gamma_c(\varphi)$ is finite for $\varphi \in \mathfrak{Ham}(T^* {\mathbb R}^n)$ such that $ \Vert H \Vert_{C^0 ( T^* {\mathbb R}^n)}< +\infty$
\end{rem}

Our last results in this section will be
\begin{prop} \label{Prop-5.27}
We have the following
\begin{enumerate} 
\item \label{Prop-5.27-1} Assume $\psi, \psi^{-1}$ send $W=U\times V$ into $W'=U'\times V'$, where $U, U' \subset {\mathbb R}^n, V, V' \subset ({\mathbb R}^n)^*$ then we have 
$$\gamma_W(\psi^{-1}\circ \varphi \circ \psi)\leq \gamma_{W'}(\varphi)$$
\item\label{Prop-5.27-2} $\gamma_r(\tau_{-a}\circ \varphi\circ \tau_a) =\gamma_r(\varphi)$
\item $\gamma_r(\rho_{\varepsilon }\circ \varphi\circ \rho_\varepsilon ^{-1}) = \varepsilon \gamma_r(\varphi)$
\end{enumerate} 
\end{prop} 
\begin{proof} 
\begin{proof} 
In $T^*( {\mathbb R}^n\times {\mathbb R}^n)$ we have that  $\gamma(\varphi)$ is obtained by $(q,p,Q,P) \mapsto (q,P,P-p,q-Q)$ where $(Q,P)=\varphi(q,p)$, while $\Gamma(\psi\circ \varphi\circ \psi^{-1})$ is obtained by applying $\psi\times \psi$ to $(q,p,Q,P)$. In other words writing $(q',p')=\psi(q,p), (Q',P')=\psi(Q',P')$, $\Gamma(\psi\circ \varphi\circ \psi^{-1})$ is obtained as 
\begin{displaymath} \left\{(q',P',P'-p', q'-Q') \mid \varphi(q,p)=(Q,P)\right \} \end{displaymath} 
 Now if $ q \in U$ and $P \in V$ we have 
$q'\in U'$ and $ P'\in V'$, hence $(\psi\times \psi)(T^*(U\times V)) \subset T^*(U'\times V')$.

As a result, since $\psi\times \psi$ preserves the diagonal (that is the zero section in the new coordinates) we have
\begin{displaymath} 
\gamma_{U\times V}(\psi^{-1}\varphi\psi)=\gamma_{U\times V}((\psi\times \psi)\Gamma(\varphi),  (\psi\times \psi)(\Delta))\leq \gamma_{U'\times V'}(\Gamma(\varphi), \Delta) = \gamma_{U'\times V'}(\varphi)
\end{displaymath} 
Statement  (\ref{Prop-5.27-2}) follows from first applying (\ref{Prop-5.27-1}) to $\psi=\tau_a$ so that, setting $U_a= \bigcup_{t\in [-a,a]} \tau_t (U) $
 \begin{displaymath} 
 \gamma_{U\times B(r)}(\tau_{-a}\varphi\tau_a)\leq \gamma_{U_a\times B(r)} (\varphi)
 \end{displaymath} 
 hence taking the limit for $U\subset {\mathbb R}^n$ we get 
 \begin{displaymath} 
 \gamma_r(\tau_{-a}\varphi\tau_a)\leq \gamma_{r} (\varphi)
 \end{displaymath} and changing $a$ to $-a$ we get equality. 
 The second equality is rather obvious since $\rho_k$ is $k$-conformal and $\rho_k(U\times B_r)=(k\cdot U) \times B_r$. 
\end{proof} 

\begin{rem}
One should be careful, in particular $\gamma_U(\varphi_1,\varphi_2)$ is NOT in general, equal to $\gamma_U(\varphi_2^{-1}\circ \varphi_1)=\gamma_U(\varphi_2^{-1}\circ \varphi_1, \id )$. We thus have {\it a priori} two types of convergence. We could say that $\varphi_\nu$ converges to $\varphi$ if for all bounded sets $U$ either the sequence  $\gamma_U(\varphi_\nu,\varphi)$ goes to $0$ or if  $\gamma_U(\varphi_\nu \varphi^{-1})$ goes to $0$. However if the $\varphi_\nu$ have { uniformly bounded propagation speed}\label{U-BPS}, that is $\varphi_\nu (T^*B_r) \subset T^*B_{r+r_0}$ for all $\nu$ and all $r$, then the two conditions are equivalent. 
)
\end{rem}

\end{proof} 
\section{Compactness and ergodicity}\label{Section-Compactness-ergodicity}\label{Section-6}

Let $H: T^* {\mathbb R}^n \times \Omega\longrightarrow {\mathbb R} $ be  Hamiltonian satisfying properties (\ref{C1})-(\ref{C6}). Then each $H_\omega=H(\bullet, \bullet, \omega)$ is in ${\mathfrak {Ham}}_{fc}(T^* {\mathbb R}^n)$ and we identify $\Omega$ with its image in ${\mathfrak {Ham}}_{fc}(T^* {\mathbb R}^n)$, denoted $\mathfrak {H}_\Omega$. Its closure for the $\gamma$-topology in the completion $\widehat{\mathfrak{Ham}}_{BP}(T^* {\mathbb R} ^n)$ is denoted by $\widehat{\mathfrak H}_\Omega$.
The action  $ \tau$ of $ {\mathbb R}^n$ on $\Omega$ induces an action on $\mathfrak H_\Omega$ by
\begin{displaymath} 
(\tau_aH)(x,p;\omega)=H(x+a,p;\omega)=H(x,p;\tau_{-a}\omega)
\end{displaymath} 
This action translates into $\varphi \mapsto \tau_{-a}\varphi\tau_a$ on ${\mathfrak {DHam}}_{fc}(T^*{\mathbb R}^n)$.

We first want to prove 
\begin{prop} \label{Prop-6.1}
The abelian group $ {\mathbb R}^n$ acts continuously by isometries on $({\mathfrak {Ham}_{fc}}(T^*{\mathbb R}^n), \gamma_c)$ and $( {\mathfrak {DHam}_{fc}}(T^*{\mathbb R}^n), \gamma_c)$ hence on  $(\widehat {\mathfrak {Ham}_{fc}}(T^*{\mathbb R}^n), \gamma_c)$ and $(\widehat {\mathfrak {DHam}_{fc}}(T^*{\mathbb R}^n), \gamma_c)$.
Therefore the action $\tau$ of  ${\mathbb R}^n$ on $\mathfrak {H}_\Omega$ is a continuous action by isometries for $\gamma_c$ which extends to a continuous action by isometries on $\widehat{\mathfrak H}_\Omega$. 
\end{prop} 
\begin{proof} That $ {\mathbb R}^n$ acts by isometries follows from  Proposition \ref{Prop-5.27} (\ref{Prop-5.27-2}). 
It is enough according to a theorem  by Chernoff and Marsden\footnote{Which claims that under our assumptions, a separately continuous action is jointly continuous. See \cite{C-M}, Theorem 1, extending a theorem of Ellis, in \cite{Ellis}.} to prove the separate continuity of the map $ {\mathbb R}^n \times \mathfrak {Ham}_{fc}(T^* {\mathbb R}^n) \longrightarrow \mathfrak {Ham}_{fc}(T^* {\mathbb R}^n)$ in each variable. In other words -since $\tau_a$ is an isometry it is obviously continuous in the second variable -  we must prove that for all $H\in \mathfrak {Ham}_{fc}(T^* {\mathbb R}^n)$, we have 
\begin{displaymath} \lim_{a \to 0} \gamma_c (H, \tau_a H)=0
\end{displaymath}  i.e. we want to prove that for all $r>0$,   $\lim_{a\to 0}\gamma_r(\tau_a^{-1}\varphi^{-1}\tau_a, \varphi)=0$. 
But $$\Gamma(\varphi)=\{(q,P, p-P,Q-q) \mid \varphi(q,p)=((Q,P)\}$$ while
$$\Gamma(\tau_a^{-1}\varphi\tau_a)=\{(q-a,P, p-P,Q-q) \mid \varphi(q,p)=((Q,P)\}$$
so that if $S(q,P;\xi)$ is a \ac{GFQI} for $\Gamma(\varphi)$ and $(\tau_aS)(q,P,\xi)=S(q-a,P;\xi)$  is a \ac{GFQI} for $\Gamma(\tau_a^{-1}\varphi\tau_a)$.
Since critical points of $S(q,P,\xi)$ are contained in $ \vert P \vert \leq R$ and $a \mapsto S(q-a,P;\xi)$ is uniformly continuous on $ \vert P \vert \leq R$, we get that  $c_W(\alpha,S\ominus \tau_aS)$ depends continuously on $a$, and for $a=0$ is equal to $0$ (since it is equal to $c_W(\varphi, \varphi)=0$). 
\end{proof}

Proposition \ref{Prop-6.1}  extends the action  $\tau$  to a continuous action by isometries of $\widehat{\mathfrak H}_\Omega$. Since $\Iso (\mathfrak H_\Omega, \gamma)\subset \Iso (\widehat {\mathfrak H}_\Omega, \gamma)$,   the map $\tau : {\mathbb R}^n \longrightarrow \Iso (\mathfrak H_\Omega, \gamma)$  extends to a map, still denoted $\tau$, from $ {\mathbb R}^n$ to $ \Iso (\widehat {\mathfrak H}_\Omega, \gamma)$. Since this is obviously a group morphism, its closure in  $ \Iso (\widehat {\mathfrak H}_\Omega, \gamma)$ is an abelian connected and complete metric group. 
\begin{prop} Let us denote the closure of $\tau ({\mathbb R}^n)$ in $ \Iso (\widehat {\mathfrak H}_\Omega, \gamma)$ by $\mathbb A_\Omega$.
Then $\mathbb A_\Omega$ is an abelian, connected and complete metric  group.  
\end{prop} 

The goal of this section is to prove that our assumptions on $H$ imply that $\mathbb A_\Omega$ is compact. For this it is enough to prove that $ \Iso (\widehat {\mathfrak H}_\Omega, \gamma_c)$ is compact, but this follows immediately by Ascoli-Arzela's theorem if we prove that $(\widehat {\mathfrak H}_\Omega , \gamma_c)$ is compact. 
Because by assumption $(\widehat {\mathfrak H}_\Omega , \gamma_c)$ is complete, it is enough to show that it is totally bounded, that is for any $ \varepsilon >0$, $(\widehat {\mathfrak H}_\Omega , \gamma_c)$ can be covered by finitely many $\gamma_c$-balls of radius $ \varepsilon $. Since $( {\mathfrak H}_\Omega , \gamma_c)$ is dense in $(\widehat {\mathfrak H}_\Omega , \gamma_c)$ it is enough to prove that $( {\mathfrak H}_\Omega , \gamma_c)$ is totally bounded. We shall prove slightly less but it will be good enough for our purposes :

\begin{prop} \label{Prop-6.2}
Let $\widehat \mu_\Omega$ be the push forward to $\widehat{\mathfrak H}_\Omega$ of the measure $\mu$ on $\Omega$. Then the support of $\widehat\mu_\Omega$ is totally bounded hence compact. 
\end{prop} 
This will follow from the following general result
\begin{prop} \label{Prop-6.3}
Let $(X,\mu)$ be a probability space endowed with a distance $d$ such that $(X,d)$ is separable. Let $G$ be a group acting on $X$ by measure-preserving isometries. Then $\supp (\mu)$ is totally bounded. 
\end{prop} 

We shall first prove

 \begin{lem} \label{Lemma-6.4}
 Let $\tau$ be an ergodic action of a group $G$ on a probability, separable\footnote{A separable topological space is a space having a countable dense subset.} metric space $(X,\mu,d)$. Then for $\mu$-almost all points $x\in X$ the orbit $G\cdot x$ is dense in $\supp (\mu)$. 
 \end{lem} 
 \begin{proof} 
 This is an immediate consequence of Birkhoff's ergodic theorem, but we shall give a simpler (or at least easier) proof.  
 Let $Y$ be countable and dense in $X$ and set
 $$W= \bigcup_{y\in Y, r\in \mathbb Q^*_+ \atop \mu(B(y,r))=0} B(y,r)$$
 If $\mu(B(x,r))=0$ for some $x\in X, r>0$ then $x\in W$. Indeed, we may assume $r$ is rational, and choose $y\in Y$ such that $d(y,x)<r/2$. Then $x\in B(y, r/2)$ so  
 $B(y,r/2)\subset B(x,r)$ and we get  $\mu(B(y,r/2))=0$.  This argument implies that $$W= \{x\in X \mid \exists U \;\text{open}\; x\in U, \mu(U)=0\}$$ and $W$ is $\tau$ invariant since $\tau$ preserves $\mu$ and the open sets. 
 Now because $W$ is a countable union of open sets of measure $0$, it is open and has measure $0$. We may then replace $X$ by $X\setminus W$, so we are reduced to the situation where all balls have $>0$ measure i.e. all open sets have positive measure. 
 
 Now let $(U_j)_{j\in \mathbb N}$ be a countable basis of open sets (since a separable metric space is second countable). Set $\tau_GA=\{\tau_gx \mid x \in A\}$, then  the orbit of $x$ misses $U_j$ if and only if
 $\tau_Gx\cap U_j=\emptyset $ i.e. $ x \notin \tau_G(U_j)$. 
The points with non-dense orbit must miss at least one $\tau_G(U_j)$ so they belong to $$\bigcup _j (X\setminus \tau_G(U_j)= X \setminus \bigcap_j \tau_G(U_j)$$
 but by ergodicity $\tau_G(U_j)$ being $\tau$ invariant has measure $1$ (since it cannot be zero, as its measure is at least the measure of $U_j$ that is positive by assumption). 
 Therefore $\bigcap_j \tau_G(U_j)$ as a countable intersection of measure one sets has measure one, and its complement has measure zero. 
  \end{proof}
 
 We are now in a position to prove Proposition \ref{Prop-6.3}.
 \begin{proof} [Proof of Proposition \ref{Prop-6.3}]  Let $\tau_Gx$ be a dense orbit in $\supp (\mu)$. We shall prove that $\tau_G(x)$ is totally bounded, arguing by contradiction. 
 Let $a_1,...,a_k...\in G$ be a sequence in $G$ such that 
 \begin{itemize}
 \item $\bigcup_{j=1}^k {\overline B}(\tau_{a_j}x, \varepsilon )$ does not cover $\tau_Gx$ where $\overline B(x,r)$ is the closed ball of radius $r$
 \item For all $i\neq j$ we have $B(\tau_{a_i}, \frac{\varepsilon}{2}) \cap B(\tau_{a_j}, \frac{\varepsilon}{2})= \emptyset $
 \end{itemize}
 We claim that if $\tau_Gx$ cannot be covered by finitely many balls of size $ \varepsilon $ then we may construct such a sequence by induction. Indeed, assume $a_1,..., a_k$ have been constructed satisfying the above properties. Then by the first property we may find $ a_{k+1}$ such that $\tau_{a_{k+1}}x\notin \bigcup_{j=1}^kB(\tau_{a_j}x, \varepsilon )$ and this implies $B(\tau_{a_j}x, \varepsilon /2 )\cap B(\tau_{a_{k+1}}x, \varepsilon/2 )=\emptyset$. Hence $a_1,..., a_{k+1}$ satisfy both properties. 
 But now we found infinitely many disjoint  balls of radius $ \varepsilon/2$ in $\tau_Gx$. Since $\tau_{a_j}x\in \supp(\mu)$ we have $\mu(B(\tau_{a_j}x, \varepsilon /2))>0$ and since all the balls $B(\tau_ax, \varepsilon/2)$ are isometric they have the same measure. But we cannot have infinitely many disjoint balls with the same positive measure, since the total measure of our space is $1$.
 \end{proof}  
 We may now conclude with
 \begin{proof} [Proof of Proposition \ref{Prop-6.2}]
 Here $G= {\mathbb R}^n$ and $\tau_a$ induces a measure preserving map on $(\mathfrak H_\Omega, \gamma, \widehat\mu_\Omega)$. This map is an isometry according to Proposition \ref{Prop-6.1}, so according to Proposition \ref{Prop-6.3} the support of $\widehat\mu_\Omega$ is totally bounded. 
 \end{proof} 
 \begin{rem}
 As we pointed out already in \cite{SHT}, there are not so many non-trivial examples of compact subset in $(\widehat {\mathfrak {Ham}}_{fc}(T^*{\mathbb R}^n), \gamma)$ or \newline   $(\widehat {\mathfrak {DHam}}_{fc}(T^*{\mathbb R}^n), \gamma)$ that is sets that are not already compact for the $C^0$-topology (since $\gamma$ is continuous for the $C^0$ topology on $\HH (T^*N)$ according to \cite{Viterbo-STAGGF}) and in $\DHam (T^*N)$ according to \cite{Seyfaddini}). In \cite{SHT} we proved that in $T^*T^n$ the sequence $(H_k)_{k\geq 1}$ where $H_k(q,p)=H(k\cdot q, p)$ is converging. Here we extend this to certain families of Hamiltonians on $ T^* {\mathbb R}^n$. 
 \end{rem}
 
 We thus proved that $\mathbb A_\Omega$ the closure of $ {\mathbb R}^n$ in $ \Iso (\widehat {\mathfrak H}_\Omega, \gamma)$ is a compact, connected, metric abelian group.  We are thus in the following situation : we have an action- again denoted by $\tau$- of $\mathbb A_\Omega$ on $\widehat {\mathfrak H}_\Omega$ continuous for $\gamma$, by isometries preserving $\widehat \mu_\Omega$. By compactness of $\mathbb A_\Omega$, we have that $\mathbb A_\Omega \cdot H$ is closed for all $H \in \widehat{\mathfrak H}_\Omega$. But since for almost all $H$ $\tau_{ {\mathbb R}^n}H$ is dense, we conclude that for almost all $H$ we have $\mathbb A_\Omega \cdot H= \widehat{\mathfrak H}_\Omega$. 
 To conclude we are reduced to the situation where
 \begin{enumerate}
 \item $\Omega=\mathbb A_\Omega$
 \item $\omega \longrightarrow H_\omega \in \widehat\HH_{BP} (T^*T^n)$ is continuous for the $\gamma$-topology
 \item On the subgroup $ {\mathbb R}^n$ in $\mathbb A_\Omega$ the action of $ {\mathbb R}^n$ on $\Omega$ can be identified with the action by translation of ${\mathbb R} ^n$ as a subgroup of $ \mathbb A_\Omega$. 
 \end{enumerate} 
 
 \section{Some results on compact abelian metric groups}\label{Section-Abelian}\label{Section-7}
 Let $\mathbb A$ be a compact metric abelian group having $ {\mathbb R}^n$ as  a dense subgroup (in particular $\mathbb A$ is connected). According to A. Weil
 (\cite{Weil2} p. 110 and \cite{HM} theorem 8.45) $\mathbb A$ is the projective limit of finite dimensional tori. In other words there are tori $\mathbb T^{n_j}$ and group morphisms $f_{j,i} : T^{n_j} \longrightarrow T^{n_i}$ for $i <j$ integers, such that $f_{k, j}\circ f_{j, i}= f_{k,i} $ and a map $f_{\infty,i}: \mathbb A \longrightarrow T^{n_i}$ such that $\mathbb A= \varprojlim_j T^{n_j}$. We denote by $\mathbb A_j$ the image of $\mathbb A$ in $T^{n_j}$, which is clearly a compact subgroup of $T^{n_j}$ hence a  subtorus, and we may replace $T^{n_j}$ by $\mathbb A_j$. Denoting by $p_j=f_{j+1,j}$ and $\pi_j=f_{\infty, j}$ we have a sequence
 $$
 \xymatrix{
 ... \ar[r]^{p_{j+2}}&\mathbb A_{j+1} \ar[r]^{p_{j+1}}&  \mathbb A_j \ar[r]^{p_{j}}&  \mathbb A_{j-1} \ar[r]^{ p_{j-1}} & ....\\
 & \mathbb A \ar[u]^{\pi_{j+1}}\ar[ur]^{\pi_j}\ar[urr]_{\pi_{j-1}}&&&
 }
 $$
 We set $\mathbb K_j= \Ker (\pi_j)$. We then have
  \begin{lem} 
  We have $$ \lim_j \diam (\mathbb K_j)=0$$
  \end{lem} 
  \begin{proof} 
  The $\mathbb K_j$ are a decreasing sequence of closed - hence compact- subgroups such that $\bigcap_j \mathbb K_j= \{0\}$ by definition of the projective limit. But this implies the lemma by an easy exercise (or \cite{Rudin} theorem 3.10).
  \end{proof} 
  Now let $(X,\delta)$ be a metric space. We need the 
  \begin{lem} 
  Let us consider the embeddings $$\begin{aligned} \pi_j^*:& C^0( \mathbb A_j, X) \longrightarrow C^0( \mathbb A, X) \\ &f \mapsto f\circ \pi_j \\ \end{aligned} $$
  Then the union of the images of the $\pi_j^*$ is dense in $C^0( \mathbb A, X)$. 
  \end{lem} 
  \begin{proof} 
  Let $f \in C^0(\mathbb A , X)$. Then $f$ is uniformly continuous by Heine's theorem (see \cite{Rudin}, theorem  4.19) :
  $$ \forall \varepsilon >0,\; \exists \eta>0\; \forall x,y \in \mathbb A , \;\; d(x,y) < \eta \Longrightarrow \delta(f(x),f(y))< \varepsilon $$
  For $j$ large enough $\diam (\mathbb K_j) < \eta$ so setting $f_j(x)= \min \{ f(x+u) \mid u\in \mathbb K_j\}$, we see that by compactness of $\mathbb K_j$ the function $f_j$ is continuous and $d(f(x), f_j(x)) < \varepsilon $ provided $\dim (\mathbb K_j) < \eta$. 
   \end{proof} 
  Now remember that we have a group morphism $\tau ({\mathbb R}^n) \longrightarrow \mathbb A$. By definition of a projective limit, the map $\tau$ is defined by a sequence of maps $\tau_j: {\mathbb R}^n \longrightarrow \mathbb A_j$ such that $p_j\circ \tau_j=\tau_{j-1}$. Of course the density of $\tau ( {\mathbb R}^n)$ implies the density of $\tau_j ({\mathbb R}^n)$ because the preimage  by $\pi_j$ of a proper closed subset is a proper closed subset (remember $\pi_j$ is onto by assumption). 
  It will be useful to state
  \begin{lem} \label{Lemma-7.3}
  Let $u: {\mathbb R}^n \longrightarrow T^d$ be a group morphism, and $\tilde u : {\mathbb R}^n \longrightarrow {\mathbb R}^d$ be its lift. Consider the vector space in $ {\mathbb R}^d$ defined by  $V=\tilde u ( {\mathbb R}^n)$. Then $\Image (u)$ is dense if and only if $$V^\perp\cap \mathbb Z^d=\{0\}$$
  \end{lem} 
  \begin{proof} 
  See  Appendix \nameref{Section-Appendix3}.
  \end{proof} 
  
  Let us now return to our setting and notice that 
  the situation is   the following : we have a subgroup $\mathbb A_\Omega$ in $\Iso ( \widehat{\mathfrak H}_\omega, \gamma)$ and for almost every $H$ (for the measure $\widehat\mu_\Omega$) we have $\mathbb A_\Omega \cdot H=\widehat{\mathfrak H}_\omega $.  Now $\mathbb A_\Omega \cdot H$ is approximated by 
  $\mathbb A_j \cdot H$ for a finite dimensional torus $\mathbb A_j$, and the action of $ {\mathbb R} ^n$ by $\tau$ yields a dense subgroup of $\mathbb A_j$. 
  At the cost of an approximation, we have thus replaced $H_\omega$  for $\omega \in \mathbb A_\omega$ by the $H_\omega$ for $\omega\in \mathbb A_j$, that is we have a continuous map $ \mathbb A_j \longrightarrow (\widehat{\mathfrak {Ham}}_{fc}, \gamma)$ and $\mathbb A_j$ is a finite dimensional torus. 
  
 \section{Regularization of the Hamiltonians in $\widehat{\mathfrak {Ham}}_{fc}$}\label{Section-8}
 Let $H \in \widehat{\HH}_{FP}(T^* {\mathbb R}^n)$ and its flow $\varphi_H^t $ in $\widehat{\DHam}_{FP}(T^* {\mathbb R}^n)$.  
 We set $c(1_{(q,p)}, S)$ to be the critical value corresponding to the unique cohomology class in $S_{q,p}(\xi)=S(q,p;\xi)$. If $L$ is a Lagrangian, we set $c(1_{(q,p)},L)=c(1_{(q,p)}, S)$ where $S$ is a \ac{GFQI} for $L$, and
 $c(1_{(q,p)}, \varphi)= c(1_{(q,p)}, \Gamma (\varphi))$. We now set
 \begin{defn} 
 For $\eta>0$ we set $$H^\eta(q,p)= \frac{1}{\eta} c(1_{(q,p)}, \varphi_H^\eta)= \frac{1}{\eta}c(1_{(q,p)}, \Gamma(\varphi_H^\eta))$$
 This defines a map $$\sigma_\eta : \widehat{\HH}_{fc}(T^* {\mathbb R}^n) \longrightarrow C^{0,1}_{fc}(T^* {\mathbb R}^n)$$ where $C^{0,1}_{fc}(T^* {\mathbb R}^n)$ is the set of Lipschitz functions with  fiberwise compact support. 
 \end{defn} 
  Our goal is to prove that $\sigma_\eta$ is a regularizing operator. This is the content of
  \begin{prop} \label{Prop-8.2}
  We have for $H\in \widehat{\HH}_{fc}(T^* {\mathbb R}^n)$
  \begin{enumerate} 
  \item \label{Item-8.2-1}
  $\gamma_c-\lim_{\eta\to 0} \sigma_\eta(H)=H$
  \item For $H$ supported in $ {\mathbb R}^n\times B(R)$ and such that $\varphi_H (T^*B(\rho)) \subset T^*B(\rho+r)$ then $\sigma_\eta(H)$ is $ \frac{C(R+r)}{\eta}$-Lipschitz. 
  \item $ \sigma_\eta:  \widehat{\HH}_{fc}(T^* {\mathbb R}^n) \longrightarrow C^0_{fc}(T^* {\mathbb R}^n, {\mathbb R} )$ is continuous for the $\gamma$-topology. 
  \item $\sigma_\eta \circ \tau_a=\tau\circ \sigma_\eta$
  \end{enumerate} 
  \end{prop} 
    \begin{rem} 
  One should be careful. The $\gamma_c$-limit in (\ref{Item-8.2-1}) is not a $C^0$ limit, since $H$ is not continuous in general. But even if $H$ is continuous, we do not claim this.  
  \end{rem}
  \begin{proof} 
  \begin{enumerate} 
  \item By density we can find $K\in C^\infty_{fc}(T^*{\mathbb R}^n, {\mathbb R} )$  such that $\gamma(H,K)\leq \varepsilon $. Now for $K\in  C^\infty_{fc}(T^*{\mathbb R}^n, {\mathbb R} )$ we may find a \ac{GFQI}, $S_{K, \eta}$ of $\varphi_K^\eta$ such that
  $$S_{K, \eta}(q,p)=\eta\cdot K(q,p) + o( \eta)$$ as $\eta$ goes to zero so that $K^\eta(q,p)= \frac{1}{\eta} c(1_{(q,p)}, S_{K,\eta})=K(q,p)+o(1)$. We need the following Lemma, that we shall prove in Appendix \nameref{Section-Appendix4}. 
  \begin{lem} \label{Lemma-appendix2}
  We can find a  \ac{GFQI} for $\varphi_K^\eta$, $S_{K,\eta}$ such that
  $$\Vert S_{K,\eta}(q,p)-\eta K(q,p) \Vert \leq C\eta^2  \Vert \nabla K \Vert_{C^0}^2$$
  \end{lem}

  Now the formula $c(1_{(q,p)}, S_{K,\eta})=\eta K(q,p)+o(\eta)$ follows immediately by applying on one hand the triangle inequality (see \cite{Viterbo-STAGGF}, prop 3.3 p.693)
  \begin{displaymath} 
 \vert  c(1_x, L) - c(1_x, L') \vert \leq \gamma (L, L')
   \end{displaymath} 
  and  on the other hand Proposition \ref{Prop-5.16}, 
  \begin{displaymath}  \Vert K^\eta(q,p)-K(q,p) \Vert \leq \eta \cdot \Vert \nabla K \Vert_{C^0}^2 \dispdot
  \end{displaymath} 
  Now for $\eta$ small enough we have $\gamma (K^\eta, K) \leq \varepsilon $. Remember that for $H, K \in \widehat {\HH} (T^* {\mathbb R}^n)$ that $H \preceq K$ means $c(1_W, \varphi_K, \varphi_H) =0$ for all $W$. The reduction inequality (\cite{Viterbo-STAGGF}, p. 705,  proposition 5.1) implies that $H^\eta (q,p) \leq K^\eta(q,p)$ for all $(q,p)\in T^*{\mathbb R}^n$. 
  
  Now $\gamma (H,K) \leq \varepsilon $ implies that  $K- \varepsilon \chi_R \preceq H\preceq  K + \varepsilon \chi_R$ for $R$ large enough : this follows from the formula 
  $c(1_W, \varphi_{K+ \varepsilon \chi_R}, \varphi_H) = c(1_W, \varphi_K, \varphi_H) + \varepsilon $ for $W$ large 
  enough\footnote{Because if $S$ is a \ac{GFQI} for $\varphi_K$ then $S_ \varepsilon (q,p;\xi)=S_0(q,p;\xi)+ \varepsilon \chi_R(p)$ 
  is a \ac{GFQI} for $\varphi_{K+ \varepsilon \chi_R}$ and $c(1_W, S_ \varepsilon )=c(1_W, S_0)+ \varepsilon $ for $R$ and $W$ large enough.}. 
  Now we have
  $K^\eta- \varepsilon \chi_R \preceq H^\eta \preceq K^\eta + \varepsilon \chi_R$
  and for $\eta$ small enough we get $  \Vert K - K^\eta \Vert \leq \varepsilon $ son
  $$K- 2 \varepsilon \preceq H^\eta \preceq K + 2 \varepsilon $$
  hence 
  $$H- 3 \varepsilon \preceq K- 2 \varepsilon \preceq H^\eta \preceq K+ 2 \varepsilon \preceq H+ 3 \varepsilon $$
  hence $\gamma (H^\eta, H)\leq 3 \varepsilon $. 
   \item  We have for $\vert q_1-q_2\vert + \vert p_1-p_2\vert \leq r$  \begin{displaymath} c(1_{(q_1,p_1)}\varphi^{\eta}_H) - c(1_{(q_2,p_2)}\varphi^{\eta}_H) \leq C ( r )\end{displaymath}  
   because for $L_{(q,p)}$ Hamiltonianly isotopic to the vertical and coinciding with $T^*_{(q,p)}\Delta_{ {\mathbb R}^{2n}}$ in $\Delta_{{\mathbb R}^{2n}}\times B_r^{2n}$  we have 
   $$c(1_{(q,p)},\Gamma(\varphi_H^\eta))=c(\Gamma(\varphi_H^\eta), L_{(q,p)})$$ and 
   \begin{displaymath} 
  \vert  c(\Gamma(\varphi_H^\eta), L)-c(\Gamma(\varphi_H^\eta), \psi(L)) \vert \leq \gamma (L, \psi(L)) \leq \gamma (\psi)
   \end{displaymath} 
   
As a result,   provided there is a Hamiltonian map $\psi$ with $\gamma(\psi) \leq C(r)$ such that 
 \begin{displaymath} 
 \psi( T_{(q_1,p_1)}^*\Delta_{{\mathbb R}^{2n}})\cap (\Delta_{{\mathbb R}^{2n}}\times B_\rho^{2n})=T_{(q_2,p_2)}^*\Delta_{{\mathbb R}^{2n}}\cap (\Delta_{{\mathbb R}^{2n}}\times B_\rho^{2n})
  \end{displaymath} 
 where $\rho$ is such that  $\Gamma(\varphi_H^\eta) \subset {\mathbb R}^{2n}\times B^{2n}_{\rho}$. Since we assumed that $H$ is supported in $B_{R}$ we may assume $\rho=2R$ 
 and we have $C(r)=CR\cdot r$. Indeed
 if $\psi_t$ is an isotopy such that $\psi_1$ sends $(q_1,p_1)$ to $(q_2,p_2)$, and $\Psi_{t}$ its natural extension to a Hamiltonian isotopy $T^*( \Delta_{T^*\mathbb R})$ 
 we truncate the Hamiltonian generating $\Psi_t$  to $ {\mathbb R}^{2n}\times B^{2n}_\rho$ where $\rho$ is an upper bound for  $ \vert Q_H(q,p)-q \vert + \vert P_H(q,p)-p \vert$. Such an upper bound is given by $r+2R$ ($r$ for $ \vert Q-q \vert $ and $2R$ for $ \vert P-p \vert $). This proves the inequality\footnote{We also can take $R \simeq \eta \Vert H \Vert _{C^{0,1}}$, and then $C(r) \simeq C r \eta  \Vert H \Vert _{C^{0,1}}$ but this requires $H$ to be Lipschitz. But this proves that the maps $\sigma_\eta$ does increase the Lipschitz norm by a bounded multiplicative constant only.}. 
 
  \item We have   \begin{gather*} \Vert\sigma_ \eta (H)- \sigma_ \eta (K)\Vert_{C^0} \leq \frac{1}{\eta} \sup_{(q,p)} c(1_{(q,p)},\varphi^\eta (\psi^\eta)^{-1}) \leq \\  \frac{1}{ \eta } \gamma (\varphi_H^\eta , \varphi_K^\eta ) \leq \frac{1}{\eta} \gamma (H,K) 
  \end{gather*} 
  where the first inequality is just the triangle inequality (see \cite{Viterbo-STAGGF}, prop 3.3 p.693) and the second inequality follows by the reduction inequality (\cite{Viterbo-STAGGF}, p. 705,  proposition 5.1). 
  
 \item  We have  $\sigma_ \eta (H_\omega)(x+a,p)= \frac{1}{ \eta }c(1_{x+a,p},\varphi^\eta _{H_\omega})= c(1, S^\omega(x+a,P;\xi))$ but 
 $S^\omega(x+a,P;\xi)$ is the generating function corresponding to $\tau_aH_\omega$ i.e. $\Gamma(\tau_{-a}\varphi_{H_\omega}^\eta  \tau_a)$ is the set of 
 $(q+a,P,P-p,Q-q)$ where $\varphi_{H_\omega}^\eta (q,p)=(Q,P)$. So we have $\Gamma(\tau_{-a}\varphi_{H_\omega}^\eta  \tau_a)= \tau_a (\Gamma(\varphi_{H_\omega}^\eta ))$ and 
  \begin{displaymath} S_{H_{\tau_{-a}\omega}}(x,P, \tau_{-a}\xi)=S_{\tau_{a}H_\omega}(x,P;\xi)=S_{H_\omega}(x+a,P;\xi) \end{displaymath} 
 We thus proved that 
 \begin{gather*}  \tau_a \sigma_\eta (H_\omega)(x,p)=\sigma_ \eta (H_\omega)(x+a,p)=\\ \sigma_ \eta (\tau_aH_\omega)(x,p) =\sigma_ \eta (H_{\tau_{-a}\omega})(x,p)=\sigma_\eta (\tau_aH_\omega)(x,p) \end{gather*} 

  \end{enumerate} 
  \end{proof} 
  
  We are now in the following  situation : we started from a continuous map $$ H : \mathbb A_j \longrightarrow (\widehat \HH (T^* {\mathbb R}^n), \gamma)$$ and have constructed a map
  $$H^\eta : \mathbb A_j \longrightarrow (C^0_{fc}(T^* {\mathbb R} ^n), d_{C_0})$$ which is continuous , satisfies $\tau _a H^\eta=H^\eta$. Note that we may replace if needed $C^0_{fc}$ by $C^k_{fc}$ by applying convolution since $ \tau_a (H\star \chi) =(\tau_a H)\star \chi = H\star \chi$ (and of course, since $ \Vert H\star \chi-H \Vert \longrightarrow 0$ as $\chi \longrightarrow \delta_0$, we also have $\gamma_c$-convergence). 
  
  Let us summarize our findings combining the results of Proposition \ref{Prop-8.2} and the conclusions of Sections \ref{Section-Compactness-ergodicity} and  \ref{Section-Abelian} the following
  \begin{cor} \label{Cor-8.4}
  Let $H: T^* {\mathbb R}^n \times \Omega \longrightarrow {\mathbb R} $ satisfy assumption (\ref{C1})-(\ref{C6}) of the Main Theorem. Given $ \varepsilon >0$ there exists $d\in {\mathbb N} $ and $H^\varepsilon : T^* {\mathbb R}^n \times T^d  \longrightarrow {\mathbb R}  $ such that 
  \begin{enumerate} 
  \item $\omega \mapsto H^\varepsilon_\omega$ is continuous from $T^d$ to $C^\infty_{fc}(T^* {\mathbb R}^n, {\mathbb R})$ 
  \item $\gamma (H_\omega, H_{\pi_d(\omega)}^\varepsilon )\leq \varepsilon $ for all $\omega \in \Omega$.
  \item The Hamiltonians $H^\varepsilon_\omega, H^\varepsilon_{\pi_d(\omega)}$ satisfy assumptions (\ref{C1})-(\ref{C6})
  \end{enumerate} 
  \end{cor} 
  \begin{proof} 
  From Section \ref{Section-Compactness-ergodicity} we get $H$ from $\mathbb A_\Omega$ to ${\widehat \HH}_c(T^* {\mathbb R}^n)$. From Section \ref{Section-Abelian} we can approximate $H$ by a map from  $T^d$  to ${\widehat \HH}_c(T^* {\mathbb R}^n)$ and from the present Section, we have an approximating map to $C^\infty_{fc}(T^*{\mathbb R}^n, {\mathbb R})$. 
  \end{proof} 
\section{Homogenization in the almost periodic case}\label{Section-9}
We assume in this section that we have a a map $(q,p;\omega) \longrightarrow H(q,p;\omega)=H_\omega(q,p)$ such that 
\begin{enumerate} 
\item $\omega \in \Omega=T^d$
\item The map $\omega \mapsto H_\omega$ is continuous for the $C_{fc}^\infty$ topology
In  particular the $H_\omega$ have  uniformly fiberwise compact support
and  the $H_\omega$ are uniformly BPS by Proposition \ref{Prop-4.3}. 
\end{enumerate} 

We set $\varphi^t_\omega$ to be the time $t$ flow for $H_\omega$ and $\varphi_{ \varepsilon , \omega}= \rho_ \varepsilon \varphi_\omega^{ \frac{1}{ \varepsilon }} \rho_ \varepsilon ^{-1}$. By compactness of $\Omega$ we also have a map $\omega \mapsto S_\omega (q,p;\xi)$ of \ac{GFQI} for $\varphi_\omega = \varphi_\omega^1$, with $\xi$ living in a vector space independent from $\omega$ : indeed its dimension is bounded by $2nN$ such that $\varphi_\omega^{1/N}$ is in a given neighborhood of $\Id$ for all $\omega \in \Omega$.  
 \begin{defn} 
 We set $$h_{k,U}^\omega (p)=\lim_{p\in V} c(\mu_{U\times V}, \varphi_{k,\omega})$$ and 
 $$h_{k}^\omega=\lim_{U\in {\mathbb R}^n} h_{k,U}^\omega $$
 \end{defn} 
 Our first result is
 \begin{prop} The sequence $h_k^\omega$ is equicontinuous and equibounded. All its converging subsequences have the same limit 
 $h_\omega(p)$ which is in fact independent from $\omega$ and denoted by $\overline H(p)$. 
 We denote by $\overline \varphi^t$ the flow of $\overline H$ in $\widehat\DHam_{fc}(T^* {\mathbb R}^n)$ which belongs to $\widehat\DHam_{FP}(T^* {\mathbb R}^n)$ . 
 
 \end{prop} 
 \begin{proof} For typographical reasons, the  $\omega$ parameter will be omitted in the notation, but of course everything depends on $\omega \in \Omega$. 
Set $\varphi_k(q,p)=(Q_k(q,p),P_k(q,p))$ and $Q=Q_1, P=P_1$. Then we have $ \frac{\partial S_k}{\partial \xi}(q,P;\xi)=0$ and  $\frac{\partial S_k}{\partial p}(q,P;\xi)=Q_k(q,P)-q$. 
But $$h_{k,U}(p)=S_k(q(p),p;\xi(p))$$ where  $ \frac{\partial S_k}{\partial \xi}(q,P;\xi)=0$ and either  $ \frac{\partial S_k}{\partial q}(q,P;\xi)=0$ if $q\in U$ or $ \frac{\partial S_k}{\partial q}(q,P;\xi)= \lambda \cdot \nu_U(q)$ if $q\in \partial U$ and $\nu_U(q)$ is the exterior normal. Now as $p$ varies, we  can choose $p \mapsto (q(p),\xi(p))$ to be piecewise smooth, so that 
\begin{gather*}  dh_k(p) = \\ \frac{\partial S_k}{\partial p}(q,P;\xi)(q,p;\xi)+  \frac{\partial S_k}{\partial q}(q,P;\xi)(q,p;\xi)\cdot \frac{\partial q}{\partial p} + \frac{\partial S_k}{\partial \xi}(q,P;\xi)(q,p;\xi)\cdot \frac{\partial \xi}{\partial p} \end{gather*} 
but $ \frac{\partial S_k}{\partial \xi}(q,P;\xi)(q,p;\xi)=0$ and then either $q\in U$ and then $\frac{\partial S_k}{\partial q}(q,P;\xi)=0$ or $q\in \partial U$ and then $\frac{\partial q}{\partial p}\in T(\partial U)$, so that the term $\frac{\partial S_k}{\partial q}(q,P;\xi)(q,p;\xi)\cdot \frac{\partial q}{\partial p} $ also vanishes. In the end $dh_k(p) =  \frac{\partial S_k}{\partial p}(q,P;\xi)(q,p;\xi)=Q_k(q,p)-q$. 
 But finite propagation speed implies that $Q_k(q,p)-q= \frac{1}{k}(Q(kq,p)-kq)$ is bounded, so $ \vert dh_{k,U}(p) \vert $ is uniformly bounded, independently from $U, k$. From this we conclude that the sequence $h_k$ is equicontinuous. 
Equiboundedness follows from Definition 5.8 in \cite{SHT}, which states that a \ac{GFQI} $S_k$ of $\varphi_k$ satisfies $ \vert S_k -B_k \vert \leq C$ where $C$ is a bound for a generating function $S(q,p)$ for $\varphi^1$ and $B_k$ is a quadratic form. This implies that $ \vert h_{k}(p) \vert \leq C$ and since all this estimates are uniform in $\omega$, this implies (uniform) equiboundedness
 We may thus apply Ascoli-Arzela's theorem, and conclude that  $h_{k}^\omega$   has a converging subsequence. 
Proving that the limit is unique follows as in \cite{SHT}, lemma 5.11. 

Finally we prove that $h_\omega(p)$ is independent from $\omega$, using the commutation of $\tau_a$ and $\rho_k$. 
We have 
\begin{gather*} 
h_{k,\tau_a\omega}(p)= \lim_{U\subset {\mathbb R} ^n} c(\mu_U\otimes 1(p), \Gamma (\varphi_{k, \tau_a\omega}))=\\
 \lim_{U\subset {\mathbb R} ^n} c(\mu_U\otimes 1(p), \Gamma (\tau_a^{-1}\varphi_{k, \omega}\tau_a)) =\\
  \lim_{U\subset {\mathbb R} ^n} c(\mu_{\tau_aU}\otimes 1(p), \Gamma (\varphi_{k, \omega}))=h_{k,\omega}(p)
 \end{gather*} 
 Since $\omega \mapsto \varphi_{k,\omega}$ is $\gamma$-continuous, we infer that $\omega\mapsto h_{k,\omega}(p)$ is continuous and we just proved that it is $\tau$-invariant. Ergodicity then implies that it is constant in $\omega$. 
\end{proof} 
 
 We define 
 $$\widehat{\HH}_{fc,BP}(T^* {\mathbb R}^n)= 
 \widehat{\HH}_{BP}(T^* {\mathbb R}^n) \cap \widehat{\HH}_{fc}(T^* {\mathbb R}^n)$$

 The next Proposition is the analog of  Proposition 5.15 in \cite{SHT}
 
 \begin{prop} \label{Prop-9.3}
 Let $\alpha \in \widehat\DHam_{fc,BP} (T^* {\mathbb R}^n)$. There exists a sequence $k_\nu$ such that 
\begin{displaymath} \lim_{\nu\to +\infty} \lim_{U\subset {\mathbb R}^n} c(\mu_U, \varphi_{k_\nu, \omega}\alpha) \leq \lim_{U\subset {\mathbb R}^n} c(\mu_U, \overline \varphi \alpha) \end{displaymath} 
 \end{prop} 
 \begin{proof} 
 The proof is identical to the proof of Proposition 5.15 in section 6 of \cite{SHT}. 
 \end{proof} 
 
 The next proposition is the analogue of Proposition 7.2 in \cite{SHT}, but requires an adaptation. 
 
 \begin{prop} \label{Prop-9.4}
 For each $ \varepsilon >0$ there exists $K$ such that for all $k\geq K$ and $U$ large enough, we have
\begin{displaymath} c(\mu_U\otimes 1(p), \varphi_{k,\omega})\leq c(1_U\otimes 1(p), \varphi_{k,\omega})+ \varepsilon \end{displaymath} 
 \end{prop} 
 This implies
 \begin{cor} \label{Cor-9.5}
 We have $\overline {\varphi^{-1}}= ( \overline \varphi )^{-1}$, or equivalently $\overline H_{\varphi^{-1}}=-\overline H_\varphi$. 
  \end{cor} 
 Now putting together Proposition \ref{Prop-9.3} and Corollary \ref{Cor-9.5} we get
 
 \begin{prop}\label{Prop-9.6}
  For almost all $\omega \in \Omega$ the sequence $\varphi_{k,\omega}$ converges to $\overline \varphi$. 
 \end{prop} 
 \begin{proof} [Proof assuming Corollary \ref{Cor-9.5} and Proposition \ref{Prop-9.3}]
 Let us  prove the above Proposition is a consequence of Corollary \ref{Cor-9.5} and Proposition \ref{Prop-9.3}. 
 Indeed the Proposition implies $$\lim_{k\to +\infty}\lim_U c(\mu_U, \varphi_{k,\omega}{\overline \varphi}^{-1})\leq \lim_U c(\mu_U, \Id)=0$$
 Applying the same inequality for $\varphi^{-1}$ instead of $\varphi$ and using the Corollary, we get 
\begin{displaymath} \lim_{k\to +\infty}\lim_U c(\mu_U, \varphi_{k,\omega}^{-1}{\overline \varphi})\leq \lim_U c(\mu_U, \Id)=0 \end{displaymath} 
 and this implies \begin{displaymath} \lim_{k\to +\infty}\lim_U \gamma (\mu_U, \varphi_{k,\omega}^{-1}{\overline \varphi})\end{displaymath}  which proves our claim. 
 \end{proof} 
 \begin{proof} [Proof of Corollary \ref{Cor-9.5} assuming Proposition \ref{Prop-9.4}]

 Set 
\begin{gather*} 
 h_{k,\omega}^+(\varphi;p)=\lim_{U\subset {\mathbb R}^n} c(\mu_U\otimes 1(p), \varphi_{k,\omega}) \\
 h_{k,\omega}^-(\varphi;p)=\lim_{U\subset {\mathbb R}^n} c(1_U\otimes 1(p), \varphi_{k,\omega}) 
 \end{gather*} 

so that $ h_{k,\omega}^-(\varphi;p) \leq  h_{k,\omega}^+(\varphi;p)$. Set $\sigma_{p_0}(q,p)=(q,p+p_0)$ then if $S(q,p;\xi)$ is a \ac{GFQI} for $\varphi$, then $S_p(x;\xi)=S(x,p;\xi)$ is a \ac{GFQI} for $\sigma_p(0_{ {\mathbb R}^n})-\varphi (\sigma_p(0_{ {\mathbb R}^n}))$. If we assume $\varphi$ has FPS we have from Proposition \ref{Prop-5.12} 
 \begin{displaymath} c(\mu_U, \sigma_p(0_{ {\mathbb R}^n}-\varphi (\sigma_p(0_{ {\mathbb R}^n})))) \leq 
c(\mu_V, \sigma_{-p}\varphi^{-1} (\sigma_p(0_{ {\mathbb R}^n})))
\end{displaymath} 
 for $V$ such that $\varphi(T^*U)\subset T^*V$. Taking the limit for $U \subset {\mathbb R}^n$ we get 
\begin{displaymath} \lim_{U\subset {\mathbb R}^n} c(\mu_U, S_p) = \lim_{U \subset {\mathbb R}^n} c(\mu_U, \sigma_{-p}\varphi^{-1}\sigma_p(0_ { {\mathbb R}^n}))) \end{displaymath} 
 and the same holds for $1_U$ instead of $\mu_U$. 
 Now we may write (omitting the $\omega$) using first Proposition \ref{Prop-5.8} (\ref{Prop-5.8-1}) and then FPS of $\varphi$
 \begin{gather*} 
 h_k^+(\varphi^{-1};p)=\lim_{U\subset {\mathbb R}^n} c(\mu_U\otimes 1(p), \sigma_{-p}\varphi_k\sigma_p(0_ { {\mathbb R}^n}))=\\
 - \lim_{U\subset {\mathbb R}^n} c(1_U\otimes 1(p), 0_{ {\mathbb R}^n}-\sigma_{-p}\varphi_k\sigma_p(0_ { {\mathbb R}^n}))\leq \\
 -\lim_{V\subset {\mathbb R}^n}  c(1_V\otimes 1(p), \sigma_{-p}\varphi_k^{-1}\sigma_p(0_ { {\mathbb R}^n})) = -h_k^-(\varphi;p)
 \end{gather*} 
 
 As a result 
 \begin{displaymath}\tag{a} h_k^+(\varphi^{-1};p)+ h_k^-(\varphi;p) \leq 0 \end{displaymath} 
 and as  $k$ goes to $+\infty$, Proposition \ref{Prop-9.4} 
\begin{displaymath} h_k^+(\varphi^{-1};p)- h_k^-(\varphi;p) \leq  \varepsilon \end{displaymath}  we get 
 so we get
 \begin{displaymath}   \tag{b} h_k^+(\varphi^{-1};p)+ h_k^+(\varphi;p) \leq  \varepsilon 
 \end{displaymath} 
 On the other hand, we have using again Proposition \ref{Prop-5.8} (\ref{Prop-5.8-1})
 \begin{displaymath} 
 -c(1_U,\sigma_{-p}\varphi_k\sigma_p(0_ {{\mathbb R}^n}))) \leq - c(1_V, 0_{ {\mathbb R}^n},\sigma_{-p}\varphi_k^{-1}\sigma_p(0_ {{\mathbb R}^n})))=
 c(\mu_V, \sigma_{-p}\varphi_k^{-1}\sigma_p(0_{{\mathbb R}^n})
 \end{displaymath} 
 so 
 \begin{displaymath} 
 -h_k^{-}(\varphi ;p) \leq h_k^+(\varphi;p)
 \end{displaymath} 
 so using  \thetag{a} we get 
 \begin{displaymath} \tag{c}
 h_k^+(\varphi;p)+h_k^-(\varphi;p)=0
 \end{displaymath} 
 Using again Proposition \ref{Prop-9.4} we get for $k$ large enough
 \begin{displaymath} \tag{d}
 h_k^-(\varphi^{-1};p) +h_k^-(\varphi;p) \geq - \varepsilon 
 \end{displaymath} 
 Adding \thetag{b} and \thetag{d} we get
 \begin{displaymath} \tag{e}
\left [ h_k^+(\varphi^{-1};p)-h_k^-(\varphi^{-1};p) \right ]+ \left [  h_k^+(\varphi;p)-h_k^-(\varphi ;p)\right ] \leq 2 \varepsilon 
 \end{displaymath} 
 Since $\overline H_{\varphi^{-1}} =\lim_k h_k^+(\varphi^{-1};p)$, inequality \thetag{b} implies
 \begin{displaymath} 
 \overline H_{\varphi^{-1}}+ \overline H_{\varphi} \leq 0 
 \end{displaymath} 
 Using \thetag{d} and \thetag{e} we get 
 \begin{displaymath} 
 \overline H_{\varphi^{-1}}+ \overline H_{\varphi} \geq 0 
 \end{displaymath} 
 so we may conclude
 \begin{displaymath} 
 \overline H_{\varphi^{-1}}+ \overline H_{\varphi} =0
 \end{displaymath} 
 \end{proof} 
 
 We shall interchangeably the notations $S_\omega(q,p;\xi)$ or $S(q,p;\xi;\omega)$ for the \ac{GFQI} of $\varphi_\omega$. 
 We shall make repeated use of the iteration formula (see \cite{SHT}), defining the \ac{GFQI} $S_{k, \omega}$ for $\varphi_{k,\omega}$ in terms of the \ac{GFQI} $S_\omega$ of $\varphi_\omega$
 \begin{defn} [Iteration formula]\label{Def-iteration-formula}
  \begin{displaymath} 
 S_{k, \omega}(x,y;\xi,\zeta;\omega)= \frac{1}{k} \left [ S_\omega (kx, p_1;\xi_1)+ \sum_{j=2}^{k-1} S_\omega(kq_j,p_j;\xi_j) + S_\omega(kq_k,y;\xi_k)\right ] + B_k(x,y;\xi) 
 \end{displaymath} 
 where $\xi=(\xi_1,.., \xi_k), \zeta=(p_1,q_2,..., p_{k-1}, q_k), q_1=x,p_k=y$ and 
 \begin{displaymath} 
 B_k(x,y;\zeta)=\langle p_1, q_2-x\rangle + \sum_{j=2}^{k-1} \langle p_j, q_{j+1}-q_j\rangle + \langle y, x-q_k\rangle
 \end{displaymath} 
 and $F_{k,\omega}=S_{k,\omega}-B_k$. 
 
 The action of $ {\mathbb R}^n$ is given by by 
 \begin{displaymath} \tau_a^{(k)}(x,y;\xi,\zeta;\omega)=(x+ \frac{a}{k},y;  \xi; \tau_{a/k}\zeta; \tau_a\omega)
 \end{displaymath} 
 \end{defn} 
 We now prove
 \begin{lem}  
 Assume $\omega \mapsto \varphi_\omega$ for $\omega \in \Omega=T^d$ to be continuous. Then we may choose $\omega \mapsto S_\omega(q,p;\xi)$ to be continuous and such that 
 \begin{displaymath} S(q+a,p;\tau_a\xi;\tau_a\omega)=S(q,p;\xi;\omega)
 \end{displaymath}  
 \end{lem} 
 \begin{proof} 
 It is enough to prove this assuming $\varphi_\omega$ is $C^1$ small, that is for $\varphi_\omega^{1/k}$ with $k$ large enough using then the iteration formula. But then the graph of $\varphi_\omega$ is the graph of a generating function with no fiber variable, which obviously depends continuously on $\omega$ and satisfies the above formula. 
  \end{proof} 
 Now remember that $\tau_a$ is given on $\Omega=T^d$ by $\tau_a(\omega)=\omega + A\cdot a$ where $A: {\mathbb R}^n \longrightarrow {\mathbb R}^d$ is a linear injective map with dense image in $T^d$. 
 Consider triples $\alpha, \beta, \gamma$ with $\alpha \in H^*(T^d), \beta \in H^*(U)\;\text{or}\; H^*(U, \partial U), \gamma \in H^*(V)\;\text{or}\; H^*(V, \partial V)$. We may then define\footnote{Caveat: the cohomology class $\alpha$ corresponds to the last variable, $\omega$ !} $c(\alpha\otimes \beta\otimes \gamma , S)$, and we have
 \begin{lem} 
 We have the inequalities
 \begin{displaymath} 
 c(\mu_U\otimes 1(p); S_\omega) \leq c(\mu_{T^d}\otimes \mu_U\otimes 1(p); S)
 \end{displaymath} 
  \begin{displaymath} 
 c(1_{T^d}\otimes \mu_U\otimes 1(p); S) \leq  c(1_U\otimes 1(p); S_\omega) 
 \end{displaymath} 
 \end{lem} 
 \begin{proof} 
 This is the reduction inequality (see \cite{Viterbo-STAGGF} prop. 5.1 p; 705). 
 \end{proof}  
 
 We now compare spectral invariants of $S$ with those of $S^0$, where we define $S^0(p;\xi;\omega)=S(0,p;\xi;\omega)$

 \begin{lem} 
 We have
 \begin{displaymath} 
 \lim_{U \subset {\mathbb R}^n} c(\mu_{T^d} \otimes \mu_U\otimes 1(p), ;S)=c(\mu_{T^d}\otimes 1(0)\otimes 1(p), S)=c(\mu_{T^d}\otimes 1(p), S^0)
 \end{displaymath} 
 and
 \begin{displaymath} 
  \lim_{U \subset {\mathbb R}^n} c(1_{T^d} \otimes \mu_U\otimes 1(p), ;S)=c(1_{T^d}\otimes 1(0)\otimes 1(p)=c(1_{T^d}\otimes 1(p), S^0)
 \end{displaymath} 
 \end{lem} 
 
 \begin{rem} 
 This is an extension to \ac{GFQI} of the following obvious identity for continuous functions $f: {\mathbb R}^n\times T^d \longrightarrow {\mathbb R} $ such that 
 $f(x+a,\tau_a\omega)=f(x,\omega)$ : for any $x_0\in {\mathbb R}^n$ we have
 \begin{displaymath} 
 \sup_{(x,\omega) \in {\mathbb R}^n\times T^d} f(x,\omega)=\sup_{\omega \in T^d} f(x_0,\omega)
 \end{displaymath} 
 Moreover if the action of $\tau$ has dense orbits, this is also equal to $\sup_{x \in {\mathbb R}^n} f(x,\omega_0)$ for any $\omega_0\in \Omega$. The analog of this last statement will be our main result. 
 \end{rem} 
 \begin{proof} 
 Clearly if $0\in U$ we have 
 \begin{displaymath} 
 c(\mu_{T^d}\otimes \mu_U\otimes 1(p);S)\geq c(\mu_{T^d}\otimes 1(0)\otimes 1(p);S)
 \end{displaymath} 
 and we need to prove the reverse inequality. 
 Let $C$ be a cycle representing $\mu_{T^d}\otimes 1(p) \in H_*((S_p^0)^c, S_p^0)^{-\infty})$ with $c\leq c(\alpha \otimes 1(0)\otimes 1(p),S)+ \varepsilon $ and set 
 \begin{displaymath} 
 \widetilde C_U=\left \{(x,p,\tau_x\xi; \tau_x\omega) \mid (0,p;\xi;\omega)\in C \right \}
 \end{displaymath} 
 Then $\widetilde C_U \subset S^c_p$ and clearly
 $[\widetilde C_U]=\mu_{T^d}\otimes \mu_U \otimes 1(p)$ thus 
 \begin{displaymath} 
 c(\mu_{T^d}\otimes \mu_U\otimes 1(p), S) \leq S(\widetilde C_U) =S^0(C)
 \end{displaymath} 
 because $S(x,p,\tau_x\xi,\tau_x\omega)=S(0,p;\xi;\omega)$ and $S^0(C) \leq c$. 
 
 This implies 
 \begin{displaymath} 
 c(\mu_{T^D}\otimes \mu_U\otimes 1(p);S) \leq c(\mu_{T^D}\otimes 1(0) \otimes 1(p);S)
 \end{displaymath} 
 This proves the first equality. The second one is the dual of the first one, since $\mu_{T^d}\otimes \mu_U$ is dual to $1_{T^d}\otimes 1(U)$. 
 \end{proof} 
 
 Our Proposition \ref{Prop-9.4} then follows from the following\footnote{The point of replacing $S$ by $S^0$ is to avoid the complications related to the non-compactness of $x\in {\mathbb R}^n$. Our proofs could be adapted to work directly with $S$, but proving that the cycles we construct are in the right homology class is slightly more involved. }
 \begin{prop} \label{Prop-9.12}
 For each $\varepsilon >0$ there exists $K$ such that for $k\geq K$
 \begin{displaymath} 
 c(\mu_{T^d}\otimes 1(p), S_k^0) \leq c(1_{T^d}\otimes 1(p), S_k^0) + \varepsilon 
 \end{displaymath} 
 \end{prop} 
\begin{proof} 
 The proof will take up the rest of the section. 
 We rewrite the iteration formula
 \begin{displaymath} 
 S_{k, \omega}(x,y;\xi,\zeta;\omega)= \frac{1}{k} \left [ S_\omega (kx, p_1;\xi_1)+ \sum_{j=2}^{k-1} S_\omega(kq_j,p_j;\xi_j) + S_\omega(kq_k,y;\xi_k)\right ] + B_k(x,y;\xi) 
 \end{displaymath} 
 where $\xi=(\xi_1,.., \xi_k), \zeta=(p_1,q_2,..., p_{k-1}, q_k), q_1=x,p_k=y$ and 
 \begin{displaymath} 
 B_k(x,y;\zeta)=\langle p_1, q_2-x\rangle + \sum_{j=2}^{k-1} \langle p_j, q_{j+1}-q_j\rangle + \langle y, x-q_k\rangle
 \end{displaymath} 
 and $F_{k,\omega}=S_{k,\omega}-B_k$. 
 the action of $ {\mathbb R}^n$ is given by by 
 \begin{displaymath} \tau_a^{(k)}(x,y;\xi,\zeta;\omega)=(x+ \frac{a}{k},y;  \xi; \tau_{a/k}\zeta; \tau_a\omega)
 \end{displaymath} 
 and now $S_{k,\omega}$ is $\tau_a^{(k)}$-invariant, i.e.
 \begin{displaymath} 
 S_k({x+ \frac{a}{k}},y; \xi, \tau_{ \frac{a}{k}}\zeta;\tau_a \omega)=  S(x,y;\xi,\zeta;\omega)
 \end{displaymath} 
 Let $a\in {\mathbb R}^n$ such that for some $\nu \in {\mathbb Z}^d$ we have $ \vert A\cdot a - \nu \vert \leq \delta$ (that is $d_{T^d} (\tau_a(0),0)\leq \delta$ where $d_{T^d}$ is the distance on the torus). Then for some constant depending on $H$ and provided $\delta$ is small enough
 
\begin{gather*} 
 \tag{$\star$} \forall t\in [0,1],  \; \forall (q,p;\xi; \omega)\in {\mathbb R}^n\times {\mathbb R}^n \times E\times \Omega,  \; \\ \vert S(kq + ta,p;\xi;\omega) - S(kq,p;\xi;\omega) \vert \leq C
 \end{gather*} 
 and 
\begin{gather*} 
 \tag{$\star\star$} \forall (q,p;\xi; \omega)\in {\mathbb R}^n\times {\mathbb R}^n \times E\times \Omega,  \; \\\vert S(kq+a,p;\xi;\omega)- S(kq,p;\xi;\omega) \vert =\\  \vert S(kq,p;\xi; \tau_{-a}\omega)- S(kq,p;\xi;\omega) \vert \leq \varepsilon 
\end{gather*} 

Indeed the first inequality holds because 
\begin{gather*} 
\vert S(q+a,p;\xi;\omega)-S(q,p;\xi;\omega) \vert = \vert S(q,p;\xi;\tau_{-a}\omega)-S(q,p;\xi;\omega) \vert \leq \\  \sup_{\omega, \omega'} \vert S(q,p;\xi;\omega)-S(q,p;\xi;\omega') \vert 
\end{gather*}  
This follows from the fact that we may assume that we have no $\xi$ variable and then use the iteration formula. In this case we may assume $ \vert S(q,p;\omega)-S(q,p;\omega') \vert  \leq \gamma (\varphi_\omega, \varphi_{\omega'})$. 
The second inequality follows from the fact that $d_{T^d}(\tau_a\omega, \omega)\leq \delta$ and the continuity of $S$. 

Now let $\gamma$ be the path in $ {\mathbb R}^n$ defined by $\gamma(t)=t\cdot a$ for $0 \leq t \leq 1$. 
Set $\widetilde\gamma^{(k)}$ be the path in $({\mathbb R}^n )^k$  defined as the concatenation of the $k$ paths
\begin{equation} 
\begin{aligned}
t\mapsto &(\gamma(t), 0,..., 0), &\;  \text{for}\; t \in [0, \frac{1}{k}]\\
 t\mapsto &(\gamma( \frac{1}{k} ), \gamma(t),..., 0), &\;  \text{for}\; t \in [\frac{1}{k}, \frac{2}{k}]\\
 \vdots  & \vdots & \vdots \\
 t\mapsto & (\gamma( \frac{1}{k}), \gamma( \frac{1}{k} ),...,\gamma( \frac{1}{k} ), \gamma(t)), \; & \text{for}\; t \in [ \frac{k-1}{k}, 1]
\end{aligned}
\end{equation} 
 The path $\widetilde\gamma^{(k)}$ connects $\widetilde\gamma^{(k)}(0)=(0,...,0)$ to $\widetilde\gamma^{(k)}(1)=( \frac{a}{k},   \frac{a}{k}, ...,  \frac{a}{k})$ through the points
$\widetilde\gamma^{(k)}( \frac{1}{k})=( \frac{a}{k},0,....,0), \widetilde\gamma^{(k)}( \frac{2}{k})=( \frac{a}{k}, \frac{a}{k},0,...,0), ...., \widetilde\gamma^{(k)}( \frac{k-1}{k})=( \frac{a}{k}, \frac{a}{k}, ...,  \frac{a}{k}, 0)$. We shall omit the superscript $k$ and set $\widetilde\gamma^{(k)}(t)=\widetilde\gamma(t)=(\gamma_1(t), ..., \gamma_k(t))=(\gamma_1(t), \overline \gamma (t))$. We then set
$\tau_{\overline \gamma (t)}\zeta = \tau_{\overline \gamma (t)}(p_1,q_2,...,p_{k-1},q_k)= (p_1, q_2+\gamma_2(t),..., p_{k-1}, q_k+\gamma_k(t))$ and $\tau_{\widetilde \gamma (t)} (x,y;\zeta)=(x+ \gamma_1(t), y; \tau_{\overline \gamma (t)} \zeta)$. 
Now from \thetag{$\star$} and \thetag{$\star\star$} and the formula
\begin{displaymath} 
F_{k}(x,y;\xi;\zeta;\omega)=  \frac{1}{k} \left [ S_\omega (kx, p_1;\xi_1)+ \sum_{j=2}^{k-1} S_\omega(kq_j,p_j;\xi_j) + S_\omega(kq_k,y;\xi_k)\right ] 
\end{displaymath} 
we infer that on $\left [ \frac{l}{k}, \frac{l+1}{k}\right ]$ for $1\leq l \leq k$

\begin{gather*} 
F_{k}( \tau_{\widetilde \gamma (t)}(x,y;\xi,\zeta;\omega))=\\
F_{k,\omega}(x,y;\xi;\zeta;\omega) + \frac{1}{k}  \Big [ S(kx+a,p_1;\xi,\zeta;\omega)-S(kx, p_1;\xi_1;\omega) +  \\   \sum_{k=2}^{l} \left ( S(kq_j+a,p_j;\xi_j;\omega)-
S(kq_j,p_j;\xi_j;\omega)\right )  +\\  S(kq_{l+1}+a,p_{l+1};\xi_{l+1};\omega)- S(kq_{l+1},p_{l+1};\xi_{l+1};\omega)
\Big ]
\end{gather*} 
so we get

 \begin{equation}\label{Eq-9.2}
\left \vert F_k(\tau_{\widetilde \gamma (t)}(x,y;\xi,\zeta);\omega)- F_k(x,y;\xi,\zeta;\omega) \right \vert \leq  \frac{ \varepsilon l}{k}+ \frac{C}{k} \leq \frac{C}{k} + \varepsilon    
\end{equation} 

We now want to estimate the variation of $B_k$ on $\tau_{\widetilde \gamma (t)}(x,y;\zeta)$ as $t$ varies from $0$ to $1$.  Note that the choice of this path is crucial to our argument : by changing coordinates one at the time, we achieve an increase of $S$ by $ 0 ( \frac{1}{k})$ instead of $0(1)$.
\begin{lem}\label{Lem-9.13} We have 
\begin{displaymath} 
\vert B_k(\tau_{\widetilde \gamma (t)}(x,y; \zeta))-B_k(x,y; \zeta) \vert \leq ( \vert p_{l+2} \vert + \vert p_{l+1} \vert ) \frac{\vert a \vert }{k} 
\end{displaymath} 
\end{lem} 
\begin{proof} 
Indeed, 
\begin{displaymath} 
B_k(x,y;\zeta) = \langle p_1, q_2-x\rangle + \sum_{j=2}^{k-1} \langle p_j, q_{j+1}-q_j\rangle + \langle y , x- q_k \rangle
\end{displaymath} 
and replacing $x,.q_2,..., q_{l}$ by $x+ \frac{a}{k}, q_2 +  \frac{a}{k} , ... , q_l+ \frac{a}{k}$, $q_{l+1}$ by $q_{l+1}+ \frac{ta}{k}$  and leaving $q_{l+2}, ..., q_k$ unchanged, we get 
\begin{displaymath} 
B_k(\tau_{\widetilde \gamma(t)} (x,y;\zeta)) = B_k(x,y;\zeta) + \langle p_{l+1}, t  \frac{a}{k}- \frac{a}{k}\rangle - \langle p_{l+2},  t\frac{a}{k} \rangle
\end{displaymath} 
and this proves the Lemma. 
\end{proof} 
We must then bound the quantity $( \vert p_{l+2} \vert + \vert p_{l+1} \vert ) \frac{\vert a \vert }{k} $ and we shall modify the cycle $C$ representing the class in $H_k(S_k^c, S_k^{-\infty})$ so that the $ \vert p_l \vert $ remain bounded. This follows from the following Lemma, already  proved in \cite{SHT} :

 \begin{lem}[See Lemma 7.5 in \cite{SHT}]\label{Lem-9.14} There exists constants $K,M$ such that, given a cycle $C \subset S_k^c$ representing a class $[C]\in H_*(S_k^c, S_k^{-\infty})$, we have a cycle $\widetilde C \subset S_k^c$ such that $[\widetilde C] = [C]$ in $H_*(S_k^c, S_k^{-\infty})$ and
 \begin{enumerate} 
 \item\label{Lem-9.13-1} $\widetilde C \subset S_k^{-4K}\cup \left (\{ (x,y;\xi,\zeta;\omega) \mid \max_j \vert p_j \vert \leq M\}\bigcap S_k^c\right ) $
 \item \label{Lem-9.13-2} $\widetilde C \cap S_k^{-3K} \subset \{ (x,y;\xi,\zeta;\omega) \mid \zeta \in E_k^-\}$ where $E_k^-$ is the negative eigenspace of $B_k$. 
  \end{enumerate} 
 \end{lem} 
 The Lemma means that we can deform $C$ so that below a certain level of $S_k$ it coincides with  the negative bundle of $B_k$. 
  \begin{proof} 
 This is as in lemma 7.5 of \cite{SHT}. Let $Z$ be the vector field
 \begin{displaymath} 
 \dot q_j= \chi( \vert p_j \vert ) (p_j-p_{j-1})=Z_{q_j}(q,p), \; \; \dot p_j=0=Z_{p_j}(q,p)
 \end{displaymath} 
 where $\chi (r)$ vanishes for $r\leq 1$. Denoting by $\psi^s$ its flow, we have
\begin{gather*} 
  \frac{d}{ds}S_k(\psi^s(q,p)) = dS_k(q,p)\cdot  Z(q,p)= \left\langle \frac{d}{dq}S_k(q,p), Z_{q_j}(q,p) \right\rangle =\\
  - \sum_{j=1}^k \chi( \vert p_j \vert ) \left\langle \frac{d}{dq_j}S_k(q,p), p_j-p_{j-1} \right\rangle = \\
  - \sum_{j=1}^k \chi( \vert p_j \vert ) \vert  p_j-p_{j-1} \vert^2 +  \left\langle \frac{d}{dq}S_k(k\cdot q_j,p_j), p_j-p_{j-1}\right\rangle =  - \sum_{j=1}^k \chi( \vert p_j \vert ) \vert  p_j-p_{j-1} \vert^2 
\end{gather*}  
the last equality because $S$ vanishes on the support of $\chi (\vert p_j \vert)$. 

Now given $y=p_k$, if $\sup_j \vert p_j  \vert \geq M$, we have that $\sum_{j=1}^k \chi( \vert p_j \vert ) \vert  p_j-p_{j-1} \vert^2 $ is bounded from below by some positive quantity $c_k$ (which is $O(1/k)$ but it does not matter). Thus outside the region $ \{ (q,p) \mid \vert p_j\vert \leq M\}$, the vector field $Z$ is a pseudo-gradient vector field for $F_k$. Since $Z$ is complete, its flow $\psi^s$ has the following properties
\begin{enumerate} 
\item It preserves the $p_j$
\item Outside $ \{ (q,p) \mid \vert p_j\vert \leq M\}$, we have $ \frac{d}{ds}S_k(\psi^s(q,p))\leq -c_k$
\end{enumerate} 
As a result if $F_k(q,p)\leq c$ we have $\psi^{ \frac{c+4K}{c_k}}(q,p) \in   \left ( \{(q,p) \mid \vert p_j\vert \leq M\} \cup F_k^{-4K} \right ) \cap F^c_k$. Thus $\widetilde C_1=\psi^{ \frac{c+4K}{c_k}}(C)$ satisfies (\ref{Lem-9.13-1}). 

Now to satisfy (\ref{Lem-9.13-2}), we use a "cut and paste" as in \cite{SHT} lemma 7.5. 
 \end{proof} 

Using Lemma \ref{Lem-9.13},  Lemma \ref{Lem-9.14} and the inequality \ref{Eq-9.2}  we obtain  the following 
 \begin{prop} \label{Prop-9.14}
 Given a class $a$ in $H_*(S_k^c, S_k^{-\infty})$, we can find a cycle $C$ representing $a$ such that
 \begin{displaymath} 
 S_k(\tau_{\widetilde \gamma(t)}( C)) \leq S_k(C)+ \varepsilon + \frac{C}{k}+ \frac{2M \vert a \vert }{k}  
 \end{displaymath} 
 \end{prop} 

 Now let $a \in H_*(T^d)$ be represented by a map $u: C \longrightarrow T^d$ and $b  \in H_1(T^d)$ be represented by a map $v: S^1 \longrightarrow T^d$. Then the Pontryagin product $a\cdot b$ is represented by $u\cdot v: S^1\times C \longrightarrow T^d$ given by $u\cdot v(z, \theta )=u(z)+v(\theta)$.  
 
 To conclude the proof of Proposition \ref{Prop-9.12} (and as a consequence of Proposition \ref{Prop-9.4}) we need the 
 \begin{lem}\label{Lemma-9.15} Let $\nu \in {\mathbb Z}^d$ be such that $ \vert A\cdot a - \nu \vert \leq \delta$, and let $\beta_\nu$ be the class in $H_1(T^d)$ of the loop
 $t \mapsto t\cdot \nu$ ( for $t\in [0,1]$). Then for $k$ large enough, we have
 \begin{displaymath} 
 c(\alpha\cdot \beta_\nu\otimes 1(p), S_k^0) \leq c(\alpha\otimes 1(p), S_k^0) + \varepsilon  \dispdot
 \end{displaymath}  
\end{lem} 
 \begin{proof} 
 Let $C$ be a cycle representing a class in $H_*((S_k^0)^c, (S_k^0)^{-\infty})$. We may assume $C$ satisfies properties (\ref{Lem-9.13-1}) and (\ref{Lem-9.13-2}). We are going to construct a cycle in the class of $\alpha \cdot \beta$ as made of three pieces. 
 First set 
 \begin{displaymath} 
 C_1= \bigcup_{t\in [0,1]} C_1(t)
 \end{displaymath} 
 where 
 \begin{displaymath} 
 C_1(t)=\left \{(0,p;\xi, \tau_{-\gamma_1(t)}\tau_{\overline \gamma(t)}\zeta; \tau_{k\gamma_1(t)}\omega)\mid (0,p;\xi,\zeta;\omega) \in C \right \} \dispdot
 \end{displaymath} 
 According to Proposition \ref{Prop-9.14} since
 \begin{gather*} 
 S_k(0,p;\xi,  \tau_{-\gamma_1(t)}\tau_{\overline \gamma(t)}\zeta; \tau_{k\gamma_1(t)}\omega)=\\
 S_k(\gamma_1(t),p;\xi, \tau_{\overline \gamma(t)}\zeta ; \omega) = S_k( \tau_{\gamma(t)}(0,p;\xi,\zeta); \omega) \leq \\
 S_k(C)+ \varepsilon + \frac{C+2M \vert a \vert }{k} 
 \end{gather*} 
 as a result we have for each $t\in [0,1]$ we have 
 \begin{displaymath} 
 S_k^0(C_1(t)) \leq S_k^0(C)+ \varepsilon + \frac{C+2M \vert a \vert }{k}
 \end{displaymath} 
 hence 
 \begin{displaymath} 
 S_k^0(C_1) \leq S_k^0(C)+ \varepsilon + \frac{C+2M \vert a \vert }{k}
 \end{displaymath} 
 Note that 
 \begin{gather*}  C_1(0)=C\\
 C_1(1)=\left \{  (0,p;\xi,\zeta;\tau_{a} \omega) \mid   (0,p;\xi,\zeta; \omega) \in C \right \}
\end{gather*} 
 Now for $u\in [0,1]$ define the path $\eta(u)=(1-u)A\cdot a + u \nu$ so that $\eta(0)+\omega = \tau_a\omega$. 
 Set 
 \begin{displaymath} 
 C_1(1+u)=\left \{(0,p;\xi,\zeta;\omega+\eta(u)) \mid (0,p;\xi,\zeta;\omega) \in C \right \}
 \end{displaymath} 
 Now $C_1(2)=C$ and since $ \vert \eta(u)-A\cdot a \vert \leq \delta$ we have 
 \begin{displaymath} 
 S_k^0(C_1(1+u) \leq S_k^0(C_1(1)) + \varepsilon 
 \end{displaymath} 
 so that the cycle 
 \begin{displaymath} 
 \widehat C = \bigcup_{s\in [0,2]} C_1(s)
 \end{displaymath} 
 satisfies for $k$ large enough
 \begin{displaymath} 
 S_k^0(\widehat C) \leq S_k^0(C) + 2 \varepsilon 
 \end{displaymath} 
 Moreover we claim that the cycle $\widehat C$ defines a cycle in the homology class of $\alpha \cdot \beta_\nu$. Indeed the lift of the variable $\omega$ starting from $\omega_0$
 is given
 \begin{enumerate} 
 \item  for $s\in [0,1]$ by the path  $s \mapsto \omega_0+sA\cdot a $ 
\item for $s\in [1,2]$ by $s \mapsto \omega_0+(2-s) A\cdot a + (s-1)\nu$
\end{enumerate} 
and since it joins $\omega_0$ to $\omega+\nu$ it belongs to the class $\beta_\nu$. As a result $[\widehat C] = \alpha \cdot \beta \in H_*((S_k^0)^{+\infty}, S_k^0)^{-\infty})$ and  this proves the lemma. 
\end{proof} 
We shall also need
 \begin{lem}\label{Lemma-9.16} Let $ \varepsilon >0$ and $ A : {\mathbb R}^n \longrightarrow {\mathbb R} ^d$ be a linear map such that $A( {\mathbb R}^n)$ has dense projection on $T^d$. Then there is a basis of integral vectors $\nu_1,..., \nu_d$ in $ \mathbb Z^d$ such that there exists vectors $a_1,...,a_d$ in $ {\mathbb R}^n$ such that 
 \begin{displaymath} 
 \vert A\cdot a_j - \nu_j \vert \leq \varepsilon 
    \end{displaymath} 
    \end{lem} 
\begin{proof} 
See Appendix \nameref{Section-Appendix3}.
\end{proof} 
 \begin{proof}[Proof of Propositions \ref{Prop-9.12}  and \ref{Prop-9.4}]
 Let $\alpha_j\in H_1(T^d)$ be the homotopy class of the path $t\mapsto t\cdot \nu_j$ where $\nu_j$ is a basis of $ {\mathbb R}^d$ given by Lemma \ref{Lemma-9.16}. Then $\alpha_1\cdot \alpha_2\cdot ... \cdot \alpha _d=c_d \mu_{T^d}$ for some $c_d\neq 0$. since $c(c_d \mu_{T^d}, f)=c(\mu_{T^d},f)$ we obtain by repeated applications of Lemma \ref{Lemma-9.15} that $c(\mu_{T^d}\otimes 1(p), S_k^0) \leq  c(1_{T^d}\otimes 1(p), S_k^0) + \varepsilon $ and this proves Proposition \ref{Prop-9.12} hence Proposition \ref{Prop-9.4}.
 \end{proof} 
 \end{proof} 
 
 \section{Proof of the Main Theorem}\label{Section-10}
 We first prove that under assumptions (\ref{C1})-(\ref{C6}) we have $\lim_{ \varepsilon \to 0} \varphi_ { \varepsilon , \omega}^t= \overline \varphi^t$ for almost all $\omega \in \Omega$.  
 We start from $H$ satisfying (\ref{C1})-(\ref{C6}), then, using the results of Section \ref{Section-6} we get a map $H: \mathbb A_\Omega \longrightarrow \widehat\HH(T^*T^n)$
 such that $\mathbb A_\Omega$ is a compact connected metric abelian  group. According to Section \ref{Section-7}, $\mathbb A_\Omega$ is the projective limit of finite dimensional tori, 
$ \mathbb  A_j$, on which $\tau_a$ is given by $\tau_a\omega= \omega+ A_j\cdot a$ where the projection of $A_j ({\mathbb R}^n)$ is dense in $\mathbb A_j$ and $\omega \mapsto H(\bullet, \bullet ;\omega)$ is continuous from $\mathbb A_j$ to $C^\infty_{fc}( T^* {\mathbb R} ^n, {\mathbb R} )$ and satisfies (\ref{C1})-(\ref{C6}).

By  Corollary \ref{Cor-8.4}  we find $H^\eta$ in $C^{\infty}(T^* {\mathbb R}^n \times \mathbb A_j , {\mathbb R} )$ such that $$\gamma(H^\eta_{\pi_j(\omega)}, H_\omega) \leq \eta $$ where $\pi_j: \mathbb A_\Omega \longrightarrow \mathbb A_j$ is the projection map. According to  Section \ref{Section-9}, we know that $$\gamma_c-\lim H^\eta_{k,\pi_j(\omega)} = \overline H^\eta_{\pi_j(\omega)}$$ and since for all $k, \omega$, we have $\gamma_c (H^\eta_{k,\pi_j(\omega)}, H_{k,\omega}) \leq \eta$ we infer for $k$ large enough 
 \begin{displaymath} 
 \gamma(\overline H^\eta_{\pi(\omega)}, H_{k,\omega}) \leq 2 \eta
 \end{displaymath} 
 
 Now consider a sequence $\eta_\nu$ converging to $0$ so  that 
 $H^{\eta_\nu}_{\pi_\nu(\omega)}$ is a $\gamma-c$-Cauchy sequence, $\gamma_c$-converging to $H_\omega$ uniformly in $\omega$, 
then $\overline H^{\eta_\nu}_{\pi_\nu(\omega)}= \overline H^{\eta_\nu}$ is also a Cauchy sequence, so converges to some $\overline H \in \widehat\HH(T^*T^n)$. But then $H_{k,\omega}$ converges a.s. in $\omega$ to $\overline H$. 

For the second part of the Main Theorem, we must go from from $\gamma$-convergence of the flow to $\gamma$-convergence of the solution of the corresponding Hamilton-Jacobi equation. In the case of a compact base this is achieved  in \cite{Viterbo-Montreal}, and the extension to a non-compact base was spelled out in \cite{Cardin-Viterbo} p. 266-276. 

For $L \in \mathfrak L(T^*N)$ we define $u_L(x)=c(1_x,L)$. Our first claim is that $\gamma$-convergence for $L$ implies $C^0$-convergence of the $u_L$
uniformly on compact sets. 
\begin{lem}\label{Lemma-10.1} Let $U$ be bounded domain in $N$. 
If $(L_\nu)_{\nu \geq 1}$ is a  Cauchy sequence for $\gamma_U$, then the sequence $u_{L_\nu}$ is a Cauchy sequence for the topology of uniform convergence on $U$. 
As a result if $(L_\nu)_{\nu \geq 1}$ $\gamma$-converges to $L\in {\widehat {\mathfrak L}}(T^*N)$ then 
the sequence $u_{L_\nu}$ converges uniformly on  compact sets to $u_L$. 
\end{lem} 
\begin{proof} 
This is an immediate consequence of the reduction inequality (\cite{Viterbo-STAGGF}, p. 705 , prop 5.1) which implies that for any $x\in U$, 
\begin{displaymath} 
\vert c(1_x,L)- c(1_x, L') \vert \leq \gamma_U(L,L')\dispdot
 \end{displaymath} 

\end{proof} 
Using Proposition \ref{Prop-5.12}, we get
 \begin{prop}
 Let $(\varphi_\nu)_{\nu \geq 1}$ be a sequence in $\widehat{\DHam}_{c,FP}$ $\gamma$-converging to $\varphi_\infty \in \widehat{\DHam}_{c,FP}$. Then for any $L \in \mathfrak L(T^* {\mathbb R}^n)$ (or in ${\widehat {\mathfrak L}}(T^* {\mathbb R}^n)$) the sequence $\varphi_\nu(L)$ $\gamma$-converges to $\varphi_\infty (L)$. 
 \end{prop} 
 \begin{proof} 
 Indeed, we proved that $\gamma_U(\psi_1(L), \psi_2(L)) \leq \gamma_{U\times V} (\psi_1,\psi_2)$ provided $\psi_j^{\pm 1}$ sends $T^*U$ to $T^*V$. 
 In our case, we get 
 \begin{displaymath} 
 \gamma_U(\varphi_\nu(L), \varphi_\infty(L)) \leq \gamma_{U\times V} (\varphi_\nu, \varphi_\infty)
  \end{displaymath} 
  and since the right hand side converges to $0$, so does the left hand side. 
\end{proof} 
We may now conclude our proof. Since $\varphi_{k,\omega}^t$ $\gamma$-converges to $\overline\varphi^t$ and is uniformly FPS for bounded $t$, we thus have
\begin{displaymath} \varphi_{k\omega}^t(L_f) \underset{\gamma}\longrightarrow  \overline{\varphi}^t(L_f))
 \end{displaymath} 
 applying Lemma \ref{Lemma-10.1} to the sequence $(\varphi_{k,\omega}^t)_{k\geq 1}$, this implies uniform convergence  on compact sets of  the sequence $(u_{k,\omega})_{k\geq 1}$   to its limit $\overline u$. This concludes the proof of our Main Theorem. 
 
 \section{The coercive case}\label{Section-11}
 We now assume $H$ satisfies assumptions \ref{C1a} - \ref{C3a} of Corollary \ref{Cor-main}. Let $\chi_A$ be a truncation function, that is an increasing function  such that $0 \leq \chi'_A(t) \leq 1$ and $\chi_A(t)=t-3A/2$ for $t\leq A$ and $\chi_A(t)=0$ for $t\geq 2A$. We set $H_A(x,p;\omega)=\chi_A(H(x,p,\omega)$. Then coercivity implies\footnote{See Remark \ref{Rem-1.5}, since $h_-(p) \leq H(x,p;\omega)\leq h_+(p)$ a.s. in $\omega$ where $\lim_{ \vert p \vert \to +\infty} h_\pm(p)=+\infty$.}
 that $H_A$ has a.s. in $\omega \in \Omega$   the same flow as $H$ in $U_R=\{ (x,p) \mid \vert p \vert \leq r(A)\}$ where $\lim_{A\to +\infty} r(A)=+\infty$. 
 We apply the Main Theorem to $H_A$ and obtain a Homogenized Hamiltonian $\overline{H_A}$. We claim now that for $B\geq A$ we have $\overline H_A=\overline H_B$ on 
 
   \section{The discrete case (Proof of Corollary \ref{Cor-1.6})}\label{Section-12}
  
If we have a $ {\mathbb Z} ^n$ action on $\Omega$, and its standard action on ${\mathbb R}^n$ we construct an $ {\mathbb R} ^n$ action on  $\widetilde \Omega = \Omega \times {\mathbb R}^n / \simeq$ where 
\begin{displaymath} ( \omega, t_1,..,t_n) \simeq (T_{-z}\omega, z_1+t_1,..., z_n+t_n)
\end{displaymath}  
where $z=(z_1,...,z_n) \in {\mathbb Z} ^n$. Then $ {\mathbb R} ^n$ acts on $ \widetilde \Omega$ by translation i.e.
 $\widetilde T_a (\omega, t_1,...,t_n)=(\omega,t_1+a_1,...,t_n+a_n)$. Notice that if $z \in {\mathbb Z} ^n$ we have $\widetilde T_z(\omega,t_1,...,t_n)=(T_z\omega, t_1,..., t_n)$.
 
 Now it is easy to see that $T$ is ergodic if and only if $\widetilde T$ is ergodic, since any $\widetilde T$-invariant set will be of the form $U\times [0,1[^n$ with $U$ a $T$-invariant set. 
 Then  if $H$ satisfies $H(x+z,p,T_z\omega) = H(x,p;\omega)$ we can consider $K(x,p,[\omega,t])= H(x-t,p;\omega)$ and this satisfies
 $K(x+a,p,\widetilde T_a[\omega,t])=K(x,p,[\omega,t])$ for all $a \in {\mathbb R} ^n$, and we can apply the stochastic homogenization from the Main Theorem.  
 
 \section{Extending the Main Theorem }
Note that one should be able to extend our methods to the case where  we have a Hamiltonian satisfying the assumptions of 
the Main Theorem, but
\begin{enumerate} 
\item we have a time dependent Hamiltonian,  $H(t,x,p;\omega)$ and an action of $ {\mathbb R} \times {\mathbb R} ^n$  such that $H(t+s,x+a,p;\tau{(s,a)}\omega )$  and consider the sequence $H( \frac{t}{ \varepsilon }, \frac{x}{ \varepsilon }, \omega)$  This has been reduced to our case in the non-stochastic situation in \cite{SHT} ( Section 12.2 The non -autonomous case)
\item We consider partial homogenization. For example if $X=N\times  {\mathbb R}^k$, then we should be able to apply the above propositions as in \cite{SHT}. 

\item   we  consider  the homogenization $H_ \varepsilon (x, \frac{p}{\varepsilon }; \omega)$ as $ \varepsilon $ goes to $0$. This has been reduced to our case in the non-stochastic situation in \cite{SHT} (Section 13 Homogenization in the $p$ variable). 
\item  we have a $ \mathbb Z^n$ action on a manifold $X$ such that the quotient 
$X/ {\mathbb Z} ^n$ is compact and the Hamiltonian satisfies
 $H(T_zx,T_z^*p,T_z\omega)=H(x,p;\omega)$ where $T$ is the action of $ {\mathbb Z} ^n$ on $X$  and we consider again the sequence $H_ \varepsilon (x, \frac{p}{\varepsilon }, \omega)$ as $ \varepsilon $ goes to $0$.  . 
\end{enumerate} 
The proof in this last case should be the same as the Main Theorem. We just need to replace $\gamma (\varphi)$ (which is not defined on $T^*X)$ with $\widehat \gamma (\varphi)$ and  we shall  get an embedding of $ {\mathbb Z}^n$ into $\Isom (\widehat{\mathfrak H}_\Omega, \gamma)$. According to Weil \cite{Weil2}, the closure of the image of $ {\mathbb Z}^n$ is the product of an abelian compact connected metric group, $A^0_\Omega$, and a totally disconnected  compact metric abelian group $D_\Omega$. 
Since we have a morphism $c: {\mathbb Z}^n \longrightarrow D_\Omega$ and the kernel $L$ must be a cocompact  free abelian group, hence a lattice, so $L$ is isomorphic to  $ {\mathbb Z}^n$ and in suitable integral coordinates, we see that $L= a_1\mathbb Z \oplus a_2 \mathbb Z \oplus ...\oplus a_n\mathbb Z$, so $D_\Omega= \mathbb Z^n/L \simeq \mathbb Z / a_1 \mathbb Z\oplus \mathbb Z / a_1 \mathbb Z \oplus ... \mathbb Z / a_n \mathbb Z$. 
 Replacing $ \mathbb Z^n$ by $L$, we can reduce ourselves to the case of a compact connected abelian group so we get $\overline K(p,\omega ) $ where $\overline K(p,\bullet ) $ is constant on the ergodic components of the action of $L$ and the ergodic component are interchanged by an element of $D_\Omega$, thus we get that $\overline K (p,\bullet )$ is indeed constant a.e.  
 
 It would be also interesting to see what can be done in the framework of more general groups, as explained in \cite{Sorrentino} (see also \cite{Contreras-Iturriaga-Siconolfi}). In this setting a discrete group $G$ is a quotient  of the $\pi_1(M)$ where $M$ is a compact manifold, and we see a Hamiltonian on $M$ as a $G$-invariant one on $\widetilde M$ a cover of $M$. Then Sorrentino considers the Hamiltonian $H(x, \frac{1}{ \varepsilon }p)$ as $ \varepsilon $ goes to zero,  and proves that it converges in some weak sense (we would say in the $\gamma$ topology) to a Hamiltonian defined on $G_\infty$ a graded Lie group associated to $G$ (at least if $G$ is nilpotent). 
  
\section*[1]{Appendix 1: \ac{GFQI} for non-compact Lagrangians: Proof of Theorem \ref{Thm-4.5}} \label{Section-Appendix2}
The goal of this section is to prove Theorem  \ref{Thm-4.5} that is 
\begin{namedthm*}{\bf Theorem \ref{Thm-4.5}}
Let $\varphi$ be an element in $\DHam_{FP}(T^*N)$. Then $\varphi(0_N)$ has a  \ac{GFQI} . Moreover such a \ac{GFQI} is unique. 
\end{namedthm*}
First we claim that the fibration theorem of Th\'eret (see \cite{Theret} theorem 4.2) goes through.  Here $\mathcal F$ is the set of sequence of \ac{GFQI} $(S_\nu)_{\nu\geq 1}$ satisfying the above property and $\mathcal L=\mathfrak L(T^* {\mathbb R}^n)$ and we have
\begin{prop} 
The projection $\pi: \mathcal F \longrightarrow \mathcal L$ is a Serre fibration up to equivalence. 
\end{prop} 
The proof is the same as theorem 4.2 in \cite{Theret}. We may reduce ourselves to the case of a single parameter (as in \cite{Theret}). The proof is then based on Sikorav's existence theorem, which uses only the fact that  for $t$ small enough, if $L$  has a \ac{GFQI} over $U_\nu$ then so does $\varphi^t(L)$. Note that we may always assume that $\varphi^t(T^*U_\nu)\subset T^*U_{\nu+1}$ and by truncating $\varphi^t$ beyond $T^*U_{\nu+1}$, we are reduced to the compact situation. 
  \begin{proof} [Proof of Theorem \ref{Thm-4.5}]
Using Lemma \ref{Lemma-4.2} we may assume we have a sequence $U_\nu$ of domains such that $\varphi^t(T^*U_\nu)\subset T^*U_{\nu+1}$. 
Applying a sequence of cut-offs  to the Hamiltonian defining $\varphi$ we  can then find a sequence $L_\nu$ of Lagrangians of the type $\varphi_\nu^1(0_N)$ where 
\begin{enumerate} 
\item $\varphi_\nu^t(T^*U_\nu)\subset T^*U_{\nu+1}$ for all $t\in [0,1]$
\item  $\varphi_\nu^t$ has compact support in $T^*U_{\nu+1}$ 
\item setting  $\varphi_\nu^t(0_N)=L_\nu(t)$ we have for $\mu \geq \nu$
\begin{displaymath} 
L_\nu(t)\cap T^*U_\nu = L_\mu(t)\cap  T^*U_\nu=\varphi^t(L)\cap T^*U_\nu
 \end{displaymath} 
 \end{enumerate} 
 
 Then each $L_\nu(t)$ has a \ac{GFQI} ,  $S_\nu(t): N\times E_\nu \longrightarrow {\mathbb R} $ and  we claim that for $\mu \geq \nu$,  $S_\nu(t)$ and $S_\mu(t)$ are equivalent over $U_\nu$. Indeed, we have a deformation from  $L_\nu$ to $L_\mu$ that is the identity on $T^*U_\nu$.  If we denote by $S_s$  a \ac{GFQI} covering  this deformation (the existence of which follows from \cite{Theret}, since we are again in the compact supported case),  then $S_s$ generates a  Lagrangian $L_s$ that is constant  over $T^*U_\nu$. Then using lemma 5.3 in \cite{Theret} we can assume that after applying a fiber preserving diffeomorphism that $\Sigma_s\cap (U\times F)=\Sigma_0\cap (U\times F)$ where  $$\Sigma_s=\left \{ (x,\xi) \mid \frac{\partial S_s}{\partial \xi }(x,\xi)=0\right \}$$ 
 But then as in \cite{Theret} p. 259,  using Hadamard's lemma we prove that there is a fiber preserving diffeomorphism such that $S_1(x, \xi(x,\eta))=S_0(x,\eta)$.

So  may now assume that the restriction of $S_\mu$ over $U_\nu$ is exactly $S_\nu\oplus q_{\nu, \mu}$   by composing $S_\mu$ with an extension of the fiber preserving diffeomorphism realizing the equivalence\footnote{The existence of the extension follows from the fact that we may assume that for $\mu, \nu$ large enough, the inclusion $U_\nu \subset U_\mu$ is a homotopy equivalence.}.
\end{proof} 
  \section*[2]{Appendix 2:  Proof of Lemma \ref{Lemma-7.3} and Lemma \ref{Lemma-9.16}}\label{Section-Appendix3}
 Let $V$ be a closed subgroup in $ {\mathbb R}^d$. Our first goal is to prove the following Lemma, of which Lemma \ref{Lemma-7.3} will be an immediate consequence
\begin{lem} \label{Lemma-12.1}
 The projection of $V\subset {\mathbb R}^d$ in $T^d$ is dense if and only if $V^\perp \cap {\mathbb Z}^d=\{0\}$.
\end{lem} 
 \begin{proof} 
 Density in $T^d$ is equivalent to saying that $V+ {\mathbb Z}^d$ is dense in $ {\mathbb R}^d$. Now let $w\in {\mathbb R}^d$ and consider a sequence $v_j+z_j$ with $v_j\in V, z_j \in {\mathbb Z}^d$ such that $v_j+z_j$ converges to $w$. Let $\xi \in {\mathbb Z}^d \cap V^\perp$. Then $\langle \xi, w\rangle = \lim_j \langle \xi, v_j\rangle + \langle \xi, z_j\rangle $. But the first term is zero, while the second is an integer. So this would imply $ \langle \xi, w\rangle \in {\mathbb Z}$ for all $w\in {\mathbb R}^d$, but this is impossible unless $\xi=0$. 
 
 Conversely assume $V^\perp \cap {\mathbb Z}^d=\{0\}$. Note that the closure of $V + {\mathbb Z}^d$ is a closed subgroup of $ {\mathbb R}^d$, hence of the form\footnote{The classical result omits the orthogonality but this is easy to recover : we can always replace $e_1$ by $e_1-w_1$ for some $w_1\in W$ so that $ e_1-w_1\perp W$.},  $W\overset{\perp}{\oplus} ({\mathbb Z}e_1\oplus ... \oplus{\mathbb Z}e_k)$ where $W$ is a real vector space  : this  is a classical result proved by induction on $d$ (see \cite{Bourbaki} theorem 2, p.72). In our case we of course have $V \subset W$. We must prove $k=0$ and $W= {\mathbb R}^d$.  

Set $W'=W \overset{\perp}{\oplus} ({\mathbb R}e_1\oplus ... \oplus{\mathbb R}e_{k-1})$. We can write $W'{\oplus} {\mathbb Z} e_k= W' \overset{\perp}{\oplus}\mathbb Z f_k$ for some vector $f_k$. Then we have  
\begin{enumerate} [label=(\roman*)]
\item \label{i} $V \subset W'$ since $V\subset W \subset W'$ 
\item  \label{ii}$ {\mathbb Z}^d \subset W'\overset{\perp}\oplus {\mathbb Z} f_k$
\end{enumerate} 
 Let $g_k$ be proportional to $f_k$ and such that  $\langle g_k, f_k\rangle=1$. Then $\langle g_k, x\rangle$ is an integer for any $x \in W'\overset{\perp}{\oplus} {\mathbb Z} f_k$, hence, according to \ref{ii},  for all $x \in {\mathbb Z} ^d$.  
Since ${\mathbb Z}^d$ is a self dual lattice this implies $g_k \in {\mathbb Z}^d$. But  $g_k\in W'^\perp $ which according to \ref{i},  is contained in $ V^\perp$ so this implies $g_k\in V^\perp \cap {\mathbb Z}^d=0$, a contradiction. 
 \end{proof} 
 
 We now prove :
 \begin{namedlem*}{\bf Lemma \ref{Lemma-9.16}}\label{Lemma-13.2} Let $ \varepsilon >0$ and $ A : {\mathbb R}^n \longrightarrow {\mathbb R} ^d$ be a linear map such that $A( {\mathbb R}^n)$ has dense projection on $T^d$. Then there is a basis of integral vectors $\nu_1,..., \nu_d$ in $ \mathbb Z^d$ such that there exists vectors $a_1,...,a_d$ in $ {\mathbb R}^n$ such that 
 \begin{displaymath} 
 \vert A\cdot a_j - \nu_j \vert \leq \varepsilon 
    \end{displaymath} 
    \end{namedlem*} 
    \begin{rem} 
    We do not claim the basis is an integral basis, i.e. it does not necessarily have determinant $1$. 
    \end{rem} 
    \begin{proof} 
    Let us consider the vector space generated by integral vectors such that there exist $ a\in {\mathbb R} ^n$ such that $ \vert A\cdot a - \nu \vert \leq \varepsilon $. We must show that the set of such vectors is not contained in a proper vector space $ V_ \varepsilon \subset {\mathbb R}^d$. Indeed, if $V_ \varepsilon$ is such a subspace it is an integral subspace (i.e. generated by integral vectors) hence projects on a subtorus $T_ \varepsilon \subset T^d$. 
    The $V_ \varepsilon $ are nested subspaces and we set $V_0= \bigcap_{ \varepsilon >0} V_ \varepsilon $. Since the  decreasing family $(V_ \varepsilon )_ { \varepsilon >0}$ must be stationary
    we may choose $ \varepsilon_0$ such that $V_ { \varepsilon _0}=V_0$. By density of the image of $A$ in the torus we may find $a_0$ such that $d(A\cdot a_0, T_0)= \varepsilon_1 < \varepsilon _0$ with $ \varepsilon _1>0$ since $T_0 \neq T^d$. We rephrase this as
    \begin{displaymath} 
    \forall v \in V_0, \forall \nu \in \mathbb Z^d, \; \vert A\cdot a_0-v-\nu \vert \geq \varepsilon _1
     \end{displaymath} 
     By Diophantine approximation\footnote{Indeed for any vector $v \in {\mathbb R}^d$, we can find a pair $(k,\nu) \in \mathbb Z \setminus\{0\} \times \mathbb Z^d$. This is equivalent to proving that $d_{T^d}(k [v],0) < \varepsilon $ for some $k\neq 0$. Since the distance is invariant by translation, this is equivalent to finding $k,\neq l \in \mathbb Z$ such that $d_{T^d}(k [v], l [v]) < \varepsilon $ but if this did not hold, we would get infinitely many points on the torus with a distance bounded from below by $ \varepsilon _1$.}, we may find $ (k_0, \nu_0) \in(\mathbb Z\setminus \{0\}) \times \mathbb Z^d $ such that $ \vert k A\cdot a_0- \nu_0 \vert < \varepsilon _1$. 
     This last inequality implies $\nu_0 \in V_{ \varepsilon _1} =V_0$, while the previous one, together with $\vert A \cdot a_0- \nu_0 \vert < \frac{\varepsilon _1}{k}$ implies $\nu_0 \notin V_0$, a contradiction. 
    \end{proof} 
 
   \section*[3]{Appendix 3: Approximation of Generating functions and symplectic integrators}\label{Section-Appendix4}
   Our goal is to prove Lemma \ref{Lemma-appendix2}. It is a consequence to the more precise
   \begin{lem}\label{Lemma-appendix2-variant}
   Let $\varphi_H^t$ have $S_t(q,p)$ a s Generating function. We have
   $$ \left \Vert  S_t(q,p)-tH(q,p) \right \Vert \leq \frac{t^2}{2} \left \Vert \frac{\partial H}{\partial q} \right \Vert   \cdot \left\Vert \frac{\partial H}{\partial p} \right\Vert $$
   \end{lem} \begin{proof} 
   Note that $S_t$ has no fibre variable. It is well known that $S_t$ satisfies the following Hamilton-Jacobi equation
 \begin{displaymath}\left\{ \begin{aligned}
 \partial S_t(q,p)&=H\left ( q+ \frac{\partial S}{\partial p}(q,p),p \right ) \\
 S_0(q,p)&=0 
 \end{aligned}  \right .
 \end{displaymath}
  Set $S_t(q,p)=t\cdot H(q,p)+R_t(q,p)$ and replace in the equation, using the inequality $$\vert H(q+\xi, p)-H(q,p) \vert \leq \vert \xi \vert \cdot \left \Vert \frac{\partial H}{\partial q} \right \Vert _{C^0}$$,
  $$  \frac{\partial R_t}{\partial t}(q,p) \leq  \left\Vert \frac{\partial H}{\partial q} \right \Vert _{C^0} \left\vert \frac{\partial S_t}{\partial p} \right\vert \leq t \cdot \left\Vert \frac{\partial H}{\partial q}  \right\Vert _{C^0} \cdot \left\Vert \frac{\partial H}{\partial p}\right\Vert _{C^0} + \left\Vert \frac{\partial H}{\partial q}  \right\Vert _{C^0} \cdot  \left\vert  \frac{\partial R_t}{\partial p}(q,p)\right\vert
  $$  and since $R_0(q,p)=0$. 
  Now the relation $$ \partial _t R_t(q,p) \leq t A + B \left \vert \frac{\partial R_t}{\partial p} \right \vert$$ implies by monotonicity of the solutions of the Hamilton-Jacobi equations\footnote{That is $H\leq K$ implies that the solutions $v,w$ of $\partial_t u=H(x,D_xu)$ corresponding to the same initial condition satisfy $u\leq v$.}, that $R_t$ is bounded by the solution $u_t$ of  $\partial_t u = tA + B \vert \nabla _xu \vert $ that is $u(t,x)= \frac{t^2}{2}A$, so
  $$R_t(q,p) \leq   \frac{t^2}{2} \left\Vert \frac{\partial H}{\partial q}  \right\Vert _{C^0} \cdot \left\Vert \frac{\partial H}{\partial p}\right\Vert _{C^0}$$
  The same argument gives an estimate from below. 
   \end{proof}  
   
   \section*[4]{Appendix 4: Proof of Proposition \ref{Prop-9.3}}\label{Section-Appendix-3}
   The goal of this section is to prove 

\begin{namedprop*}{\bf Proposition \ref{Prop-9.3}}\label{Prop-17.4}
Given any $\alpha$, there exists a sequence $(\ell_\nu)_{\nu\geq 1}$ such that for almost all $\omega \in \Omega$
\begin{displaymath} 
 \lim_{\nu \to \infty}\lim_{U\subset {\mathbb R}^n}c(\mu_U, \varphi_{\ell_{\nu}}^\omega \alpha ) \leq \lim_{U\subset {\mathbb R}^n} c(\mu_U,\overline\varphi_\infty^\omega \alpha )
\end{displaymath} 
\end{namedprop*}

The proof is essentially the same as in section 6 of \cite{SHT}. We reproduce it here adapted to our situation and notations but notice that $\omega$ just appears as a parameter and so does not change the proof of Proposition \ref{Prop-9.3}. In particular the cycles we construct in the proof, do not need to depend continuously on $\omega$. 
We first need the next lemma. We define a closed cycle in $X$ to be a cycle for the Borel-Moore homology of $X$ : any closed submanifold represents a cycle in Borel-Moore homology, while in ordinary homology, this is the case only for compact submanifolds. 

\begin{lem}\label{Lemma-17.5} Let $S$ a G.F.Q.I. defined on $E$ and $c= \lim_{U\subset N} c(\mu_U,S)$. 
There exists a closed cycle $\Gamma$ such that $\Gamma_U=\Gamma\cap \pi^{-1}(U)$ satisfies 
$[\Gamma_U]=\mu_U$ in $H_*(S_U^{+\infty},S^{-\infty})$  and $S(\Gamma_U)\leq c(\mu_U,S)+ \varepsilon $ for $U$ belonging to a sequence of exhausting open sets with smooth boundary. 
\end{lem} 
\begin{proof} Consider an increasing  sequence $U_n$ of open sets with smooth boundary such that $N =\bigcup_nU_n$. Notice that there is a restriction map for $U\subset V$ sending $H_*(V,\partial V) \longrightarrow H_*(U,\partial U)$. It induces a map that we denote $\rho_{U,V}$ 

\begin{displaymath} 
H_*(S_V^t,S_V^{-\infty}\cup E_{\mid \partial V}) \longrightarrow H_*(S_U^t,S_U^{-\infty}\cup E_{\mid \partial U}) 
\end{displaymath} 
and a diagram

\centerline{\xymatrix{
H_*(S_V^{+\infty},S_V^{-\infty}\cup E_{\mid \partial V}) \ar[r]^{\rho_{U,V}}& H_*(S_U^{+\infty},S_U^{-\infty}\cup E_{\mid \partial U})\\
H_*(S_V^{c+ \varepsilon},S_V^{-\infty}\cup E_{\mid \partial V}) \ar[u]\ar[r]^{\rho_{U,V}}& H_*(S_U^{c+ \varepsilon} ,S_U^{-\infty}\cup E_{\mid \partial U})\ar[u]
}
}
Now the upper horizontal map sends $\mu_V$ to $\mu_U$, so applying this to the sequence $U_n$, we get a sequence $\widetilde \Gamma_n \in H_*(S_{U_n}^{c+ \varepsilon} ,S_{U_n}^{-\infty}\cup E_{\mid \partial {U_n}})$ with image $\mu_{U_n}$, and 
we have a sequence such that $\rho_{U_n, U_m}[\widetilde \Gamma_n]= [\widetilde \Gamma_{n}\cap \pi^{-1}(U_m)]$ is constant for $n\geq m$. Then we may glue the $\widetilde \Gamma_n$ as follows :
since $[\widetilde \Gamma_m]=[\widetilde \Gamma_{m+1}\cap \pi^{-1}(U_m)]$ in $H_*(S_{U_n}^{c+ \varepsilon} ,S_{U_n}^{-\infty}\cup E_{\mid \partial {U_n}})$, we have $D_m$ such that  $\partial D_m\cap \pi^{-1}(U_m)=\widetilde \Gamma_m -\widetilde \Gamma_{m+1}\cap \pi^{-1}(U_m)$ and we can assume $D_m\subset \pi^{-1}(\overline U_{m})$. Now we may consider the cycle 
\begin{displaymath} 
\Gamma_m=\widetilde \Gamma_m\cup (\partial D_m \cap \pi^{-1}(\overline U_{m}))\cup \widetilde \Gamma_{m+1}\cap \pi^{-1}(U_{m+1}\setminus U_m)
\end{displaymath} 

  \begin{figure}[H]
 \begin{center}  \subfloat{\begin{overpic}[width=7cm]
{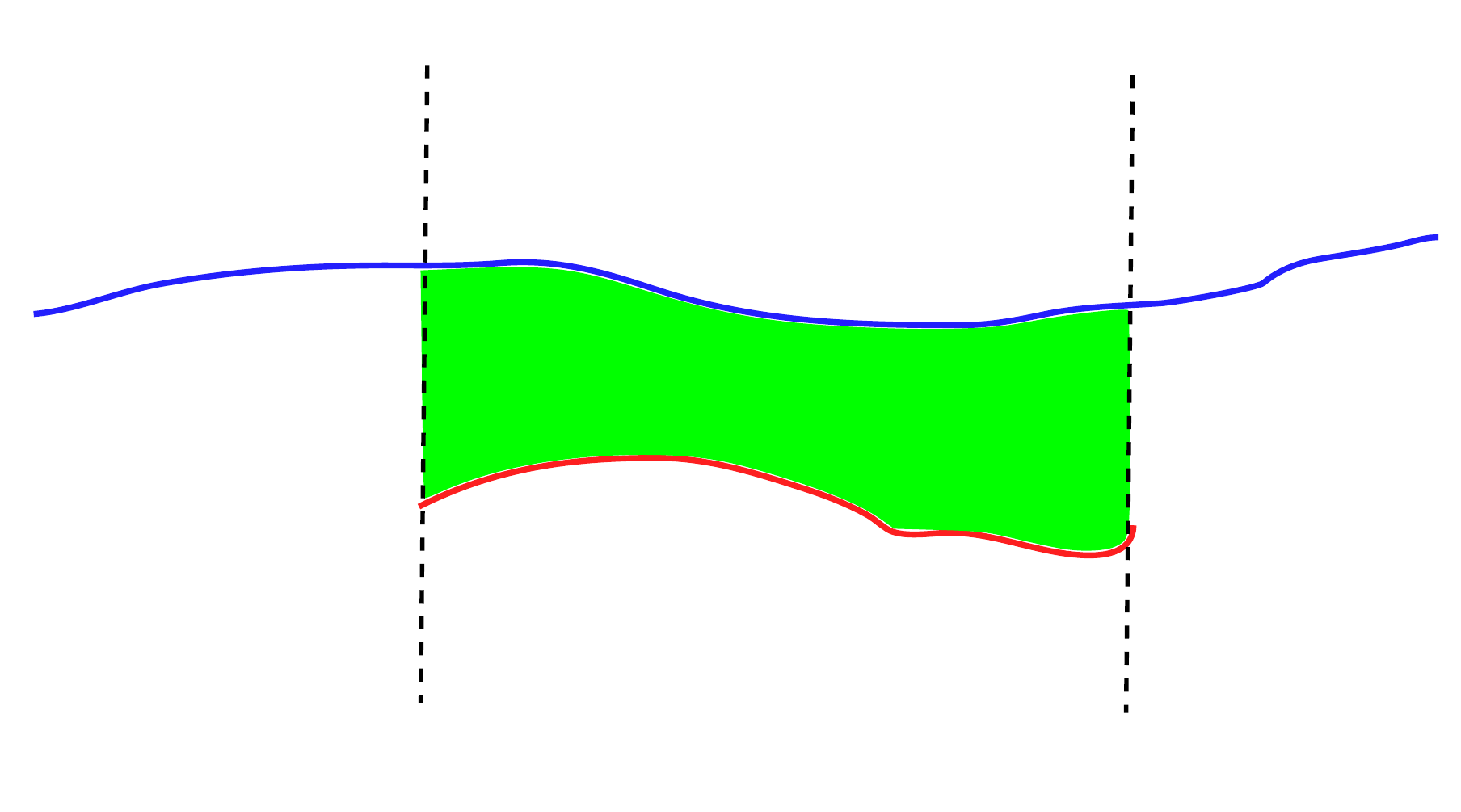}
 \put (5,40){$\widetilde\Gamma_{m+1}$}
  \put (55,10){$\widetilde\Gamma_{m}$}
    \put (55,25){$D_{m}$}
 \end{overpic} }
  \subfloat{\begin{overpic}[width=7cm]
{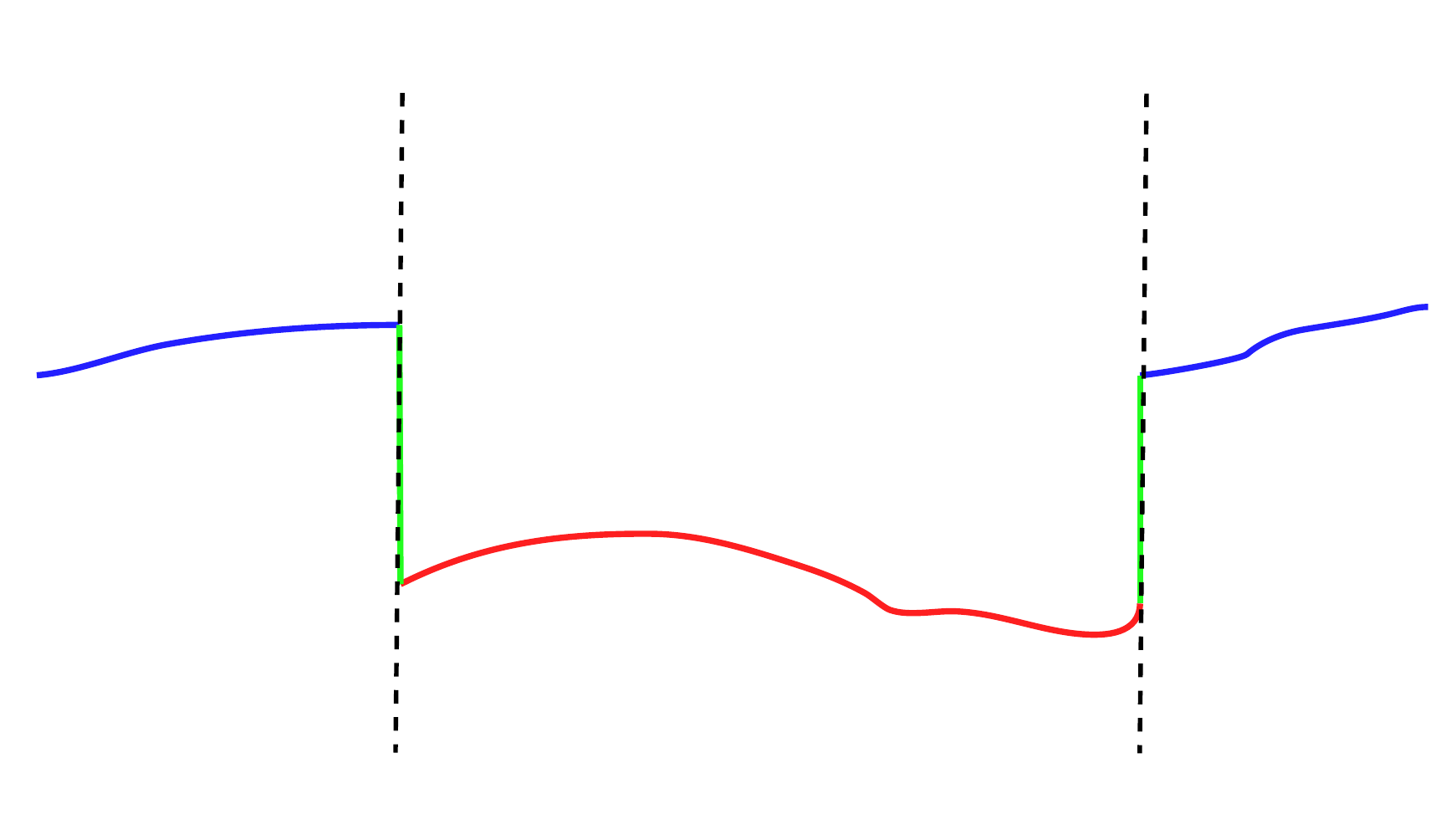}
 \put (85,40){$\Gamma_{m+1}$}
 \end{overpic} }
\end{center}
\caption{$\widetilde\Gamma_m$ in red, $\widetilde\Gamma_{m+1}$ in blue and $D_m$ in green on the left and $\Gamma_{m+1}$ on the right}
\end{figure}

We easily check that 
\begin{enumerate} 
\item $\Gamma_m \cap \pi^{-1}(U_m)=\widetilde\Gamma_m$
\item  $ \Gamma_m \cap \pi^{-1}(U_{m+1}\setminus U_m)=(\widetilde \Gamma_{m+1}\cap \pi^{-1}(U_{m+1})\cup (\partial D_m\cap \pi^{-1}(\partial U_m))$
\item $\partial \Gamma_m \subset E_{\partial U_{m+1}}$
\item $\Gamma_m \subset S^{c+ \varepsilon }$
\end{enumerate} 
By induction we can build a sequence $\Gamma_m$ and we have $\Gamma_n\cap \pi^{-1}(U_m)=\Gamma_m\cap \pi^{-1}(U_m)$ for $n>m$. Therefore $\bigcup_n  \Gamma_n$ is stationary over any compact set and defines a closed cycle $\Gamma$ such that $S(\Gamma) \leq c+ \varepsilon $. 
\end{proof} 

Now as in \cite{SHT}, section 5, the generating function for $\varphi_{k}^{\omega} $ is given by 
\begin{displaymath} 
 S_{k}^{\omega}(x,y;\xi,\zeta)= \frac{1}{k} \left [ S^\omega(kx,p_1,\xi_1)+\sum_{j=2}^{k-1}S^\omega(kq_j,p_j,\xi_j)+S^\omega(kq_k,y, \xi_k)\right ] +B_k(x,y,\zeta)
 \end{displaymath} 
 where $\xi=(\xi_1,...,\xi_k)$, $\zeta=(p_1,q_2,.....,q_{k-1},p_{k-1},q_k)$ 
 \begin{displaymath} \tau_a\zeta=(p_1,q_2+a,....,q_{k-1}+a, p_{k-1}, q_k+a)
 \end{displaymath} 
  and 
 \begin{displaymath} 
 B_k(x,y,\zeta)= \langle p_1, q_2-x\rangle + \sum_{j=2}^{k-1} \langle p_j, q_{j+1}-q_j\rangle +\langle y, x-q_k\rangle
 \end{displaymath} 

Now let $F(q,P;\eta)$ be a G.F.Q.I. for the graph of $\alpha$, then 
\begin{displaymath} 
G_{k}^{\omega} (u,v; x,y,\eta;\xi, \zeta)=S_{k}^{\omega} (u,y;\xi)+F(x,v;\eta)+ \langle y-v, u-x\rangle
\end{displaymath} 
is a G.F.Q.I. of $\varphi_k\alpha$. 
We set 
\begin{displaymath} 
\overline G_{k}^{\omega}(u,v;x,y,\eta)=h_{k}^{\omega} (y)+F(x,v;\eta)+\langle y-v,u-x\rangle 
\end{displaymath} 
We shall omit the subscripts $a,\chi$ for the moment, so in the sequel, $\overline G_{k,a,\chi}^{\omega}$ means $\overline G_{k,a,\chi}^{\omega}$
Here the variables $u,v,x,y$ are in $ {\mathbb R}^n$ and we denote by $E_k$ the space of the  $\theta=(\zeta, \xi)$  where $\xi \in E^k, \zeta \in ( {\mathbb R}^{2n})^k$ and  $\eta \in V$.
By definition we have a cycle $\Gamma_U^\omega$ in $U\times {\mathbb R}^n_v\times {\mathbb R}^n_x\times {\mathbb R}^n_y\times E\times V$ relative to 
$(\overline G_{k}^{\omega})^{-\infty} \cup \partial U \times {\mathbb R}^n_v\times {\mathbb R}^n_x\times {\mathbb R}^n_y\times E\times V$ and homologous (as a closed cycle) to $U\times  {\mathbb R}^n_v\times \Delta_{x,y}\times E^- \times V^-$
(where $\Delta$ is the diagonal)
such that 
\begin{displaymath} 
\overline G_{k}^{\omega}  (\Gamma_U^\omega) \leq c(\mu_U, \overline G_{k}^{\omega} )+ \varepsilon = c(\mu_U, \overline \varphi_{k,U}^\omega\alpha) + \varepsilon 
\end{displaymath} 
where $\overline \varphi_{k,U}^\omega$ is the flow of $h_{k,U}^{\omega}(y)$. 

Moreover according to Lemma \ref{Lemma-17.5}, we can assume there is a closed (i.e. Borel-Moore)  cycle $\Gamma^\omega$ such that 
$\Gamma_U^\omega= \Gamma^\omega \cap \pi^{-1}(U)$ (at least for a cofinal sequence of $U$'s). 

Now let $C^\omega_U(y)$ be a cycle in the class of $U \times E_k^-$ in $H_*((S_{k,y}^\omega)^{+\infty}, (S_{k,y}^\omega)^{-\infty})$, depending continuously on $y$, such that\footnote{We still have a problematic notation with respect to the order of our variables. By $S_k(y, C_U(y))$ we mean the maximum of $S_k(x,y;\xi,\zeta)$ where $(x;\xi,\zeta) \in C(y)$}
\begin{displaymath} 
S_{k}^{\omega} (y, C^\omega_U(y)) \leq h_{k,U}^\omega(y) +  a \chi(y)+ \varepsilon 
\end{displaymath} 

Thus we set for $a \in {\mathbb R}_+, \chi \in C^\infty( {\mathbb R}^n)$
\begin{displaymath} 
\overline G_{k,a,\chi}^{\omega}(u,v;x,y,\eta)=h_{k}^{\omega} (y)+F(x,v;\eta)+\langle y-v,u-x\rangle + a \chi(y)
\end{displaymath} 
We shall omit the subscripts $a,\chi$ for the moment, so in the sequel, $\overline G_{k,a,\chi}^{\omega}$ means $\overline G_{k,a,\chi}^{\omega}$

We again invoke Lemma \ref{Lemma-17.5} in order to obtain a (closed) cycle $C^\omega(y)$ such that for a cofinal sequence of $U$'s we have $C^\omega_U(y)=C^\omega(y)\cap \pi^{-1}(U)$. 
As in \cite{SHT}, section 6, this is possible provided $\chi$ is the characteristic function of $\Lambda_\delta$, the complement of a disjoint union of sets of diameter less than $\delta$. For example, we can take $\Lambda_\delta$ to be a neighborhood of $\Lambda (\delta) =\{(x_1,..., x_n) \mid \exists j, x_j \in \delta {\mathbb Z}\}$. 
We now construct a new (Borel-Moore) cycle, symbolically denoted $\Gamma \times_YC$ and defined as follows (everything depends on $\omega$ but for notational convenience we omit it)
\begin{displaymath} 
\Gamma\times_Y C=\left\{(u,v;x,y,\theta,\eta) \mid (u,v,x,y,\eta)\in \Gamma, (u,\theta)\in C(y)\right\} \end{displaymath} 
we have 
\begin{enumerate} 
\item \label{item-6.2} $(\Gamma\times_YC)_U$ is homologous to $U\times {\mathbb R}^n_v\times \Delta_{x,y}\times E_k^-\times V^-$
\item \label{item-6.1} $G_k^\omega((\Gamma\times C)_{U})\leq \overline G_{k,a,\chi}^\omega(\Gamma_U) + \varepsilon $

\end{enumerate} 
Indeed for (\ref{item-6.2}), it follows from the fact that the homology class of $A\times_YB$ only depends on the homology class of $A, B$ and so $\Gamma_U\times_YC_U^-$ is homologous to 
\begin{gather*} 
(U\times  {\mathbb R}^n_v\times \Delta_{x,y}\times V^-)\times _Y (U\times E_k^-)=\left\{ (u,v;x,y,\eta,\theta) \mid u\in U, x=y, \eta \in V^-, \theta \in E_k^- \right \} = \\
U \times {\mathbb R}^n_v \times \Delta_{x,y}\times V^-\times E_k^-
\end{gather*} 
As for (\ref{item-6.1}), we have  $(\Gamma\times_Y C^-)_U=\Gamma_U\times_YC^-_{U}$ 
and 
\begin{displaymath} 
G_{k}^{\omega} (\Gamma_U\times_YC^-_{U})\overset{def}= \sup\left\{S_{k}^{\omega} (u,y;\theta)+F(x,v;\eta) + \left\langle y-v, u-x\right \rangle \mid (u,v;x,y,\eta)\in \Gamma, \; (u;\theta)\in C(y) \right \}
\end{displaymath} 
but since $S_k(u,y;\theta)\leq h_{k,U}^+(y)+a\chi(y)+ \varepsilon $ for $(u;\theta)\in C(y)$ we have 
\begin{gather*} 
G_k^\omega(\Gamma_U\times _Y C_U^-) \leq \\ \sup\left \{ F(x,v;\eta) + h_{k,U}^+(y)+a \chi(y)+ \varepsilon + \left \langle y-v,u-x\right \rangle \mid (u,v;x,y,\eta) \in \Gamma^\omega_U, (u,\theta)\in C^\omega_U(y) \right \} \\ 
\leq \overline G_{k,a,\chi}^{\omega}(\Gamma_U) + \varepsilon
\end{gather*}  

Now as in \cite{SHT} section 6 p.27 let us consider a collection of $\ell$ open sets $\Lambda_\delta^j$ for $1\leq j \leq \ell$ such that each of them is a translate of $\Lambda_\delta$ and any $n+1$ of them have empty intersection. We denote by $\chi_j$ ($1\leq j \leq \ell$) the corresponding functions.  
We set  $\overline x=(x_1,...,x_\ell), \overline y=(y_1,...,y_\ell), \overline \theta=(\theta_1,..., \theta_\ell)$
and define\footnote{Here we omit $\omega$ from the notation, which would otherwise become unwieldy.}
\begin{displaymath} 
G_{k,\ell}(u,v,\overline x, \overline y, \overline \theta, \eta) =F(x_1,v;\eta)+ \frac{1}{\ell} \sum_{j=1}^\ell S_k(\ell x_j,y_j,\theta_j)+B_\ell(\overline x, \overline y) + \left\langle y_\ell-v,u-x_1\right \rangle
\end{displaymath} 

This is a G.F.Q.I. for $\rho_\ell^{-1}\varphi_k^\ell \rho_{\ell}^{-1}\alpha = \rho_\ell^{-1} \rho_k^{-1}\varphi^{k\ell}\rho_k\rho_\ell \alpha= \varphi_{k\ell}\alpha$. 

Let 
\begin{displaymath} \overline G_{k,\ell}(u,v;\overline x, \overline y, \eta) =F(x_1,v,\eta) + \frac{1}{\ell} \sum_{j=1}^\ell (h_{k}^+(y_j)+a \chi_j(y_j))+B_\ell (\overline x, \overline y) + \left\langle y_l-v, u-x_1\right\rangle
\end{displaymath}
By definition there is a $\Gamma_{k,\ell}$ such that 
\begin{displaymath} 
\overline G_{k,\ell}((\Gamma_{k,\ell})_U) \leq c(\mu_U, \overline G_{k,\ell})+ \varepsilon 
\end{displaymath}  
and using $C_j(y_j)$ as before for $1\leq j \leq \ell$ and setting 
\begin{displaymath} 
\Gamma_{k,\ell}\times_Y C^-[\ell]=\left \{(u, v;\overline x,\overline y,\overline\theta,\eta) \mid (u,v,x,y,\eta) \in \overline\Gamma , (\ell x_j,\xi_j) \in C^-_j(y_j)\right \}
\end{displaymath} 
we have  
\begin{displaymath} 
c(\mu_U,G_{k,\ell})\leq G_{k,\ell}((\Gamma_{k,\ell})_U\times_Y(C^-)_U[\ell])\leq \overline G_{k,\ell}((\Gamma_{k,\ell})_U)\leq c(\mu_U, \overline G_{k,\ell})+ 2 \varepsilon 
\end{displaymath} 
Finally we claim that 
\begin{displaymath} 
c(\mu_U, \overline G_{k,\ell}) \leq c(\mu_U, \overline G_{k}) + \frac{(n+1)a}{\ell}
\end{displaymath} 
Indeed, $\overline G_{k,\ell}$ is the generating function of $\psi_{k,\ell}=\rho_{\ell}^{-1} \psi_{k}^1\circ ...\circ \psi_k^\ell\rho_\ell$ where $\psi_k^j$ is the time one flow of $h_k(y)+a \chi_j(y)$. But these flows commute, so $\psi_{k,\ell}$ is the time one flow of
\begin{displaymath} 
K_{k,\ell}(y)=\frac{1}{\ell} \sum_{j=1}^\ell (h_{k,U}(y)+a\chi_j(y))
\end{displaymath} 
and we have $ \vert K_{k,\ell}(y) - h_k(y) \vert \leq \frac{(1+n)a}{\ell}+ \varepsilon $ therefore
\begin{displaymath} 
c(\mu_U, \overline G_{k,\ell}) \leq c(\mu_U, \psi_{k,\ell}\alpha) \leq c(\mu_U, \psi_k^1\alpha) + \frac{(1+n)a}{\ell}+ \varepsilon
\leq c(\mu_U, \overline\varphi_k\alpha) + \frac{(1+n)a}{\ell}+ \varepsilon
\end{displaymath} 
Thus for $\ell $ large enough
$c(\mu_U,\overline G_{k,\ell}) \leq c(\mu_U, \psi_k^1\alpha)+ 2 \varepsilon $. 
Taking the limit as $k$ goes to infinity, we get 

\begin{displaymath} c(\mu_U,\varphi_{k\ell}\alpha) =c(\mu_U, G_{k,\ell}) \leq c(\mu_U,\overline \varphi_k \alpha) + 2 \varepsilon \leq c(\mu_U, \overline \varphi \alpha) + 3 \varepsilon 
\end{displaymath} 
This concludes the proof of Prop \ref{Prop-9.3}.


\begin{thebibliography}{}
\bibitem[A-T-Y]{ATY}
 S. N. Armstrong, H. V. Tran, and Yifeng Yu, 
 \newblock{\em Stochastic homogenization of a nonconvex
Hamilton-Jacobi equation.}
 \newblock{ Calc. Var. Partial Differential Equations}, vol. 54, 1507-1524
(2015).
\newblock{\url{https://doi.org/10.1016/j.jde.2016.05.010}}

\bibitem[B-C]{B-C}
M. Bardi, I. Capuzzo-Dolcetta,
\newblock{\em  Optimal Control and Viscosity Solutions of Hamilton-Jacobi-Bellman Equations.}
\newblock{\em Modern Birkh\" auser Classics,}
\newblock{ Birkh\"auser Basel,    Springer Science+Business Media New York, 1997.}
\newblock{\url{https://doi.org/10.1007/978-0-8176-4755-1}}

\bibitem[Ba]{Ba}
G. Barles,
\newblock{\em Solutions de viscosit\'e des \'equations de Hamilton-Jacobi.}
\newblock{Springer-Verlag Berlin Heidelberg, 1994.}

\bibitem[Bou]{Bourbaki}
N. Bourbaki,
\newblock{Topologie G\'en\'erale. Chap. 1 \`a 4.}
\newblock{Springer-Verlag,1995}


\bibitem[C-L]{C-L}
M.Crandall and P.L. Lions,
\newblock{\em Viscosity solutions of Hamilton-Jacobi equations.}
\newblock{Transactions of the Amer. Math. Soc.} vol 277(1),1983. 
\newblock{\url{https://doi.org/10.1090/S0002-9947-1983-0690039-8}}

\bibitem[C-M]{C-M}
P.~Chernoff and J.~Marsden,
\newblock{\em On continuity and smoothness of group actions.}
\newblock{Bull. Amer. Math. Soc.}, Volume 76, Number 5 (1970), 1044-1049.
\newblock{\url{https://projecteuclid.org/euclid.bams/1183532211}}

\bibitem[C-V]{Cardin-Viterbo}
F. Cardin et C. Viterbo,
\newblock {\em Commuting Hamiltonians and multi-time Hamilton-Jacobi equations. }
\newblock {Duke Math Journal,}  vol. 144, (2008), pp. 235-284.
\newblock{\url{https://projecteuclid.org/euclid.dmj/1218716299}}

\bibitem[Ch]{Ch}
Marc Chaperon. 
\newblock{\em Lois de conservation et g\'eom\'etrie symplectique. }
\newblock{C. R. Acad. Sci. Paris S\'er. I Math.,} 312(4):345--348, 1991.

\bibitem[C-I-S]{Contreras-Iturriaga-Siconolfi}
G. Contreras, R. Iturriaga, A. Siconolfi
\newblock{\em Homogenization on arbitrary manifolds.}
\newblock{\url{https://arxiv.org/pdf/1211.1081.pdf}}

\bibitem[E]{Ellis}
R. Ellis, 
\newblock{\em  Locally compact transformation groups. }
\newblock{Duke Mathematical Journal,}  24(2)(1957), pp. 119-125. 
\newblock{\url{https://doi.org/10.1215/s0012-7094-57-02417-1}}

\bibitem[F-S]{F-S}
W. Feldman and P.E. Souganidis,
\newblock{\em Homogenization and non-homogenization of certain non-convex Hamilton-Jacobi equations.}
\newblock {Journal de Math\'ematiques Pures et Appliqu\'ees.}
Volume 108, Issue 5, November 2017, Pages 751-782.
\newblock{\url{https://doi.org/10.1016/j.matpur.2017.05.016}}

\bibitem[Gr-Sch]{Greschonig-Schmidt}
G. Greschonig and K.Schmidt
\newblock{\em Ergodic decomposition of quasi-invariant probability measures.}
\newblock{Colloquium Mathematicae.} vol. 84-85, pp. 495-514 (2000).
\newblock{\url{http://eudml.org/doc/210829}}

\bibitem[Hat]{Hatcher}
A. Hatcher,
\newblock{\em Algebraic Topology.}
\newblock{Cambridge University press}, 2002.
\newblock{\url{http://pi.math.cornell.edu/~hatcher/AT/ATpage.html}}

\bibitem[H-M]{HM}
K.H. Hofmann and S.A. Morris
\newblock{\em The structure of  compact groups. A primer for the student. A handbook for the expert. Third  edition.}
\newblock{De Gruyter, 2013.}

\bibitem[Hum]{Hum}
V. Humili\`ere, 
\newblock{\em On some completions of the space of Hamiltonian maps. }
\newblock{ Bull. Soc. math. France},
136 (3), 2008, p. 373-404.
\newblock{\url{https://doi.org/10.24033/bsmf.2560}}
and \newblock{\url{http://www.numdam.org/item/BSMF_2008__136_3_373_0/}}

\bibitem[L-R]{L-R}
C.N. Lee and F. Raymond.
\newblock{\em  \v Cech extensions of contravariant functors.}
\newblock {Trans. Amer. Math. Soc.} 133 (1968), pp. 415-434. 
\newblock{\url{https://doi.org/10.1090/S0002-9947-1968-0234450-X }}

\bibitem[L-P-V]{L-P-V}
P. L. Lions, G. C. Papanicolaou and S. R. S. Varadhan,
\newblock  {\em Homogenization of Hamilton-Jacobi Equations.}
\newblock  Preprint, 1988.
\newblock{Available from \url{http://localwww.math.unipd.it/~bardi/didattica/Nonlinear_PDE_%20homogenization_Dott_%202011/LPV87.pdf}}

\bibitem[P-V]{P-V}
G.C.Papanicolaou and, S.R.S.Varadhan.
\newblock{ \em Boundary value problems with rapidly oscillating random coefficients.}
\newblock{ In J. Fritz, J. L. Lebowitz \& D. Szasz, editors, Random Fields Volume II: Rigorous Results in Statistical Mechanics and Quantum Field Theory}, pp. 835-873.
\newblock{Amsterdam, 1981. North-Holland Publishing Company.}

\bibitem[P-R]{P-R}
A. Pelayo and F. Rezakhanlou,
\newblock{\em 
The Poincar\'e--Birkhoff Theorem in Random Dynamics. }
\newblock{ Trans. Amer. Math. Soc.} Volume 370, 601-639 (2018).
\newblock{\url{https://doi.org/10.1090/tran/6967 }}

\bibitem[R-T]{R-T}
F. Rezakhanlou, and J. E. Tarver. 
\newblock{\em Homogenization for Stochastic Hamilton-Jacobi Equations.}
\newblock{Archive for Rational Mechanics and Analysis}, vol.151, no. 4 (April 1, 2000): 277--309. 
\newblock{\url{https://doi.org/10.1007/s002050050198}}

\bibitem[R]{R}
V. Roos,
\newblock{Solutions variationnelles et solutions de viscosit\'e.}
\newblock{PhD thesis, Universit\'e de Paris-Dauphine, 2017.}
\newblock{\url{http://www.theses.fr/2017PSLED023}} and
\newblock{\url{https://basepub.dauphine.fr/handle/123456789/16992}} 

\bibitem[R2]{Roos2}
V. Roos,
\newblock{\em  Communications in Contemporary Mathematics}, Vol. 21(2019), 1850018.
\newblock{\url{https://doi.org/10.1142/S0219199718500189}}

\bibitem[Ru]{Rudin}
W. Rudin,
\newblock{\em Principles of Mathematical Analysis.}
\newblock{McGraw-Hill}, 1976. 

\bibitem[Sey]{Seyfaddini}
S. Seyfaddini. 
\newblock{\em Descent and $C^0$-rigidity of spectral invariants on monotone symplectic manifolds.}
\newblock{J. Topol. Anal.}, vol. 4(2012), pp.481-498.
\newblock{\url{https://doi.org/10.1142/S1793525312500215}}

\bibitem[Sik1]{Sikorav-GF}
J.-C. Sikorav, 
\newblock{Probl\`emes d'intersection et de points fixes en G\'eom\'etrie Hamiltonienne.}
\newblock{\em Commentarii Math. Helv.,} vol.62(1987), 62-73.
\newblock{\url{https://doi.org/10.1007/BF02564438}}

\bibitem[Sik2]{Sik}
J.C. Sikorav,
\newblock{Private communication, 1989.}

\bibitem[Sor]{Sorrentino}
A. Sorrentino,
\newblock{\em On the homogenization of the Hamilton-Jacobi equation.}
\newblock{ \url{https://arxiv.org/pdf/1904.01359.pdf}}

\bibitem[Soug]{Soug}
P.E. Souganidis,
\newblock{\em Stochastic homogenization of Hamilton-Jacobi equations and some applications.}
\newblock{Asymptotic Analysis} vol.20 (1999), pp. 1--11.


\bibitem [Th\'e]{Theret}
D. Th\'eret,
\newblock{ \em A complete proof of Viterbo's uniqueness theorem on generating functions.}
\newblock{Topology and its Applications} vol. 96, Issue 3(1999), Pages 249-266
\newblock{\url{https://doi.org/10.1016/S0166-8641(98)00049-2}}

\bibitem[Tr]{Traynor}
L. Traynor,
\newblock{\em Symplectic homology via Generating Functions.}
\newblock{Geometric and Functional Analysis}, vol.4(1994), pp. 718-748.
\newblock{\url{https://doi.org/10.1007/BF01896659}}

  \bibitem[V1]{Viterbo-STAGGF} C.~Viterbo, 
  \newblock { Symplectic  topology as the geometry of generating
functions.} \newblock  {\em Mathematische Annalen}, vol. 292, (1992), pp.
685--710.
\newblock{\url{https://doi.org/10.1007/BF01444643}}

\bibitem[V2]{Viterbo-FCFH2}
C. Viterbo,
\newblock{\em Functors and Computations in Floer cohomology, II.}
\newblock{\url{https://arxiv.org/abs/1805.01316}}

\bibitem[V3]{Viterbo-Ott}
C. Viterbo,
\newblock{\em Solutions d'\'equations d'Hamilton-Jacobi et g\'eom\'etrie symplectique},
\newblock{S\'eminaire \'Equations aux d\'eriv\'ees partielles (Polytechnique)},
\newblock{Ecole Polytechnique, Centre de Math\'ematiques, 1995-1996.}
\newblock { \url{http://www.numdam.org/item/SEDP_1995-1996____A22_0}}

\bibitem[V4]{Viterbo-Montreal}
C. Viterbo,
\newblock{\em Symplectic topology and Hamilton-Jacobi equations.}
\newblock{ In book ``Morse Theoretic Methods in Nonlinear Analysis and in Symplectic Topology",}(2006), pp.439-459.
\newblock{Springer Netherlands}

\bibitem[V5]{SHT}
C.~Viterbo,
\newblock{\em Symplectic homogenization.}
\newblock{\url{https://arxiv.org/abs/0801.0206}}  


\bibitem[WQ1]{WQ1}
Q. Wei
\newblock{\em Solutions de viscosit\'e des \'equations de Hamilton-Jacobi et minmax it\'er\'es.}
\newblock{PhD thesis, Universit\'e de Paris 7} (2013).
\newblock{\url{https://tel.archives-ouvertes.fr/tel-00963780/}}

\bibitem[WQ2]{WQ2}
Q. Wei
\newblock{\em Viscosity solution of the Hamilton-Jacobi equation by a limiting minimax method.}
\newblock{Nonlinearity,} vol. 27(1)(2013), pp. 17--41. 
\newblock{\url{https://doi.org/10.1088/0951-7715/27/1/17}}

\bibitem[We1]{Weil-uniform}
A. Weil,
\newblock{\em Sur les espaces \`a structure uniforme et sur la topologie g\'en\'erale.}
\newblock{Act. Sci. Ind.}, vol. 551, Hermann, Paris, 1938. 


\bibitem[We2]{Weil2}
A. Weil
\newblock{\em L'int\'egration dans les groupes topologiques. 2\`eme \'edition.}
\newblock{Hermann, 1965.}

\bibitem[Wi]{Wi}
N. Wiener, 
\newblock{\em The ergodic theorem.}
\newblock{Duke Math. J.} vol. 5 (1939), pp. 1-18.
\newblock{\url{https://doi.org/10.1215/S0012-7094-39-00501-6}}

\bibitem[Z1]{Z2}
T. Zhukovskaya,
\newblock{\em Singularit\'es de minimax et solutions faibles d'\'Equations aux  d\'eriv\'ees partielles.}
\newblock{Ph.D. dissertation, Universit\'e Paris Diderot-Paris 7, Paris, 1993.}

\bibitem[Z2]{Z1}
T. Zhukovskaya,
\newblock{\em  Metamorphoses of the Chaperon-Sikorav weak solutions of
Hamilton-Jacobi equations.} 
\newblock{J. Math. Sci.} 82 (1996), 3737--3746. 
\newblock{\url{https://doi.org/10.1007/BF02362583}}


  \bibitem[Zi]{Zilliotto}
B. Zilliotto, .
 \newblock{\em Stochastic Homogenization of Nonconvex Hamilton-Jacobi Equations: A Counterexample.}
\newblock{Comm. Pure Appl. Math.}, 70: 1798-1809. 
\newblock{\url{https://doi.org/10.1002/cpa.21674}}
\end{thebibliography}
\end{document}